\documentclass{daj}

\dajAUTHORdetails{%
  title = {Notes on compact nilspaces}, 
  author = {Pablo Candela},
  plaintextauthor = {Pablo Candela},
  runningtitle = {Notes on compact nilspaces}, 
  runningauthor = {P. Candela},
}   

\dajEDITORdetails{%
   year={2017},
   number={16},
   received={31 May 2016},   
   revised={25 May 2017},    
   published={7 September 2017},  
   doi={10.19086/da.2106},       
}   

\usepackage{amsmath, amsfonts, amssymb, amsthm, mathtools}
\usepackage{enumerate}
\usepackage{enumitem}
\usepackage{fancyhdr}
\usepackage[utf8]{inputenc}
\usepackage{tikz,tikz-cd}
\usepackage[hang,flushmargin]{footmisc}
\usepackage[compact]{titlesec}
\usepackage{graphicx}
\usepackage{xr}

\titleformat{\chapter}[display]
{\normalfont\huge\bfseries}{\chaptertitlename\ \thechapter}{20pt}{\Huge}

\newtheorem{theorem}{Theorem}[section]
\newtheorem{lemma}[theorem]{Lemma}

\newtheorem{corollary}[theorem]{Corollary}
\newtheorem{proposition}[theorem]{Proposition}

\theoremstyle{definition}
\newtheorem{defn}[theorem]{Definition}
\newtheorem{remark}[theorem]{Remark}
\newtheorem{example}[theorem]{Example}
\def\E{\mathbb{E}}
\def\Z{\mathbb{Z}}
\def\R{\mathbb{R}}
\def\T{\mathbb{T}}
\def\C{\mathbb{C}}
\def\N{\mathbb{N}}

\newcommand{\ud}{\,\mathrm{d}}
\newcommand{\id}{\mathrm{id}}

\DeclareMathOperator{\stab}{Stab}
\DeclareMathOperator{\aut}{Aut}
\DeclareMathOperator{\ab}{Z}
\DeclareMathOperator{\bnd}{B}
\DeclareMathOperator{\tran}{\Theta}

\DeclareMathOperator{\supp}{Supp}

\DeclareMathOperator{\codim}{codim}
\DeclareMathOperator{\comp}{K}
\DeclareMathOperator{\q}{c}
\DeclareMathOperator{\ns}{X}
\DeclareMathOperator{\nss}{Y}
\DeclareMathOperator{\co}{\circ\hspace{-0.02 cm}}
\DeclareMathOperator{\cu}{C}
\DeclareMathOperator{\cor}{Cor}
\DeclareMathOperator{\cs}{s}

\newcommand{\Zmod}[1]{\Z/#1\Z} 
\newcommand*{\sbr}[1]{\scalebox{0.8}{$(#1)$}}
\newcommand{\rk}{\mathrm{rk}}

\providecommand{\arr}[1]{\langle#1\rangle}

\DeclareMathOperator{\cA}{\mathcal{A}}

\DeclareMathOperator{\cC}{\mathcal{C}}
\DeclareMathOperator{\cD}{\mathcal{D}}
\DeclareMathOperator{\cE}{\mathcal{E}}
\DeclareMathOperator{\cF}{\mathcal{F}}

\DeclareMathOperator{\cK}{\mathcal{K}}
\DeclareMathOperator{\cL}{\mathcal{L}}

\DeclareMathOperator{\cQ}{\mathcal{Q}}

\DeclareMathOperator{\cT}{\mathcal{T}}
\DeclareMathOperator{\cU}{\mathcal{U}}

\DeclareMathOperator{\trem}{\Psi}
\DeclareMathOperator{\tRe}{Re}

\begin{document}

\begin{frontmatter}[classification=]

\title{\Huge{Notes on compact nilspaces} } 
\author[pc]{Pablo Candela}
\end{frontmatter}


\tableofcontents

\titlespacing*{\chapter}{0pt}{-0.9cm}{40pt}

\chapter{Introduction}\label{chap:intro}
These notes form the second part of a detailed account of the theory of nilspaces developed by Camarena and Szegedy in the paper \cite{CamSzeg}. The material in these notes expands on the third chapter of their paper, providing more detailed proofs of the main results. To that end we also include several additional  results that are implicit in \cite{CamSzeg}. This material relies strongly on the first part of our exposition, given in \cite{Cand:Notes1}. We shall have to assume some familiarity with the basic theory of nilspaces and some of their algebraic properties, but we shall always refer to the relevant results in \cite{Cand:Notes1}.

\smallskip
Recall the notion of a nilspace from \cite[Definition 1.2.1]{Cand:Notes1}. The principal objects treated in these notes are \emph{compact nilspaces}. These are nilspaces equipped with a compact topology that is compatible with the cube structure. To give the precise definition, we adopt the following terminology.

\begin{defn}\label{def:compspace}
Throughout the sequel, to be concise we shall use the term \emph{compact space} to mean a compact, Hausdorff, second-countable topological space.
\end{defn}
\noindent In particular, compact spaces in these notes are Polish spaces \cite[Theorem 5.3]{Ke}. More generally, topological spaces will usually be assumed to be Hausdorff and second-countable.

Recall from \cite[Definition 1.2.3]{Cand:Notes1} the notion of a cubespace.
\begin{defn}\label{def:compnils}
A cubespace $\ns$ is called a \emph{compact cubespace} if $\ns$ is a compact space and $\cu^n(\ns)$ is a closed subset of $\ns^{\{0,1\}^n}$ for every $n\in \N$. A \emph{compact nilspace} is a compact cubespace satisfying the ergodicity and corner-completion axioms from \cite[Definition 1.2.1]{Cand:Notes1}.
\end{defn}

\noindent The main goal in the sequel is to characterize compact nilspaces.

A central result in this direction is a description of a general compact nilspace as an inverse limit of simpler spaces, namely compact nilspaces of finite rank (these are defined in Section \ref{sec:CFRdef}). This result is treated in Section \ref{sec:invlim} (see Theorem \ref{thm:invlim}). 

Another central result concerns finite-rank compact nilspaces, and involves \emph{nilmanifolds}, i.e. compact homogeneous spaces of nilpotent Lie groups, the study of which goes back to \cite{Malcev}. To situate  the result, let us note that nilmanifolds provide important examples of compact nilspaces. Indeed, a nilmanifold becomes a compact nilspace when, given a filtration on the Lie group (with certain topological properties), the nilmanifold is equipped with the natural cubes associated with the filtration, namely the cubes  introduced by Host and Kra \cite{HK,HKparas}. This construction is treated from a purely algebraic viewpoint in \cite[Section 2.3]{Cand:Notes1}, and its basic topological aspects are detailed in Section \ref{sec:filnilmcompns} below. The central result in question here goes in the converse direction, and states that if a compact nilspace of finite rank has connected structure groups,\footnote{For the notion of the structure groups of a nilspace, recall \cite[Definition 3.2.17]{Cand:Notes1} and \cite[Theorem 3.2.19]{Cand:Notes1}.} then it is isomorphic to a nilmanifold with a cube structure of the kind mentioned above. This is treated in Section \ref{sec:CFRnilsnilm} (see Theorem \ref{thm:toralnilspace}).

The work toward these main theorems yields several other results of inherent interest, including the following: endowing every compact nilspace with a Borel probability measure that generalizes the Haar measure on compact abelian groups (Proposition \ref{prop:nilspaceHaar}); an automatic continuity result for Borel morphisms between compact nilspaces (Theorem \ref{thm:Klepgen}); a rigidity result for morphisms into compact nilspaces of finite rank (Theorem \ref{thm:rigidity}). Several of the main tools used to obtain these results rely on the theory of \emph{continuous systems of measures}, which provides a natural framework for measure theoretic aspects of  compact nilspaces. This is detailed in Section \ref{sec:measprel}.

An alternative treatment of compact nilspaces is given by Gutman, Manners, and Varj\'u in the series of papers \cite{GMV1,GMV2,GMV3}.

Let us end this introduction by evoking one of the central motivations for the study of compact nilspaces. This motivation concerns the analysis of \emph{uniformity norms}. These norms were introduced in arithmetic combinatorics by Gowers \cite{GSz}, and independently in ergodic theory (as the analogous \emph{uniformity seminorms}) by Host and Kra \cite{HK}. The uniformity norms, or $U^d$ norms (one such norm for each integer $d\geq 2$), are defined on the space of complex-valued functions on a finite (or more generally a compact) abelian group. These norms provide useful tools to control averages of bounded functions over certain linear configurations in abelian groups, configurations such as arithmetic progressions (such control can be obtained via the so-called \emph{generalized Von Neumann theorems}); see \cite{GTarith} and \cite{GTlin}. A major topic concerning these norms is the analysis, for each $d\geq 2$, of the harmonics that determine whether the $U^d$ norm of a function is small or large. For the $U^2$ norm these characteristic harmonics are simply the characters from Fourier analysis. For $d>2$, this topic leads to the theory of \emph{higher order Fourier analysis}. Central to this theory are the results known as \emph{inverse theorems} for the $U^d$ norms, proved by Green, Tao and Ziegler \cite{GTZ}, and independently by Szegedy \cite{Szegedy:HFA}. Essentially, these theorems tell us that, for the $U^d$ norm of functions on $\Zmod{N}$, one can use as characteristic harmonics the functions known as \emph{$(d-1)$-step nilsequences}. In the approach to these theorems developed by Szegedy (consisting principally in \cite{Szegedy:HFA} and his work with Camarena \cite{CamSzeg}), compact nilspaces play a fundamental role. Indeed, roughly speaking, in this approach a nilsequence on $\Zmod{N}$ is obtained as a composition $F\co\varphi$, where $\varphi$ is a nilspace morphism from $\Zmod{N}$ to a compact nilspace $\ns$, and $F$ is some continuous function $\ns\to\C$. In this way, compact nilspaces play a role in higher order Fourier analysis that generalizes the role played by the circle group $\R/\Z$ in the definition of characters in Fourier analysis. In these notes we shall concentrate on the theory of compact nilspaces in itself. For further introduction to higher order Fourier analysis and its applications, we refer the reader to the survey \cite{GHFA}. 

\vspace{3cm}

\noindent \textbf{Acknowledgements.} I am very grateful to Bal\'azs Szegedy for discussions that were crucial for my understanding of \cite{CamSzeg}, and to  anonymous referees for advice that helped to improve this paper. I also thank Yonatan Gutman, Frederick Manners and P\'eter Varj\'u for informing me of their work prior to publication, and for useful comments. The present work was supported by the ERC Consolidator Grant No. 617747.

\newpage

\section{Filtered nilmanifolds as compact nilspaces}\label{sec:filnilmcompns}

\medskip
Recall from \cite[Section 2.2]{Cand:Notes1} the notion of a filtration $G_\bullet$ on a group $G$ and the basic properties of  cube sets of the form $\cu^n(G_\bullet)$.
\begin{defn}\label{def:nilmani}
A \emph{nilmanifold} is a quotient space $G/\Gamma$ where $G$ is a connected nilpotent Lie group and $\Gamma$ is a discrete, cocompact subgroup. If $G_\bullet$ is a filtration on $G$ of degree at most $k$, with terms $G_i$ being closed subgroups of $G$, and with each subgroup $\Gamma \cap G_i$ being cocompact in $G_i$, then we call $(G/\Gamma, G_\bullet)$ a \emph{filtered nilmanifold} of degree at most $k$.
\end{defn}
\noindent The groups in the filtration $G_\bullet$ are sometimes required  to be connected (in particular in applications of nilmanifolds in arithmetic combinatorics; see for instance \cite[Definition 1.3]{GTarith}), but we shall not need this assumption here.

A compact nilspace $\ns$ is \emph{connected} if the underlying  topological space is connected. We now show that every filtered nilmanifold of degree $k$ can be viewed as a $k$-step connected compact nilspace, in the following sense.
\begin{proposition}\label{prop:nilmanilspace}
Let $G$ be a connected nilpotent Lie group with a degree-$k$ filtration $G_\bullet$ of closed subgroups, let $\Gamma\leq G$ be discrete, and for each $i\in [k]=\{1,2,\dots,k\}$ let $\Gamma_i=\Gamma \cap G_i$. Let $\ns$ denote the coset space $G/\Gamma$ equipped with the sets $\cu^n(\ns)=\{\pi_\Gamma\co \q: \q\in \cu^n(G_\bullet)\}$, where $\pi_\Gamma$ is the quotient map $G\to G/\Gamma$. Then $\ns$ is a connected compact nilspace if and only if $\Gamma_i$ is cocompact in $G_i$ for each $i\in [k]$.
\end{proposition}
\begin{proof}
The nilspace axioms are proved for $\ns$ in \cite[Proposition 2.3.1]{Cand:Notes1}. To complete the proof, we show that the compactness of each $G_i/\Gamma_i$ is necessary and sufficient for each set $\cu^n(\ns)$ to be closed in $\ns^{\{0,1\}^n}$.\\
\indent To see the sufficiency, it is enough to deduce that $\cu^n(G_\bullet)\cap \Gamma^{\{0,1\}^n}$ is cocompact in $\cu^n(G_\bullet)$, as then $\cu^n(\ns)$ is homeomorphic to the compact set $\cu^n(G_\bullet)/ \big(\cu^n(G_\bullet)\cap \Gamma^{\{0,1\}^n}\big)$ (see \cite[Proposition 2.7]{Leib}). This cocompactness is proved elsewhere in the literature but we include a proof here for completeness, following the argument from \cite[Lemma E.10]{GTlin}. By Definition \ref{def:nilmani} there is a compact set $K_j \subseteq G_j$ such that the product set $K_j \,\Gamma_j$ equals $G_j$. Recall from \cite[Lemma 2.2.5]{Cand:Notes1} that every cube $\q\in \cu^n(G_\bullet)$ has a unique factorization $\q= g_0^{F_0}g_1^{F_1}\cdots g_{2^n-1}^{F_{2^n-1}}$ 
with $g_i\in G_{\codim(F_i)}$ for each $i$, where the $F_i$ are the upper faces of $\{0,1\}^n$, of the form $F_i=F(v_i)=\{v:\supp(v)\supset \supp(v_i)\}$ and the $v_i$ are ordered in the colex order. We then have the following inductive formula for the coefficients $g_i$:
\begin{equation}\label{eq:compcubecoeffs}
 g_0=\q(v_0),\text{ and for each }i>0,\; g_i=  \Big( \prod_{ j < i : \;  \supp(v_j)\subset\supp(v_i)} g_j^{F_j}\Big)^{-1}\cdot \q \,(v_i).
\end{equation}
Let $H_0=\cu^n(G_\bullet)$ and for each $i\in [2^n]$ let $H_i=\{\q\in\cu^n(G_{\bullet}): g_0= \dots = g_{i-1} = \id_G\}$. It follows from \eqref{eq:compcubecoeffs} that $\q\in H_i$ if and only if $\q(v_j)=\id_G$ for all $j<i$. We deduce that $H_i$ is a normal subgroup of $\cu^n(G_\bullet)$. We shall now use these subgroups to show that every $\q\in \cu^n(G_\bullet)$ is a product of an element of $\cu^n(G_\bullet)\cap \Gamma^{\{0,1\}^n}$ with an element of $K:=\{ k_0^{F_0} \dots k_{2^n-1}^{F_{2^n-1}} : k_i \in  K_i\}$. This will prove the claimed sufficiency, since the set $K$ is compact. Let $i \in [2^n]$ and suppose that $\q \in H_{i-1}$. Then $\q = g_i^{F_i}\,\q'$, where $\q' \in H_i$. We then have $g_i^{F_i}=(k_i \gamma_i)^{F_i}$ where $k_i \in K_{\codim(F_i)}$ and $\gamma_i \in \Gamma_{\codim(F_i)}$. Since $H_i$ is normal, we have 
$\q = k_i^{F_i}\; \q'' \; \gamma_i^{F_i}$ with $\q''\in H_i$. Starting with any cube in $H_0$ and repeating this  procedure, we eventually obtain the desired factorization showing that $\q\in K\cdot\big(\cu^n(G_\bullet)\cap \Gamma^{\{0,1\}^n}\big)$.\\
\indent To see the necessity of each set $G_i/\Gamma_i$ being compact, let $Q_i$ denote the set of cubes $\q\in \cu^i(\ns)$ such that $\q(v)$ is the trivial coset $\Gamma$ for all $v\neq 0^i$. Since $\cu^i(\ns)$ is compact, so is $Q_i$. We have $\q=\pi_\Gamma\co \q'$ for some $\q'\in \cu^i(G_\bullet)$ with $\q'(v)\in \Gamma$ for $v\neq 0^i$. The fact that $\q'\in \cu^i(G_\bullet)$ implies that the Gray-code product $\sigma_i(\q')$ lies in $G_i$ (recall  \cite[Definition 2.2.22 and Proposition 2.2.25]{Cand:Notes1}). By the formula for the Gray-code product, we deduce that $\sigma_i(\q')\q'(0^i)^{-1} \in \Gamma$. It follows that $\q'(0^i)\in G_i\,\Gamma$. Thus the projection $\pi_{0^i}: \q\to \q(0^i)$ maps $Q_i$ into $(G_i\,\Gamma)/\Gamma$. Conversely, for every coset $g \Gamma$ with $g\in G_i$ there is a cube $\q\in Q_i$ with $\pi_{0^i}(\q)=g\Gamma$, namely $\q=\pi_\Gamma\co \q'$ where $\q'(0^i)=g$ and $\q'(v)=\id_G$ otherwise. It follows that $\pi_{0^i}(Q_i)=(G_i\,\Gamma)/\Gamma$, so the compactness of $Q_i$ implies that $(G_i\,\Gamma)/\Gamma$ is compact, whence  $G_i/\Gamma_i$ is compact.
\end{proof}

\chapter{Characterization of compact nilspaces}\label{chap:compnils}

\section{Topological preliminaries}\label{sec:toprelims}

\medskip
In order to characterize a general compact nilspace, we first have to examine how the basic algebraic structures treated in \cite{Cand:Notes1} behave under the additional topological assumptions in Definition \ref{def:compnils}.\\

\noindent We begin with a result showing that the closure condition on $\cu^n(\ns)$ needs to be checked only for $n=k+1$.

\begin{lemma}\label{lem:compclosure}
A $k$-step nilspace $\ns$ is a compact nilspace if and only if $\ns$ is a  compact space and $\cu^{k+1}(\ns)$ is a closed subset of $\ns^{\{0,1\}^{k+1}}$.
\end{lemma}
\begin{proof}
Note that for all $n\leq m$ we can view $\cu^n(\ns)$ as the projection of $\cu^m(\ns)$ to $\ns^{\{0,1\}^n}$ (by the composition axiom, embedding $\{0,1\}^n$ as a face in $\{0,1\}^m$). In particular, if we assume that $\cu^{k+1}(\ns)$ is closed then for every $n<k+1$ the projection $\cu^n(\ns)$ is also closed. If $n> k+1$, then for any $(k+1)$-dimensional face $F$ let
\[
Q_F=\big\{f:\{0,1\}^n\to \ns ~|~ f\co \phi_F \in \cu^{k+1}(\ns)\big\},
\]
where $f\co \phi_F$ is the restriction of $f$ to $F$ (recall \cite[Definition 1.1.5]{Cand:Notes1}). Thus $Q_F$ is the preimage of a closed set under the restriction map $f\mapsto f\co \phi_F $, and is therefore closed. Then, by \cite[Lemma 3.2.13]{Cand:Notes1}, we have $\cu^n(\ns)=\bigcap_F Q_F$ where $F$ runs through $(k+1)$-dimensional faces of $\{0,1\}^n$.
\end{proof}

\noindent In the remainder of this section we establish topological properties of the structures in \cite[Chapter 3]{Cand:Notes1} under the compactness assumption. The main result that we shall obtain along these lines is that if a finite-step nilspace is compact, then the associated abelian bundle obtained in \cite[Theorem 3.2.19]{Cand:Notes1} is endowed with a topology making it an iterated \emph{continuous} abelian bundle (in a sense that we shall formalize in Definition \ref{def:CpctAbBund}), with structure groups being compact abelian groups.

We begin by looking at the behaviour of some basic constructions. Recall from \cite[Definitions 3.1.19 and 3.1.22]{Cand:Notes1} the notions of arrow spaces and the spaces $\partial_x\ns$.

\begin{lemma}
Let $\ns$ be a compact nilspace, let $\ab$ be a compact abelian group, and fix any $x\in \ns$. Then the arrow spaces $\ns \Join_i \ns$, $i\geq 1$, the nilspaces $\cD_k(\ab)$, $k\geq 1$, and the space $\partial_x \ns$, are all compact nilspaces.
\end{lemma}

\begin{proof}
The space $\ns \times \ns$ underlying $\ns \Join_i \ns$ is compact and $\cu^{k+1}(\ns \Join_i \ns)$ is the image under a continuous map of a closed set of cubes in $\cu^{k+1+i}(\ns)$ (see \cite[(3.4)]{Cand:Notes1}), so it is a compact set as required. The cube set $\cu^{k+1}(\cD_k(\ab))$ is also compact, since it is the set of solutions of a linear equation (recall \cite[(2.9)]{Cand:Notes1}). Finally, the cube set $\cu^{k+1}(\partial_x \ns)$ is a closed subset of $\cu^{k+1}(\ns\Join_1 \ns)$ (recall \cite[Definition 3.1.22]{Cand:Notes1}).   
\end{proof}

\noindent We now look at the more involved constructions, especially those using quotients by equivalence relations. 
\begin{remark}\label{rem:stfacts} We shall use the following standard facts.
\begin{enumerate}\vspace{-0.2cm}
\item Let $X,Y$ be compact spaces (as in Definition \ref{def:compspace}). Then a function $f:X\rightarrow Y$ is continuous if and only if its graph is closed. (This is a special case of the closed graph theorem \cite[Ex. 8, p. 171]{Munkres}.)
\item Compact spaces are metrizable. (This follows from Urysohn's metrization theorem \cite[Thm. 34.1, p. 215]{Munkres}, using that a compact Hausdorff space is regular \cite[Ex. 3, p. 205]{Munkres}.)
\item Let $X$ be a compact space and let $\sim$ be a closed equivalence relation on $X$, that is $\{(x,y): x\sim y\}$ is a closed subset of $X\times X$. Then $X/\sim$ is compact in the quotient topology \cite[Prop. 8,  p. 105]{Bourb1}.
\end{enumerate}
\end{remark}
\noindent Recall from \cite[Definition 1.2.3]{Cand:Notes1} that $\ns$ is said to be $k$\emph{-fold ergodic} if $\cu^k(\ns)=\ns^{\{0,1\}^k}$.

\begin{lemma}\label{lem:topkfolderg} Let $\ns$ be a $k$-step $k$-fold ergodic compact nilspace. Then $\ns$ is isomorphic as a compact nilspace to $\cD_k(\ab)$ for some compact abelian group $\ab$. 
\end{lemma}

\begin{proof} As seen in \cite[Proposition 2.4.1]{Cand:Notes1}, in the case $k=1$, if we fix any $e\in \ns$ as a distinguished point, then the cube structure on $\ns$ yields a commutative group operation on $\ns$ with identity element $e$, making $\ns$ a principal homogeneous space of the resulting abelian group $\ab$. Let us check the continuity of addition and inverse on $\ab$. The graph of addition can be written
\[
\{(x_{00},x_{10},x_{01},x_{11})\in \cu^2(\ns) : x_{00}=e\}.
\]
The graph of the inverse is similarly defined by the equation $x_{00}=x_{11}=e$. These are closed sets, so the operations are continuous.

For $k>1$, as in \cite[Proposition 3.2.14]{Cand:Notes1} we consider the 1-step nilspace $\partial_e^{k-1}(\ns)$, which is here a compact nilspace, and therefore isomorphic to $\cD_1(\ab)$ for some compact abelian group $\ab$. We then argue just as in the proof of \cite[Proposition 3.2.14]{Cand:Notes1} to show that $\cu^n(\ns)=\cu^n(\cD_k(\ab))$ for all $n$. 
\end{proof}

\noindent Recall from \cite[Definition 3.2.3]{Cand:Notes1} the relation $\sim_k$ on a nilspace $\ns$ and the corresponding nilspace factor $\cF_k(\ns)$ from \cite[Lemma 3.2.10]{Cand:Notes1}.
\begin{lemma}\label{lem:compfactor} Let $\ns$ be a compact nilspace and let $k\in\N$. Then $\cF_k(\ns)$ with the quotient topology is a compact nilspace.
\end{lemma}

\begin{proof} As a subset of $\ns\times \ns$, the relation $\sim_k$ consists of  restrictions to the 1-face $\{0^{k+1},(1,0,\dots,0)\}$ of the cubes satisfying the condition in \cite[Lemma 3.2.4]{Cand:Notes1}. The set of these cubes is
\[
\cu^{k+1}(\ns)\;\cap\;\big\{f\in \ns^{\{0,1\}^{k+1}}~:~\forall\, v,v'\in \{0,1\}^{k+1}\setminus\{0^{k+1}\},\;f(v)=f(v')\big\},
\]
and is therefore closed.
\end{proof}

\noindent We now move toward the main result of this section, namely the addition of suitable topological properties to the bundle decomposition of a nilspace.

\subsection{Continuous abelian bundles}\label{sec:contbundec}

Recall from \cite[Definition 3.2.17]{Cand:Notes1} the notion of a $k$-fold abelian bundle. 
\begin{defn}\label{def:CpctAbBund}
Let $\bnd$ be an abelian bundle with base $S$, structure group $\ab$ and projection $\pi$. We say that $\bnd$ is  \emph{continuous} if the following conditions hold:
\begin{enumerate}
\item $\bnd$ and $S$ are topological spaces.
\item $\ab$ is an abelian topological group.
\item The action $\alpha:\ab\times \bnd\to \bnd$ is continuous.
\item A set $U\subset S$ is open if and only if $\pi^{-1}(U)$ is open in $\bnd$.
\end{enumerate}
We say that $\bnd$ is a \emph{compact} abelian bundle if, in addition to these conditions, $\bnd,S,\ab$ are compact spaces.
\end{defn}
\begin{remark}\label{rem:open-and-closed}
Recall that by \cite[Definition 3.2.17]{Cand:Notes1} the projection $\pi$ yields a bijection between $S$ and the orbits of $\alpha$, namely the bijection $s\mapsto \pi^{-1}(s)$. It follows that we obtain a definition equivalent to Definition \ref{def:CpctAbBund} if we replace condition (iv)  with the following:\vspace{0.1cm}\\
\hspace*{0.15cm} (iv')\; Denoting by $\bnd/\alpha$ the set of $\ab$-orbits  equipped with the quotient topology, we have that the map\\
\hspace*{1cm} $S\to \bnd/\alpha$, $\;s\mapsto \pi^{-1}(s)$ is a homeomorphism. \vspace{0.1cm}\\
\noindent Thus the base $S$ is topologically identified with the orbit space for the $\ab$-action on $\bnd$. In particular, since the orbit map $\bnd\to \bnd/\alpha$, $x\mapsto x+\ab$ is an open map, we have that the projection $\pi$ is an open map. Note also that if $\ab$ is compact, then the orbit map is also a closed map, and hence so is $\pi$.
\end{remark}
\begin{defn}\label{def:compdegkbund}
Let $\ns=\bnd_k$ be a $k$-fold abelian bundle, with factors $\bnd_i$, $i=0,\ldots,k$. We say that $\ns$ is a \emph{$k$-fold compact abelian bundle} if for each $i\in [k]$ the factor $\bnd_i$ is a compact $\ab_i$-bundle with base $\bnd_{i-1}$. A \emph{compact degree-$k$ bundle} is a $k$-fold compact abelian bundle that is also a compact cubespace with respect to the same topology, and such that condition (3.5) from \cite[Definition 3.2.18]{Cand:Notes1} is satisfied, that is, for every $i \in [0,k-1]$ and every $n\in \N$, we have $\cu^n(\bnd_i) = \{\pi_i\co \q: \q\in  \cu^n(\ns)\}$, and for every $\q\in \cu^n(\bnd_{i+1})$ we have
\begin{equation}\label{eq:comp-k-deg-bund}
\{\q_2 \in \cu^n(\bnd_{i+1}): \pi_i\co \q = \pi_i\co \q_2\} = \{\q+\q_3: \q_3\in \cu^n(\cD_{i+1}(\ab_{i+1}))\}.
\end{equation}
\end{defn}
We can now establish the main result of this subsection.
\begin{proposition}\label{prop:topbundec} A compact cubespace $\ns$ is a compact degree-$k$ bundle if and only if $\ns$ is a $k$-step compact nilspace. 
\end{proposition}

\begin{proof} The algebraic part of the statement is given by \cite[Theorem 3.2.19]{Cand:Notes1}, so we only need to check the topological properties in each direction.\\
\indent If $\ns$ is a compact degree-$k$ bundle then in particular it is a compact cubespace satisfying the nilspace axioms, so it is a $k$-step compact nilspace. Conversely, if $\ns$ is a $k$-step compact nilspace then, to show that the degree-$k$ bundle structure given by \cite[Theorem 3.2.19]{Cand:Notes1} is compact, it suffices to show that the abelian structure group $\ab_k$ can be equipped with a topology making it a compact abelian group such that $\ns$ is a compact $\ab_k$-bundle with base $\cF_{k-1}(\ns)$ (if we show this then the result follows by repeating the argument for $\ab_{k-1}$ and so on by induction). For this it suffices to show that $\ab_k$ can be equipped with a compact topology such that the action of $\ab_k$ on $\ns$ is continuous (since Lemma \ref{lem:compfactor} then implies the remaining condition (iv') in Definition \ref{def:CpctAbBund}).  
Let $F$ be a class of $\sim_{k-1}$ in $\ns$. From the proof of \cite[Corollary 3.2.16]{Cand:Notes1} we know that $F$ together with the $F$-valued cubes in each $\cu^n(\ns)$ is a $k$-fold ergodic $k$-step nilspace. Using \cite[Lemma 3.2.4]{Cand:Notes1} we deduce that $F$ is closed and it follows that $F$ is a compact nilspace. Then by Lemma \ref{lem:topkfolderg} there is a compact topology on $\ab_k$ such that $F$ is isomorphic to $\cD_k(\ab_k)$ as a compact nilspace. We now check that the action of $\ab_k$ is continuous on $F$ by showing that its graph $\{(a,y_0,y_0+a): a\in \ab,y_0\in \ns\}$ is closed in $\ab_k\times \ns^2$. This graph can be described using the arrow space $\nss=\ns\Join_1 \ns$. Indeed, recall from \cite[Lemma 3.2.15]{Cand:Notes1} that for $x=(x_0,x_1)$, $y=(y_0,y_1)$ in $F^2\subset \nss$ we have $x\sim_{k-1} y$ if and only if $x_0-x_1=y_0-y_1$ in $\ab_k$. It follows that, fixing some $e\in F$, the graph above is homeomorphic to the closed set $\{(x,y) \in F^2\times F^2: x \sim_{k-1} y,\; x_0 =e\}$. We have thus shown that from any given fibre $F$ the group $\ab_k$ acquires a compatible compact topology making its action on $F$ continuous. To see that this topology is the same for every fibre, recall that by \cite[Lemma 3.2.24]{Cand:Notes1} for any two fibres $F_0,F_1$ and fixed points $x_i\in F_i$, we have an isomorphism $\vartheta:\ab_{F_0}\to \ab_{F_1}$, defined by $\vartheta(a)=b$ if and only if $(x_0,x_1)\sim_{k-1}(x_0+a,x_1+b)$. We then see that $\vartheta$ is continuous by identifying its graph with the closed set $\{y \in F_0\times F_1: y \sim_{k-1} (x_0,x_1)\}$. It follows that $\vartheta$ is a homeomorphism.
\end{proof}

\noindent Recall the notions of a bundle morphism \cite[Definition 3.3.1]{Cand:Notes1} and of a sub-bundle \cite[Definition 3.3.3]{Cand:Notes1}. These notions are transferred to the setting of compact abelian bundles in a straightforward manner, by specifying in their definitions that all the maps involved are continuous. The same holds for the kernel of a bundle morphism $\psi:\bnd\to\bnd'$ (recall \cite[Definition 3.3.4]{Cand:Notes1}). Note that such a kernel $K$ is indeed a compact abelian bundle, since each set $K_i$ in \cite[Definition 3.3.4]{Cand:Notes1} is a closed subset of $\bnd_i\times \bnd'$ and therefore compact, and every structure group $\ker(\alpha_i)$ is compact.

Recall also that in the purely algebraic setting we had a description of a set  of restricted morphisms $\hom_f(P,\ns)$ as a sub-bundle of $\ns^P$, namely \cite[Lemma 3.3.11]{Cand:Notes1}. This has the following version for compact nilspaces.

\begin{lemma}\label{lem:top-restrmorph=subbund}
Let $P$ be a subcubespace of $\{0,1\}^n$ with the extension property, let $S$ be a subcubespace of $P$ with the extension property in $P$, let $\ns$ be a $k$-step compact nilspace and let $f:S\to \ns$ be a morphism. Then $\hom_f(P,\ns)$ is a compact $k$-fold abelian bundle that is a sub-bundle of $\ns^P$, with factors $\hom_{\pi_i\co f}(P,\ns_i)$ and structure groups $\hom_{S\to 0}(P,\cD_i(\ab_i))$, where $\ab_i$ is the $i$-th structure group of $\ns$.
\end{lemma}
\begin{proof}
We just have to check the topological properties. The set $\hom_f(P,\ns)$ is   closed in $\ns^P$ and therefore compact. For each $i\in [k]$ the abelian group $\hom_{S\to 0}(P,\cD_i(\ab_i))$ is also compact, as a closed subgroup of $\ab_i^P$, and the action of $\hom_{S\to 0}(P,\cD_i(\ab_i))$ inherits continuity from that of $\ab_i$. Condition (iv) from Definition \ref{def:CpctAbBund} is also inherited from $\ns^P$ since the topology on each factor $\hom_{\pi_i\co f}(P,\ns_i)$ is the subspace topology from $\ns_i^P$.
\end{proof}

\subsection{Metrics}

\noindent By Remark \ref{rem:stfacts} (ii), we know that a compact nilspace  can always be equipped with a compatible metric. We shall often use the fact that this metric can be assumed to be invariant under the action of the last  structure group. We record this fact as follows.

\begin{lemma}\label{lem:d-invariance}
Let $\ns$ be a $k$-step compact nilspace and let $d_0$ be a metric on $\ns$ generating its topology. Define $d(x,y)=\int_{z\in \ab_k} d_0(x+z,y+z) \ud\mu_{\ab_k}(z)$, for any $x,y\in \ns$. Then $d$ is a $\ab_k$-invariant metric that generates the same topology on $\ns$.
\end{lemma}
\noindent This invariant metric is a basic construction in the theory of $G$-spaces  (see \cite[Proposition 1.1.12]{Pal-G}).
\begin{proof}
It is clear that $d$ is a $\ab_k$-invariant metric. To see that $d$ generates the same topology as $d_0$, note first that the function $\ns\times \ns\to \R_{\geq 0}$, $(x,y)\mapsto d(x,y)$ is continuous and in particular for a fixed $x\in \ns$ we have that $y\mapsto d(x,y)$ is continuous (relative to $d_0$). It follows that the topology generated by $d$ is a subtopology of the given topology on $\ns$. In particular the former topology is compact. However, that topology is also Hausdorff, so it is maximal among the compact topologies on $\ns$, and so it must be the original topology on $\ns$ (see \cite[\S 9.4, Corollary 3]{Bourb1}).
\end{proof}
\noindent From now on we always assume that a metric on $\ns$ is $\ab_k$-invariant in this sense. Given such a metric $d$ on $\ns$, we may define the following quotient metric on $\cF_{k-1}(\ns)$ (see  \cite[Proposition 1.1.12]{Pal-G}):
\begin{equation}\label{eq:quotientmetric}
d'(x',y')=\inf\big\{d(x,y): x,y\in \ns,\; \pi_{k-1}(x)=x',\; \pi_{k-1}(y)=y'\big\}.
\end{equation}
\noindent In the sequel we shall use the following fact several times.
\begin{lemma}\label{lem:contcomp}
Let $\ns$ be a $k$-step compact nilspace, and let $\cor^{k+1}(\ns)$ denote the set of $(k+1)$-corners on $\ns$ with the subspace topology obtained from the product topology on $\ns^{\{0,1\}^{k+1}\setminus \{1^{k+1}\}}$. Let $\comp:\cor^{k+1}(\ns)\to \ns$ denote the map sending a $(k+1)$-corner $\q'$ to $\q(1^{k+1})$, where $\q$ is the unique completion of $\q'$. Then $\comp$ is continuous.
\end{lemma}
\begin{proof}
We argue by contradiction. If $\comp$ were not continuous then there would exist some corner $\q'$, some open set $U\ni \comp(\q')$ and a sequence of corners $\q_n'$ converging to $\q'$ such that for all $n$ we have $\comp(\q_n')\not\in U$. Let $\q_n$ denote the completion in $\cu^{k+1}(\ns)$ of $\q'_n$. By the compactness of $\cu^{k+1}(\ns)$ there exists a subsequence $(\q_m)$ of $(\q_n)$  such that $\q_m$ converges to some cube $\q\in \cu^{k+1}(\ns)$. This convergence combined with the assumption that $\q_n'\to \q'$ implies that for each $v\neq 1^{k+1}$ we have $\q(v)=\q'(v)$, so by uniqueness of completion we have $\q(1^{k+1})=\comp(\q')\in U$. On the other hand, the convergence $\q_m\to \q$ also implies that $\q_m(1^{k+1})\to \q(1^{k+1})$. Since by assumption $\q_m(1^{k+1})=\comp(\q_m')\not\in U$ for all $m$, we have $\q(1^{k+1})\not\in U$, a contradiction.
\end{proof}
\noindent Let us record also the following consequence of Lemma \ref{lem:contcomp}, to the effect that proximate cubes on $\cF_{k-1}(\ns)$ can always be lifted to proximate cubes on $\ns$.
\begin{lemma}\label{lem:closelifts}
Let $\ns$ be a $k$-step compact nilspace and let $d,d'$ be the metrics on $\ns$, $\cF_{k-1}(\ns)$ given by Lemma \ref{lem:d-invariance} and \eqref{eq:quotientmetric}  respectively. Then for every $\epsilon>0$ there exists $\delta>0$ such that the following holds. If $\q_1,\q_2\in \cu^{k+1}(\cF_{k-1}(\ns))$ satisfy $d'(\q_1(v),\q_2(v))\leq \delta$ for all $v$, then there exist cubes $\tilde\q_1,\tilde\q_2\in \cu^{k+1}(\ns)$ such that $\pi_{k-1}\co\tilde\q_i=\q_i$ for $i=1,2$ and $d\big(\tilde\q_1(v),\tilde\q_2(v)\big)\leq \epsilon$ for all $v$.
\end{lemma}
\begin{proof}
Let $P=\{0,1\}^{k+1}\setminus\{1^{k+1}\}$ and note that, by definition of $d'$ and \cite[Remark 3.2.12]{Cand:Notes1}, the corners $\q_i|_P$ can be lifted to corners $\q_i'$, $i=1,2$ such that $d\big(\q_1'(v),\q_2'(v)\big)\leq \delta$ for every $v\in P$. By Lemma \ref{lem:contcomp}, if $\delta\leq \epsilon$ is sufficiently small then the unique completions $\tilde\q_i$ of $\q_i'$ satisfy $d\big(\tilde\q_1(1^{k+1}),\tilde\q_2(1^{k+1})\big)\leq \epsilon$.
\end{proof}

\medskip
\section{Measure-theoretic preliminaries}\label{sec:measprel}

\medskip
In this section we collect the main tools from measure theory that we shall use in the sequel.\\
\indent A notion that unifies most of these tools is that of a \emph{continuous system of measures}, discussed in Subsection \ref{subsec:csm}.\\
\indent Using continuous systems of measures, we shall then define in Subsection \ref{subsec:Haar} an invariant measure on a continuous abelian bundle, and use this to construct a generalization of the Haar measure for  compact nilspaces.\\
\indent In Subsection \ref{subsec:morph-prob-spaces} we collect several constructions of probability spaces that will be used repeatedly later in the chapter.

\subsection{Continuous systems of measures}\label{subsec:csm}
Recall that a measure $\mu$ on a measure space $X$ is said to be \emph{concentrated on} $S\subset X$ if $\mu(X\setminus S)=0$.

The study of continuous systems of measures goes back at least to Bourbaki (see \cite[Ch. V, \S 3]{BourInt}). A recent treatment is given in \cite{C&Gra}. We shall use the following definition.

\begin{defn}\label{def:CSM}
Let $X,Y$ be compact spaces and let $\pi : X \to Y$ be continuous. A \emph{continuous system of measures} (CSM) on the map $\pi$ is a family of Borel measures $\{ \mu_y :y \in Y\}$ on $X$ satisfying the following conditions:
  \begin{enumerate}
    \item For every $y\in Y$ the measure $\mu_y$ is concentrated on $\pi^{-1}(y)$.
    \item\label{contprop} For every continuous function $f:X\to\C$, the function $Y\to \C$, $y \mapsto \int_{\pi^{-1}(y)} f \; \ud\mu_y$ is continuous.
  \end{enumerate}
\end{defn}
\noindent As explained in \cite{C&Gra}, replacing $\C$ by $\R$ in condition \eqref{contprop} yields an equivalent definition.\\

\noindent Note that all the CSMs that we shall consider in these notes consist of strictly positive measures $\mu_y$, i.e. Borel measures taking positive values on non-empty open sets. 

A simple example of a CSM is obtained by taking a product space, as follows.

\begin{lemma}\label{lem:GenCSM}
Let $V$ and $Y$ be compact spaces, let $\mu$ be a Borel probability measure on $V$, let $X=V\times Y$ and let $\pi : X\to Y,\, (v,y)\mapsto y$. For each $y\in Y$ let $\mu_y$ denote the measure on $X$ concentrated on $\pi^{-1}(y)$, where it equals the pushforward of $\mu$ under the map $V\to \pi^{-1}(y)$, $v\mapsto (v,y)$. Then $\{\mu_y:y\in Y\}$ is a CSM on $\pi$.
\end{lemma}
\begin{proof}
Let $f : V\times Y \to \R$ be a continuous function. Then $f$ is a uniform limit of functions of the form $F:(v,y) \mapsto \sum_{i=1}^n f_{i,1}(v)f_{i,2}(y)$ where $f_{i,1}:V\to \R$, $f_{i,2}:Y\to \R$ are continuous. Indeed, the algebra of such functions $F$ satisfies the assumptions of the Stone-Weierstrass theorem \cite[Appendix A, \S A14]{Rudin}. (The fact that this algebra separates points follows from Urysohn's lemma \cite[Theorem 33.1]{Munkres}.) Now, for a function $f_{i,1}(v)f_{i,2}(y)$ as above, it is clear that integrating over $v$ yields a continuous function of $y$. It follows that $y\mapsto \int_{\pi^{-1}(y)} f\; \ud\mu_y$ is a uniform limit of continuous functions and is therefore continuous.
\end{proof}

\noindent One of our main uses of CSMs below is to equip any compact abelian bundle with a probability measure that is invariant under the action of the structure group. This measure will be obtained as a simple application of the following result concerning general CSMs.
\begin{lemma}\label{lem:CSM-measure}
Let $\{\mu_y:y\in Y\}$ be a family of Borel probability measures on $X$ forming a CSM on $\pi : X \to Y$, and let $\nu$ be a Borel probability measure on $Y$. Then the following function on Borel sets $E\subset X$ is a Borel probability measure:
\[
\mu(E) = \int_Y \mu_y\big(E \cap \pi^{-1}(y)\big)\, \ud\nu(y).
\]
\end{lemma}

\begin{proof}
This is a special case of \cite[Corollary 3.7]{C&Gra}, using \cite[Proposition 2.23]{C&Gra}.
\end{proof}

\subsection{Haar measure on compact abelian bundles and nilspaces}\label{subsec:Haar}

Recall that every compact abelian group can be equipped with a regular Borel probability measure that is invariant under translation, called the (normalized) Haar measure \cite{Rudin}. In this subsection we shall construct a generalization of this measure for compact nilspaces. As we shall see,  there is a natural inductive construction of a probability measure on a $k$-fold compact abelian bundle, starting with the Haar measure on a compact abelian group. This construction will yield the desired measure on a $k$-step compact nilspace, thanks to \cite[Theorem 3.2.19]{Cand:Notes1}.

Our first step, then, is to define a Haar measure on a compact abelian bundle.

\begin{lemma}\label{lem:bundHaar}
Let $\bnd$ be a compact abelian bundle with base $S$, structure group $\ab$ and projection $\pi$. Let $\mu_S$ be a regular Borel probability measure on $S$. Then there is a unique regular Borel probability measure $\mu$ on $\bnd$ that is invariant under the action of $\ab$ and satisfies $\mu_S=\mu\co\pi^{-1}$.
\end{lemma}
We call $\mu$ the \emph{Haar measure} on $\bnd$ (given $\mu_S$).
\begin{proof}
For each $s\in S$ the set $\pi^{-1}(s)$ is a principal homogeneous space of $\ab$, so there is a $\ab$-equivariant homeomorphism $f:\ab\to \pi^{-1}(s)$. Let $\mu_s$ be the pushfoward measure $\mu_{\ab}\co f^{-1}$ on $\pi^{-1}(s)$, where $\mu_{\ab}$ is the Haar measure on $\ab$. If we show that $\{\mu_s:s\in S\}$ is a CSM on $\pi$, then by Lemma \ref{lem:CSM-measure} we can define a Borel measure $\mu$ on $\bnd$ by the following formula:
\begin{equation}\label{eq:HaarMeas}
\mu(E)=\int_S \; \mu_s\big(\pi^{-1}(s)\cap E\big)\; \ud\mu_S(s),
\end{equation}
for every Borel set $E\subset \bnd$.

To show that $\{\mu_s:s\in S\}$ is a CSM on $\pi$ we check condition \eqref{contprop} from Definition \ref{def:CSM}. Given any continuous function $f : \bnd \to \mathbb{R}$, we need to show that $g : S \to \mathbb{R}$, 
$s\mapsto \int_{\pi^{-1}(s)} f \, \ud\mu_s$
is continuous. Since $S$ has the quotient topology (Definition \ref{def:CpctAbBund}), it suffices to show that $g \circ \pi : \bnd \to \R$ is continuous.
Now $g(\pi(b)) = \int_{\ab} f(a+b) \; \ud\mu_{\ab}(a) = \int_{\ab} f'(a,b) \; \ud\mu_{\ab}(a)$, where the function $f':\ab\times \bnd \to \R, \; (a,b) \mapsto f(a+b)$ is continuous. Therefore, the continuity of $g \circ \pi$ follows from (the proof of) Lemma \ref{lem:GenCSM} applied with $V=\ab$ and $Y=\bnd$.\\
\indent We thus obtain the Borel measure $\mu$ given by \eqref{eq:HaarMeas}, and  $\mu$ is regular by the standard fact that any Borel probability measure on a metric space is regular \cite[Theorem 1.1]{Bill2}. The $\ab$-invariance of $\mu$ clearly follows from that of each measure $\mu_s$. The uniqueness can be deduced from uniqueness of Haar measure on $\ab$ combined with a disintegration result for invariant measures, such as \cite[Proposition 2]{Rip}.
\end{proof}

As mentioned above, this lemma yields the following result.

\begin{proposition}\label{prop:nilspaceHaar}
Let $\bnd$ be a compact $k$-fold abelian bundle, with factors $\bnd_0,\bnd_1,\dots,\bnd_k=\bnd$ and structure groups $\ab_1,\dots,\ab_k$. Then there is a unique regular Borel probability measure $\mu$ on $\bnd$ with the following property: for each $i\in [k]$, letting $\pi_i$ be the projection $\bnd\to \bnd_i$, we have that the Borel probability measure $\mu\co\pi_i^{-1}$ on $\bnd_i$ is invariant under the action of $\ab_i$. 
\end{proposition}
\noindent In particular, for every $k$-step compact nilspace $\ns$, applying the proposition to the bundle structure we obtain the measure that we shall call the \emph{Haar measure on} $\ns$.
\begin{proof}
By induction on $i\in [k]$, starting with the Haar probability on $\ab_1$ and applying Lemma \ref{lem:bundHaar} for $i=2,\dots,k$.
\end{proof}

\noindent Recall from \cite[Definition 3.3.1]{Cand:Notes1} that a bundle morphism $\phi:\bnd\to \bnd'$ between $k$-fold abelian bundles is said to be \emph{totally surjective} if the structure morphism $\alpha_i:\ab_i\to\ab_i'$ is surjective for each $i\in [k]$.

The following basic result is the analogue for nilspaces of the fact that continuous surjective homomorphisms between compact abelian groups preserve the Haar measures.

\begin{lemma}\label{lem:MeasPres}
Let $\bnd,\bnd'$ be compact $k$-fold abelian bundles, and let $\phi:\bnd\to \bnd'$ be a totally surjective continuous bundle morphism.  Then $\phi$ preserves the Haar measures, that is, for every Borel set $E\subset \bnd'$ we have $\mu_{\bnd}\big(\phi^{-1}(E)\big)=\mu_{\bnd'}(E)$.
\end{lemma}

\begin{proof}
We argue by induction on $k$. The case $k=1$ follows from the fact recalled above. For $k>1$, recall that by \cite[Definition 3.3.1]{Cand:Notes1} the map $\phi$ induces a totally surjective continuous bundle morphism $\phi_{k-1}:\bnd_{k-1}\to \bnd'_{k-1}$. Let $\{\mu_s:s\in \bnd_{k-1}\}$ and $\{\mu'_{s'}:s'\in \bnd'_{k-1}\}$ be the CSMs on $\pi_{k-1}:\bnd\to\bnd_{k-1}$, $\pi_{k-1}':\bnd'\to\bnd_{k-1}'$ respectively (as in the proof of Lemma \ref{lem:bundHaar}). For $s\in \bnd_{k-1}$ and any Borel set $E \subset \bnd'$, let $f_{\bnd}(s)=\mu_s\big(\pi_{k-1}^{-1}(s)\cap \phi^{-1}(E)\big)$, and for $s'\in \bnd_{k-1}'$ let $f_{\bnd'}(s')=\mu'_{s'}\big({\pi'}^{-1}_{k-1}(s')\cap E\big)$. The restriction of $\phi$ to a fibre $\pi_{k-1}^{-1}(s)$, being a continuous affine homomorphism onto some fibre ${\pi'}^{-1}_{k-1}(s')$, preserves the measures $\mu_s,\mu'_{s'}$. Hence we have $f_{\bnd}(s)= f_{\bnd'}(\phi_{k-1}(s))$ for each $s \in \bnd_{k-1}$. Since $\phi_{k-1}$ is measure-preserving (by the induction hypothesis), we have
\[
\mu_{\bnd}\big(\phi^{-1}(E)\big) = \int_{\bnd_{k-1}} f_{\bnd}\ud \mu_{\bnd_{k-1}}= \int_{\bnd_{k-1}} f_{\bnd'}\co\phi_{k-1}\ud \mu_{\bnd_{k-1}}
 =  \int_{\bnd_{k-1}'} f_{\bnd'}\ud \mu_{\bnd_{k-1}'}= \mu_{\bnd'}(E). \qedhere
\]
\end{proof}
\noindent Recall from \cite[Definition 3.3.7]{Cand:Notes1} the notion of a fibre-surjective morphism.
\begin{corollary}\label{cor:ctsfibsurmorph}
Let $\ns,\ns'$ be $k$-step compact nilspaces and let $\phi:\ns\to\ns'$ be a continuous fibre-surjective morphism. Then $\phi$ preserves the Haar measures.
\end{corollary}
\begin{proof}
We combine Lemma \ref{lem:MeasPres} with \cite[Lemma 3.3.8]{Cand:Notes1}.
\end{proof}
\noindent In the sequel we shall have to use CSMs on certain maps from cube sets $\cu^n(\ns)$ to $\ns$, such as the evaluation map $\q\mapsto \q(0^n)$. These CSMs will be obtained as special cases of the construction in the next lemma.\\
\indent Recall from \cite[Definition 3.2.17]{Cand:Notes1} the notion of a \emph{relative} $k$-fold abelian bundle. We may similarly define a \emph{compact} relative $k$-fold abelian bundle $\bnd$ with ground set $\bnd_0$. Note that each fibre of the projection $\pi_0:\bnd\to\bnd_0$ is then itself a compact $k$-fold abelian bundle and so it has a Haar measure by  Lemma \ref{lem:bundHaar}. 

\begin{lemma}\label{lem:relcsm} Let $\bnd$ be a compact relative $k$-fold abelian bundle with factors $\bnd_0, \bnd_1, \ldots,\bnd_{k-1}$. Then the Haar measures on the fibres of the projection $\pi_0:\bnd\to \bnd_0$ form a CSM on $\pi_0$.
\end{lemma}

To prove this we shall have to compose several CSMs, in the following sense.

\begin{defn}
Given CSMs $\{\mu_y: y \in Y\}$ on $\pi : X \to Y$ and $\{\mu_z: z \in Z\}$ on $\tau : Y \to Z$, their \emph{composition} is the CSM $\{\nu_z: z \in Z\}$ on $\tau \co \pi : X \to Z$ consisting of measures defined for Borel sets $E\subset X$ by
\begin{equation}\label{eq:CSMcomp}
\nu_z(E) = \int_{\tau^{-1}(z)} \mu_y\big( \pi^{-1}(y) \cap E\big) \; \ud\mu_z(y).
\end{equation}
\end{defn}
The fact that this composition is indeed a CSM is established in \cite[Proposition 3.3]{C&Gra}.

\begin{proof}[Proof of Lemma \ref{lem:relcsm}]
By the proof of Lemma \ref{lem:bundHaar}, for each $i\in [k]$ there is a CSM on the bundle map $\pi_{i-1,i}:\bnd_i\to \bnd_{i-1}$ consisting of the pushforward  of the Haar measure on $\ab_i$ to each fibre of $\pi_{i-1,i}$. Let us compose these CSMs for $i\in [k]$, and let $\{\mu_x: x\in \bnd_0\}$ be the resulting CSM on $\pi_0$. For each $x\in \bnd_0$, a straightforward inductive argument using \eqref{eq:CSMcomp} shows that the measure $\mu_x$ on the compact $k$-fold abelian bundle $\pi^{-1}(x)$  satisfies the properties of the measure in Proposition \ref{prop:nilspaceHaar}, so it must be the Haar measure on $\pi^{-1}(x)$.
\end{proof}
\noindent The following result is an analogue for compact abelian bundles of the quotient integral formula.
\begin{lemma}\label{lem:quotint}
Let $\bnd,\bnd'$ be compact $k$-fold abelian bundles, with Haar measures $\mu,\mu'$ respectively, and let $\phi:\bnd\to \bnd'$ be a totally surjective continuous bundle morphism. For each $t\in \bnd'$ let $\mu_t$ denote the Haar measure on $\phi^{-1}(t)$. Then the measures $\mu_t$, $t\in \bnd'$ disintegrate $\mu$ relative to $\mu'$, that is for every Borel set $E\subset \bnd$ we have $\mu(E)=\int_{\bnd'}\mu_t\big(\phi^{-1}(t)\cap E\big)\,\ud\mu'(t)$.
\end{lemma}
Recall from \cite[Definition 3.3.4]{Cand:Notes1} the notion of the kernel of a bundle morphism.
\begin{proof}
The kernel of $\phi$ is a compact relative $k$-fold abelian bundle, with ground set $\bnd'$, and the fibres $\phi^{-1}(t)$ are compact $k$-fold abelian bundles (recall \cite[Lemma 3.3.6]{Cand:Notes1}). By Lemma \ref{lem:relcsm} the Haar measures $\mu_t$ form a CSM on $\phi$. By Lemma \ref{lem:CSM-measure}, with this CSM and $\mu'$ we can define a Borel probability $\nu$ on $\bnd$ by
\begin{equation}\label{eq:nukey}
\nu(E)=\int_{\bnd'}\mu_t\big(\phi^{-1}(t)\cap E\big)\,\ud\mu'(t).
\end{equation}
We have to show that $\nu=\mu$. For this it suffices to show that $\nu$ has the invariance properties that characterize the Haar measure $\mu$, namely that for each projection $\pi_i:\bnd\to\bnd_i$ the measure $\nu\co\pi_i^{-1}$ on $\bnd_i$ is $\ab_i$-invariant (recall Proposition \ref{prop:nilspaceHaar}). We prove this by induction on $k$.\\
\indent For $k=1$, we have that $\bnd,\bnd'$ are principal homogeneous spaces of compact abelian groups $\ab,\ab'$ respectively, and we just have to show that $\nu$ is $\ab$-invariant. But from the quotient integral formula \cite[Theorem 1.5.2]{D&E} it actually follows that $\nu=\mu$ in this case, so $\nu$ is indeed $\ab$-invariant.\\
\indent Supposing now that the claim holds for $k\geq 1$, we prove it for $k+1$. First we show that we can suppose by induction that $\nu\co\pi_i^{-1}$ is already $\ab_i$-invariant on $\bnd_i$ for each $i\in [k]$. For this it suffices to show that $\nu\co \pi_k^{-1}$ is of the same form as $\nu$ (as in \eqref{eq:nukey}), but relative to the totally surjective morphism $\phi_k:\bnd_k\to \bnd_k'$ induced by $\phi:\bnd_{k+1}\to \bnd_{k+1}'$. By construction, the Haar measure $\mu'$ has a disintegration relative to the Haar measure $\mu'_k$ on $\bnd'_k$, into measures $\mu'_s$, $s\in \bnd_k'$, each of which is the Haar measure on  $\ab_{k+1}'$ pushed forward to the fibre ${\pi'_k}^{-1}(s)$. Thus for any Borel set $E\subset \bnd_k$ we have
\[
\nu\co\pi_k^{-1}(E)\; = \;\int_{\bnd'}\mu_t\big(\pi_k^{-1}(E)\cap \phi^{-1}(t)\big) \ud\mu'(t) \; = \; \int_{\bnd'_k}\Big(\int_{{\pi'_k}^{-1}(s)} \mu_t\big(\pi_k^{-1}(E)\cap \phi^{-1}(t)\big) \ud\mu'_s(t)\Big) \ud\mu'_k(s).
\]
For each $s\in \bnd'_k$ and each $t\in \bnd'$ with $\pi'_k(t)=s$, note that $\phi^{-1}(t)$ is a $(k+1)$-fold compact abelian bundle, with $k$-th factor $\phi_k^{-1}(s)$ having Haar measure denoted $\mu_s$. Then $\pi_k$ restricted to $\phi^{-1}(t)$ pushes $\mu_t$ forward to $\mu_s$. It follows that for any such $s,t$ we have $\mu_t\big(\pi_k^{-1}(E)\cap \phi^{-1}(t)\big)=\mu_s\big(E\cap \phi_k^{-1}(s)\big)$. Hence
\[
\nu\co\pi_k^{-1}(E) = \int_{\bnd'_k}\Big(\int_{{\pi'_k}^{-1}(s)} \mu_s\big(E\cap \phi_k^{-1}(s)\big)\ud\mu_s'(t)\Big) \ud\mu_k'(s) = \int_{\bnd'_k} \mu_s\big(E\cap \phi_k^{-1}(s)\big) \ud\mu_k'(s).
\]
The right side here is indeed the measure on $\bnd_k$ constructed in the same way as $\nu$ in \eqref{eq:nukey}. Therefore, as mentioned above, by induction it now suffices to show that $\nu$ is $\ab_{k+1}$-invariant.\\
\indent Again we use the disintegration of $\mu'$ relative to $\mu'_{k-1}$, so that for any Borel set $E\subset \bnd'$ we have
\[
\nu(E) = \int_{\bnd'_k}\Big(\int_{{\pi'_k}^{-1}(s)} \mu_t\big(E\cap \phi^{-1}(t)\big) \ud\mu_s'(t)\Big) \ud\mu_k'(s).
\]
For each $s\in \bnd_k'$, we claim that the inner integral on the right side here integrates $1_E$ over a union of orbits of $\ab_{k+1}$. Indeed, we have $\mu_t\big( E\cap \phi^{-1}(t)\big)=\int_{\phi^{-1}(t)} 1_E(x)\ud\mu_t(x) $, and by disintegrating $\mu_t$ into the fibres of $\pi_k:\phi^{-1}(t)\to \phi_k^{-1}(s)$, the inner integral in question is written
\[
\int_{{\pi'_k}^{-1}(s)} \; \int_{\phi_k^{-1}(s)} \;\int_{\pi_k^{-1}(y)\cap \phi^{-1}(t)} 1_E(x) \;\ud\mu_{\pi_k^{-1}(y)}(x) \; \ud\mu_s(y) \; \ud\mu_s'(t).
\]
By Fubini's theorem we can interchange the two outer integrals, so this equals
\[
\int_{\phi_k^{-1}(s)} \; \int_{{\pi'_k}^{-1}(s)} \;\int_{\pi_k^{-1}(y)\cap \phi^{-1}(t)} 1_E(x) \;\ud\mu_{\pi_k^{-1}(y)}(x) \; \ud\mu_s'(t)\; \ud\mu_s(y) .
\]
For each $y\in  \phi_k^{-1}(s)$, note that ${\pi'_k}^{-1}(s)$ is a $\ab_{k+1}'$-torsor (or principal homogeneous space of $\ab_{k+1}'$) and for every $t$ in this torsor we have that $\pi_k^{-1}(y)\cap \phi^{-1}(t)$ is a $\ker(\alpha_{k+1})$-torsor, where $\alpha_{k+1}$ is the $(k+1)$-th structure morphism of $\phi$ (recall \cite[Definition 3.3.1 and Lemma 3.3.6]{Cand:Notes1}). Fixing any $t_0$ in the former torsor, and then for each $r\in \ab_{k+1}'$ fixing some $x_r\in \pi_k^{-1}(y)\cap \phi^{-1}(t_0+r)$, the last inner double-integral equals
\[
\int_{\ab_{k+1}'}  \int_{\pi_k^{-1}(y)\cap \phi^{-1}(t_0+r)}  1_E(x)\, \ud\mu_{\pi_k^{-1}(y)}(x)\, \ud\mu_{\ab_{k+1}'}(r)
= \int_{\ab_{k+1}'} \int_{\ker(\alpha_{k+1})} 1_E(x_r+u) \,\ud\mu_{\ker(\alpha_{k+1})}(u)\, \ud\mu_{\ab_{k+1}'}(r).
\]
Note that the sets $\{x_r+u:u\in \ker(\alpha_{k+1})\}$, $r\in \ab_{k+1}'$ form a partition of $\pi_k^{-1}(y)$. It follows from the quotient integral formula \cite[Theorem 1.5.2]{D&E} that for any $x_0\in \pi_k^{-1}(y)$ this integral equals 
\[
\int_{\ab_{k+1}} 1_E(x_0+z) \ud\mu_{\ab_{k+1}}(z)=  \int_{\pi_k^{-1}(y)} 1_E(x) \ud\mu_{\pi_k^{-1}(y)}(x).
\]
This is clearly invariant under shifting $E$ by any fixed $z\in \ab_{k+1}$. The $\ab_{k+1}$-invariance of $\nu$ follows.
\end{proof}
\noindent Recall the fact that the Haar measure on a compact abelian group is strictly positive. This generalizes to compact nilspaces.
\begin{proposition}\label{prop:posmeasopen}
Let $\ns$ be a $k$-step compact nilspace, with Haar measure $\mu$. Then for every open set $U\subset \ns$ we have $\mu(U)>0$.
\end{proposition}
\begin{proof}
We argue by induction on $k$. The case $k=1$ follows from the fact recalled above. Let $k>1$ and let $U$ be open in $\ns$. By Remark \ref{rem:open-and-closed} we have that $\pi_{k-1}$ is an open map, so $B=\pi_{k-1}(U)$ is open in $\cF_{k-1}(\ns)$. Hence, letting $\mu'$ denote the Haar measure on $\cF_{k-1}(\ns)$, we have by induction that $\mu'(B)>0$. By regularity of $\mu'$ there is a compact subset $C\subset B$ such that $\mu'(C)>0$, and then since $\pi_{k-1}$ preserves Haar measures we have $\mu(\pi_{k-1}^{-1}(C))>0$. Now by Remark \ref{rem:open-and-closed} we also have that $\pi_{k-1}$ is a proper map, so $\pi_{k-1}^{-1}(C)$ is a compact set. This set is covered by the union of open sets $\bigcup_{z\in \ab_k} U+z$, so there is a finite set of translates of $U$ by elements of $\ab_k$ that covers $\pi_{k-1}^{-1}(C)$, and it follows that $\mu(U)>0$.
\end{proof}

\subsection{Probability spaces of morphisms}\label{subsec:morph-prob-spaces}

Recall from \cite[Definition 3.1.3]{Cand:Notes1} that given a cubespace $P$ and a subcubespace $S$ of $P$, we say that $S$ has the \emph{extension property} in $P$ if for every non-empty nilspace $\ns$ and every morphism $g': S \to \ns$ there is a morphism $g: P\to \ns$ with $g|_S=g'$. 

In this subsection we collect several constructions of probability spaces of morphisms between nilspaces, which will be used repeatedly in the sequel. These constructions rely on two basic ideas.\\
\indent Firstly, every set of restricted morphisms $\hom_f(P,\ns)$ has a compact abelian-bundle structure, as described in Lemma \ref{lem:top-restrmorph=subbund}, which we restate here for convenience. 
\begin{lemma}\label{lem:comp-restrmorph=subbund}
Let $P$ be a subcubespace of $\{0,1\}^n$ with the extension property, let $S$ be a subcubespace of $P$ with the extension property in $P$, let $\ns$ be a $k$-step compact nilspace and let $f:S\to \ns$ be a morphism. Then $\hom_f(P,\ns)$ is a compact $k$-fold abelian bundle that is a sub-bundle of $\ns^P$, with factors $\hom_{\pi_i\co f}(P,\ns_i)$ and structure groups $\hom_{S\to 0}(P,\cD_i(\ab_i))$, where $\ab_i$ is the $i$-th structure group of $\ns$.
\end{lemma}
\noindent As a consequence of this structure we have a Haar measure on $\hom_f(P,\ns)$, by Proposition \ref{prop:nilspaceHaar}. 

The second idea consists in giving a simple description of the measure-theoretic properties of a restriction map from one set of such morphisms to another. In particular, we want a convenient criterion to decide whether, for a  subcubespace $S$ of $P$, the restriction map $\hom(P,\ns)\to  \hom(S,\ns)$ preserves the Haar measures. The following notion provides a general criterion of this type.
\begin{defn}\label{defn:GoodPair}
Let $P$ be a cubespace, and let $P_1,P_2$ be subcubespaces of $P$. We call the pair of sets $P_1,P_2$ a \emph{good pair} if $P_1$ and $P_1\cap P_2$ both have the extension property in $P$ and, for every abelian group $\ab$ and every $k\in \N$, every morphism $f':P_2\to \cD_k(\ab)$ satisfying $f'|_{P_1\cap P_2}=0$ extends to a morphism $f:P\to \cD_k(\ab)$ satisfying $f|_{P_1}=0$.
\end{defn}
\noindent Note that in particular if $S$ has the extension property in $P\subset \{0,1\}^n$ and $P_1=\emptyset$ then $P_1,S$ is a good pair in $P$. 
The main purpose of this definition is the following result.
\begin{lemma}\label{lem:GoodPairHoms}
Let $P\subset\{0,1\}^n$ be a subcubespace with the extension property, and let $P_1,P_2$ be a good pair in $P$. Let $\ns$ be a $k$-step nilspace and let   $f:P_1\to \ns$ be a morphism. Then the restriction map
\begin{equation}\label{eq:restricthom}
\phi:\hom_f(P,\ns)\to \hom_{f|_{P_1\cap P_2}} (P_2,\ns)
\end{equation}
is a totally-surjective bundle morphism. In particular, if $\ns$ is a $k$-step compact nilspace, then $\phi$ preserves the Haar measures. Moreover, the Haar measures on the fibres $\phi^{-1}(t)$ form a CSM on $\phi$ that disintegrates the Haar measure on $\hom_f(P,\ns)$ relative to the Haar measure on $\hom_{f|_{P_1\cap P_2}} (P_2,\ns)$.
\end{lemma}
\begin{proof}
We prove that $\phi$ is a totally-surjective bundle morphism by induction on $k$. For $k=0$ the claim holds trivially,  so let $k\geq 1$ and suppose that $\ns$ is a $k$-step nilspace and that the lemma holds for $(k-1)$-step nilspaces. To see that condition (i) from \cite[Definition 3.3.1]{Cand:Notes1} holds, note that by a straightforward calculation we have that $\phi$ induces a map $\phi_{k-1}:\hom_{\pi_{k-1}\co f}(P,\ns_{k-1})\to \hom_{\pi_{k-1}\co f|_{P_1\cap P_2}} (P_2,\ns_{k-1})$ well-defined by $\pi_{k-1}\co g\mapsto \pi_{k-1}\co \phi(g)$, so the condition holds for $i=k-1$. We also have that $\phi_{k-1}$ is precisely the restriction map on $\hom_{\pi_{k-1}\co f}(P,\ns_{k-1})$, so by induction the condition holds for all $i\leq k-1$. Condition (ii) from the same definition is clearly satisfied, with the structure morphism $\alpha_i$ being the restriction $\hom_{P_1\to 0}(P,\cD_i(\ab_i))\to \hom_{P_1\cap P_2\to 0}(P_2,\cD_i(\ab_i))$. To check that $\phi$ is totally surjective, by induction it suffices to check that $\alpha_k$ is surjective, which holds by Definition \ref{defn:GoodPair}. Lemma \ref{lem:MeasPres} then implies that $\phi$ preserves the Haar measures. The last sentence of the lemma follows from Lemma \ref{lem:quotint}.
\end{proof}
Another useful feature of good pairs is the following extension result for morphisms.
\begin{lemma}\label{lem:Gpair-union-ext}
Let $P$ be a subcubespace of $\{0,1\}^n$ with the extension property, and let $P_1,P_2$ be a good pair in $P$. Then $P_1\cup P_2$ equipped with the union of the cube structures on $P_1,P_2$ has the following extension property: any morphism $f:P_1\cup P_2\to \ns$ into a non-empty $k$-step nilspace $\ns$ extends to a morphism $f':P\to \ns$.
\end{lemma}
\begin{proof}
We argue again by induction on $k$. The statement is trivial for $k=0$, as $\ns$ is then a 1-point space, so suppose that the lemma holds for $(k-1)$-step nilspaces and let $\ns$ be a $k$-step nilspace.
Let $f:P_1\cup P_2\to \ns$ be a morphism and let $f_2:P\to\cF_{k-1}(\ns)$ be an extension of $\pi_{k-1}\co f$. By \cite[Lemma 3.2.11]{Cand:Notes1} there is a morphism $f_2':P\to \ns$ such that $\pi_{k-1}\co f_2'=f_2$. Let $g=f'_2|_{P_1}-f|_{P_1}$. By Definition \ref{defn:GoodPair}, there is an extension $f_3:P\to\cD_k(\ab_k)$ of $g$. Let $g_2=f_2'-f_3$. We have that $g_2|_{P_1}=f|_{P_1}$. Now let $g_3$ be an extension of $g_2|_{P_2}-f|_{P_2}$ to $P$ with $g_3|_{P_1}=0$. Then $f'=g_2-g_3$ is an extension of $f$ to $P$.
\end{proof}
\noindent For the remainder of this section we apply the results above to define various  probability spaces of morphisms. We begin with the cube sets $\cu^n(\ns)$ on a compact nilspace. Note that $\cu^n(\ns)=\hom(\{0,1\}^n,\ns)$ is a compact $k$-fold abelian bundle and therefore has a Haar measure (by Proposition \ref{prop:nilspaceHaar}).

\begin{lemma}\label{lem:surjmorphcubemeas}
Let $\ns,\nss$ be $k$-step compact nilspaces and let $n\in \N$. Then for every continuous fibre-surjective morphism $\psi:\ns \to \nss$, the induced map $\cu^n(\ns)\to \cu^n(\nss)$, $\q\mapsto \psi\co\q$ preserves the Haar measures. Moreover, the fibres of this map can all be equipped with the Haar measure.
\end{lemma}
\begin{proof}
By \cite[Lemma 3.3.12 (ii)]{Cand:Notes1} we have that $\q\mapsto \psi\co\q$ is a totally surjective bundle morphism, so it preserves the Haar measures by Lemma \ref{lem:MeasPres}. Moreover, by \cite[Lemma 3.3.12 (iii)]{Cand:Notes1}, the fibres of this map are compact $k$-fold abelian bundles, so the last claim in the lemma follows from Proposition \ref{prop:nilspaceHaar}.
\end{proof}
\noindent We shall often want to use a disintegration of the Haar measure on $\cu^n(\ns)$ into a CSM on the projection $\q\mapsto \q(0^n)$. For this purpose we record the following special case of Lemma \ref{lem:GoodPairHoms}.
\begin{lemma}\label{lem:cube-set-CSM}
Let $\ns$ be a $k$-step compact nilspace with Haar measure $\mu$, let $\pi:\cu^n(\ns)\to \ns$, $\q\mapsto \q(0^n)$, and let $\mu_n$ denote the Haar measure on $\cu^n(\ns)$. Then for every $x\in \ns$, the space $\pi^{-1}(x)=\cu^n_x(\ns)$ is a compact $k$-fold abelian bundle and therefore has a Haar measure $\mu_x$. The family $\{\mu_x:x\in \ns\}$ is a CSM on $\pi$  disintegrating $\mu_n$ relative to $\mu$.
\end{lemma}
\begin{proof}
By Lemma \ref{lem:GoodPairHoms} with $P=\{0,1\}^n$, with $P_1=\emptyset$, $P_2=\{0^n\}$, identifying $\ns$ with $\hom(P_2,\ns)$.
\end{proof}
\noindent We move on to some results concerning the tricube $T_n$. Recall from  \cite[Subsection 3.1.3]{Cand:Notes1} the definitions of tricubes, of morphisms $\trem_v$, and of the outer-point map $\omega_n:\{0,1\}^n\to T_n$.

\begin{lemma}\label{lem:ope-to-tricube-ext}
The set $\omega_n(\{0,1\}^n)\subset T_n$ has the extension property in $T_n$.
\end{lemma}
\begin{proof}
Let $\ns$ be a nilspace, let $P=\omega_n(\{0,1\}^n)$ and let $f:P\to \ns$ be a morphism. Let $h:\{-1,0,1\}\to \{0,1\}$ be the map $-1\mapsto 0$, $0\mapsto 1$, $1\mapsto 1$. Then the map $f'=f\co h^n$ is an extension of $f$ to $T_n$.
\end{proof}

\begin{lemma}\label{lem:subcube-in-Tn}
For every $v\in \{0,1\}^n$, the subcube $\Psi_v(\{0,1\}^n)$ of $T_n$ has the extension property in $T_n$.
\end{lemma}
\begin{proof}
The claim is checked for $v=0^n$ using the map $h:\{-1,0,1\}\to \{0,1\}$, $-1\mapsto 1$, $0\mapsto 0$, $1\mapsto 1$. (Indeed then $f\co h^n$ extends $f$ to $T_n$.) The lemma then follows by the composition axiom.
\end{proof}

\begin{corollary}\label{cor:tricubeprobspaces}
Let $P\subset T_n$ be $\Psi_v(\{0,1\}^n)$ or $\omega_n(\{0,1\}^n)$, let $\ns$ be a $k$-step compact nilspace, and let $f:P \to \ns$ be a morphism. Then $\hom_f(T_n,\ns)$ has a Haar measure.
\end{corollary}
\begin{proof}
This follows again from Lemma \ref{lem:comp-restrmorph=subbund} and Proposition \ref{prop:nilspaceHaar}.
\end{proof}
\begin{lemma}\label{lem:Gpair-in-Tn}
The sets $P_1=\omega_n(\{0,1\}^n)=\{-1,1\}^n$ and $P_2=\{0,1\}^n$ form a good pair in $T_n$.
\end{lemma}
\begin{proof}
Note first that we may assume that $\omega_n\in \cu^n(T_n)$, since by \cite[Lemma 3.1.16]{Cand:Notes1} this assumption does not affect any set of morphisms from $T_n$ into any nilspace. Now since $P_1\cap P_2=\{1^n\}$ is a singleton, it has the extension property. The set $P_1$ itself has the extension property, by Lemma \ref{lem:ope-to-tricube-ext}.
Let $f:\{0,1\}^n \to \cD_k(\ab)$ be a morphism with $f(1^n)=0$. Let $h(-1)=1$, $h(0)=0$, $h(1)=1$. Then $f'=f \circ h^n$ is an extension of $f$ with $f|_{P_1}=0$.
\end{proof}
\noindent Combining this with Lemma \ref{lem:GoodPairHoms}, we obtain that a set of morphisms $T_n\to \ns$ with a constraint at the outer points of $T_n$ is still large enough to cover a whole set of rooted cubes $\cu^n_x(\ns)$, in the following sense.
\begin{corollary}\label{cor:ext-in-Tn}
Let $\ns$ be a $k$-step compact nilspace, let $P_1$ and $P_2$ be as in Lemma \ref{lem:Gpair-in-Tn}, and let $f:P_1\to \ns$ be a morphism. Then the restriction map $\hom_f(T_n,\ns)\,\to\, \hom_{0^n\to f(1^n)}(P_2,\ns)=\cu^n_{f(1^n)}(\ns)$, $\;t\mapsto t\co \trem_{0^n}$ preserves the Haar measures.
\end{corollary}
\noindent As a complement to this corollary, we also have that a set of morphisms $T_n\to \ns$ with a given constraint at just one outer point is always large enough to cover a full cube set $\cu^n(\ns)$, as follows. 
\begin{lemma}\label{lem:key-tricube-mp}
Let $\ns$ be a $k$-step compact nilspace, let $x\in \ns$, and let $Q_x$ denote the compact abelian bundle $\hom_{1^n\mapsto x}(T_n,\ns)$. Then for every $v\in\{0,1\}^n\setminus \{0^n\}$, the map $Q_x\to \cu^n(\ns)$, $t\mapsto t\circ\Psi_v$ preserves the Haar measures.
\end{lemma}
\begin{proof}
There is a Haar measure on $Q_x$ by Lemma \ref{lem:comp-restrmorph=subbund}. Let $P_1=\{1^n\}$ and $P_2=\Psi_v(\{0,1\}^n)$. The lemma will follow if we show that $P_1,P_2$ form a good pair in $T_n$, by Lemma \ref{lem:subcube-in-Tn}. Since the singleton $P_1$ has the extension property in $T_n$ and so does $P_1\cap P_2=\emptyset$, it suffices to show that any given morphism $f':P_2\to \cD_k(\ab)$ can be extended to a morphism $f:T_n\to\cD_k(\ab)$ satisfying $f|_{P_1}=0$. We can certainly extend $f'$ to a morphism $f_2:T_n\to\cD_k(\ab)$, by Lemma \ref{lem:subcube-in-Tn}. Now $v\neq 0^n$ implies that $\trem_v(0^n)\neq 1^n$, so this has some coordinate equal to $-1$, say without loss of generality $\trem_v(0^n)\sbr{1}=-1$. Then the set $F=\{1\}\times\{-1,0,1\}^{n-1}$ is disjoint from $P_2$. Let $g:T_n\to\cD_k(\ab)$ equal $f_2(1^n)$ on $F$ and $0$ otherwise. A straightforward calculation shows that $g$ is a morphism, using \cite[(2.9)]{Cand:Notes1}. Then the function $f=f_2-g$ is also a morphism $T_n\to\cD_k(\ab)$, it extends $f'$,  and satisfies $f|_{P_1}=0$.
\end{proof}

\section{Topological spaces associated with continuous systems of measures}\label{subsec:CSM-bundle}

\medskip
\begin{defn}
Let $X,Y$ be compact spaces and let $\mu$ be a Borel measure on $X$. We denote by $L(X,Y)$ the quotient of the set of Borel measurable functions $f:X\to Y$ by the equivalence relation $\sim$ defined by $f\sim g\;$ $\Leftrightarrow$ $\;\mu\big(\{x\in X: f(x)\neq g(x)\}\big)=0$.
\end{defn}

\noindent In this section, given a continuous system of measures, we construct an associated topological space that will be crucial in the sequel. In order to motivate the construction, let us outline how it will be used.

The central application will be given in the next section, which is one of the main sections of this chapter. The aim there is to show that given a compact nilspace $\ns$ and a compact abelian group $\ab$, given a cocycle $\rho:\cu^k(\ns)\to \ab$ (recall \cite[Definition 3.3.14]{Cand:Notes1}), if $\rho$ is Borel measurable then the extension $M(\rho)$ constructed in \cite[Proposition 3.3.26]{Cand:Notes1} can be equipped with a compact nilspace structure compatible with that of $\ns$. The purpose of the construction in this section is to provide an ambient continuous abelian bundle from which $M(\rho)$ will inherit the desired topological structure. The construction will be completed with Proposition \ref{prop:CSM-dif-bund}.\\
\indent Let us now turn to the definition of the topological space in the construction. To motivate this, consider the following general theme of the application mentioned above: given a Borel measurable function on a space (e.g. the cocycle $\rho$), we have to build a topological space that is associated with the function in a useful way. A very basic example of such a construction is the following: let $\ab$ be a compact abelian group, suppose that we only know the Borel $\sigma$-algebra on $\ab$ (not the underlying topology), and that we are given a Borel function $f:\ab \to \R$. Then we can build a topology on $\ab$ naturally associated with $f$ using the $L^1$ seminorm on the Borel functions on $\ab$ (relative to the Haar measure), by restricting this seminorm to the translates of $f$, i.e. the functions $f_z:x\mapsto f(x+z)$, $z\in \ab$. More precisely, we can define a pseudometric $d$ on $\ab$ by $d(z,z')=\|f_z-f_{z'}\|_{L^1(\ab)}$, and then take the topology generated by the open balls with respect to $d$. The central application mentioned above will use a more elaborate version of this idea. Indeed, recall from \cite[(3.11)]{Cand:Notes1} that the extension generated by the cocycle $\rho$ is $M(\rho)=\bigcup_{x\in \ns} \{\rho_x+z:z\in \ab\}$, where for each $x$ we denote by $\rho_x$ the restriction of $\rho$ to the set of rooted cubes $\cu^k_x(\ns)=\{\q\in\cu^k(\ns):\q(0^k)=x\}$. To put a useful topology on $M(\rho)$, we shall first put an analogue of the $L^1$-topology on each family of shifts $\rho_x +z,\,z\in \ab$, for each $x\in \ns$, and we shall then want to tie together adequately these different topologies over different points $x$ into a global topology. Now the Haar measures on sets $\cu^k_x(\ns)$ form a CSM, by Lemma \ref{lem:cube-set-CSM}. This leads to the following definition, which concerns CSMs more generally, and which provides a topology that will be adequate for our purposes.
\begin{defn}\label{def:extspace}
Let $V,W,Z$ be compact spaces, and let $\{\mu_w: w\in W\}$ be a family of strictly positive Borel probability measures forming a CSM on a continuous map $\pi:V\to W$. For each function $f$ in $\bigcup_{w\in W} L\big(\pi^{-1}(w),Z\big)$, let $\tilde\pi(f)$ be the element $w\in W$ such that $f\in L(\pi^{-1}(w),Z)$. We define the topological space $\cL(V,Z)$ to be the set $\bigcup_{w\in W} L\big(\pi^{-1}(w),Z\big)$ equipped with the coarsest topology making the following functions continuous:
\begin{equation}\label{eq:initopfns}
\varphi_{F_1,F_2}:f\mapsto \int_{\pi^{-1}(\tilde{\pi}(f))}F_1(f(v))\;F_2(v)~\ud\mu_{\tilde{\pi}(f)}(v),
\end{equation}
for every pair of continuous functions $F_1:Z\to\C$ and $F_2:V\to\C$. 
\end{defn}
\noindent We shall discuss below how this topology relates to the $L^1$ topology, but for now let us record some of its basic properties.
\begin{proposition}\label{prop:CSM-top}
The topological space $\cL(V,Z)$ is regular Hausdorff and second-countable,  and the map $\tilde\pi$ is continuous.
\end{proposition}
In particular $\cL(V,Z)$ is metrizable, by Urysohn's theorem \cite[Theorem 34.1]{Munkres}.
\begin{proof}
We show first that $\tilde\pi$ is continuous. Fix a metric on $W$, and let $U$ be any open subset of $W$. Then for any $w\in U$ there is an open ball $B_r(w)\subset U$ of radius $r>0$. By Urysohn's lemma \cite[Theorem 33.1]{Munkres} there exists a continuous function $f_w$ with value 1 on the closure of $B_{r/2}(w)$ and value 0 outside $B_r(w)$. Let $B_w'$ denote the open set $\{x\in W: f_w(x)>0 \}$, and note that $U=\bigcup_{w\in U} B_w'$. By definition of $\cL(V,Z)$, for every $w\in U$ the function $\varphi_{1,f_w\co\pi}$ is continuous. Since $\varphi_{1,f_w\co\pi}(f)=f_w(\tilde\pi(f))$, this function is positive precisely on $\tilde\pi^{-1}(B_w')$, so this set is open. Therefore $\tilde\pi^{-1}(U)=\bigcup_{w\in U} \tilde\pi^{-1}(B_w')$ is an open set.\\
\indent Next we show that $\cL(V,Z)$ is Hausdorff. Let $f,f'$ be distinct points in $\cL(V,Z)$. We distinguish two cases. In the first case $f,f'$ lie in different fibres of $\tilde\pi$. Then since $W$ is Hausdorff there exist disjoint open sets $U,U'$ containing $\tilde\pi(f)$, $\tilde\pi(f')$ respectively. The preimages of these sets under $\tilde\pi$ are open sets separating $f,f'$. In the second case we have $f,f'$ in the same fibre $L(\pi^{-1}(w),Z)$. Then there are Borel functions $f_0,f_0'$ in the classes $f,f'$ respectively such that the set $D=\{v:f_0(v)\neq f_0'(v)\}$ has $\mu_w(D)>0$. For some $\epsilon<\mu_w(D)$ to be fixed later, by Lusin's theorem \cite[Appendix E, \S E8]{Rudin} there exist continuous functions $g,g'$ and a set $C\subset \pi^{-1}(w)$ of probability at most $\epsilon$ such that  $f_0=g$ and $f_0'=g'$ outside $C$. Then there is a point $v_0\in D\setminus C$ where $g(v_0)\neq g'(v_0)$. By continuity there exist $r_1,r_2,r_3>0$ such that    for all $v\in B_{r_1}(v_0)\subset V$ we have $g(v)\in B_{r_2}(g(v_0))\subset Z$, $g'(v)\in B_{r_3}(g'(v_0))$, and $B_{r_3}(g'(v_0))\cap B_{2r_2}(g(v_0))=\emptyset$. Urysohn's lemma gives us continuous functions $F_1:Z\to [0,1]$, $F_2:V\to [0,1]$, such that $1_{B_{r_2}(g(v_0))}\leq F_1\leq 1_{B_{2r_2}(g(v_0))}$ and $1_{B_{r_1/2}(v_0)}\leq F_2 \leq 1_{B_{r_1}(v_0)}$. We then have $\varphi_{F_1,F_2}(g)\geq \mu_w(B_{r_1/2}(v_0))>0$, whereas $\varphi_{F_1,F_2}(g')=0$. We also have $|\varphi_{F_1,F_2}(g)-\varphi_{F_1,F_2}(f_0)|\leq 2 \mu_w(C)$ and similarly for $g',f_0'$. Hence, choosing $\epsilon$ small enough, we deduce that $\varphi_{F_1,F_2}(f_0)>\varphi_{F_1,F_2}(f_0')$.\\
\indent To see that $\cL(V,Z)$ is regular, first note that since by definition the finite intersections of sets $\varphi_{F_1,F_2}^{-1}(U)$ form a base for the topology, we have that for every closed set $S$ and $f\not\in S$, there exist $\varphi_{F_{1,i},F_{2,i}}^{-1}(U_i)$, $i\in [N]$ such that $f\in \bigcap_i \varphi_{F_{1,i},F_{2,i}}^{-1}(U_i)$ and $S\subset \bigcup_i \varphi_{F_{1,i},F_{2,i}}^{-1}(\C\setminus U_i)$. Since $\C$ is regular, for each $i$ there exist disjoint open sets $U_i',U_i''$ such that $\varphi_{F_{1,i},F_{2,i}}(f)\in U_i'\subset U_i$ and $\C\setminus U_i \subset U_i''$. Then the open sets $\bigcap_i \varphi_{F_{1,i},F_{2,i}}^{-1}(U_i')$ and $\bigcup_i \varphi_{F_{1,i},F_{2,i}}^{-1}(U_i'')$ are disjoint, the former contains $f$, and the latter includes $S$.\\
\indent Finally, we show that $\cL(V,Z)$ is second-countable. It suffices to find sequences of continuous functions $F_{1,i}:Z\to \C$, $F_{2,j}:V\to \C$ and open sets $U_k\subset \C$, $i,j,k\in \N$, such that for every $\varphi_{F_1,F_2}$, every open set $U\subset \C$, and every $f\in \cL(V,Z)$  with $\varphi_{F_1,F_2}(f)\in U$, there exist $i,j,k$ such that $f\in \varphi_{F_{1,i},F_{2,j}}^{-1}(U_k)\subset \varphi_{F_1,F_2}^{-1}(U)$. Indeed, if this holds then the finite intersections of sets $\varphi_{F_{1,i},F_{2,j}}^{-1}(U_k)$ form a countable base for the topology on $\cL(V,Z)$. 
Let $\{U_k:k\in \N\}$ be a base of open discs in $\C$. The spaces of continuous functions $C(V,\C)$ and  $C(Z,\C)$, equipped with the uniform metric, are separable \cite[Theorem (4.19)]{Ke}. Let $(F_{1,i})_{i\in \N}$ and $(F_{2,j})_{j\in \N}$ be separating sequences for $C(V,\C)$ and  $C(Z,\C)$ respectively. Given any $\varphi_{F_1,F_2}$, open set $U\subset \C$, and $f\in \varphi_{F_1,F_2}^{-1}(U)$, for some $\epsilon>0$ there is a disc $U_k\subset U$ of radius at most $\epsilon$ with center at most $\epsilon/2$ away from $\varphi_{F_1,F_2}(f)$ and such that the distance from $U_k$ to the complement of $U$ is at least $\epsilon$. There exist also $i,j$ such that $\|F_1- F_{1,i}\|_\infty<\epsilon/(4(1+\|F_2\|_\infty))$ and $\|F_2- F_{2,j}\|_\infty< \epsilon/(4(1+\|F_{1,i}\|_\infty))$. Then for every $g\in \cL(V,Z)$, letting $w=\tilde\pi(g)$, we have
\[
|\varphi_{F_1,F_2}(g)-\varphi_{F_{1,i},F_{2,j}}(g)|  \leq \int_{\pi^{-1}(w)} |F_1(g(v)) F_2(v)- F_{1,i}(g(v)) F_{2,j}(v)|\;\ud\mu_w(v).
\]
By the triangle inequality, this is at most $\|F_2\|_\infty \, \|F_1- F_{1,i}\|_\infty + \|F_{1,i}\|_\infty \, \|F_2- F_{2,j}\|_\infty < \epsilon/2$. 
This implies that $f\in \varphi_{F_{1,i},F_{2,j}}^{-1}(U_k)\subset \varphi_{F_1,F_2}^{-1}(U)$.
\end{proof}

\begin{remark}
There are spaces $\cL(V,Z)$ that do not admit any complete metric. For example, let $V$ be the one-point compactification of the half-strip $[0,\infty)\times [0,1] \subset \R^2$, let $W$ be the one-point compactification of $[0,\infty)$, let $\pi:V\to W$ be the projection to the first coordinate, let $Z=\{0,1\}$, and for each $w\in W\setminus \{\infty\}$ let the interval $\pi^{-1}(w)$ be equipped with the Lebesgue probability measure. For each $n\in \N$ let $f_n$ be the function on $\pi^{-1}(\{n\})$ that takes value $0$ on $\{n\}\times [0,1/2)$ and $1$ on $\{n\}\times [1/2,1]$. The sequence $(f_n)$ is Cauchy for any compatible metric on $\cL(V,Z)$ but it does not converge to a function $\{\infty\}\to Z$. To see this, we can view $\cL(V,Z)$ as a subspace of a larger  space $D(V,Z)=\bigcup_{w\in W} L\big(\pi^{-1}(w),P(Z)\big)$, where each fibre  $L\big(\pi^{-1}(w),P(Z)\big)$ consists of probability-valued functions $\tilde f$ on $\pi^{-1}(w)$, i.e. functions mapping each $v\in \pi^{-1}(w)$ to a probability measure $p$ on $Z$. (We can identify $L\big(\pi^{-1}(w),P(Z)\big)$ with $L\big(\pi^{-1}(w),[0,1]\big)$ by identifying $\tilde f$ with the map $v\mapsto p_v(\{0\})$.) The topology on $D(V,Z)$ is the initial topology generated by the functions $\tilde f\mapsto  \int_{\pi^{-1}(w)} \E\big(F_1\,|\,\tilde f(v)\big)\;F_2(v)~\ud\mu_w(v)$. This restricts to the original topology on $\cL(V,Z)$. A straightforward calculation shows that $f_n$ converges in $D(V,Z)$ to the function $\tilde f$ mapping $\infty$ to the uniform distribution on $\{0,1\}$. (I am grateful to Szegedy for providing this example.)
\end{remark}

\noindent The topology on $\cL(V,Z)$ can be viewed as a generalization of the $L^1$ topology, as illustrated by the following result.
\begin{lemma}\label{lem:Ltop-restrict}
Let $V,W$ be compact spaces, let $Z$ be the closed unit disc in $\C$, and let  $w\in W$. Then a sequence  $(f_n)$ in $L(\pi^{-1}(w),Z)$ converges to $f$ in $\cL(V,Z)$ if and only if it converges to $f$ in $L^1(\pi^{-1}(w),Z)$.
\end{lemma}
\noindent This is a special case of a more general result that we shall use in the sequel, namely Lemma \ref{lem:Ltop-gen-restrict} below. This result  involves the following generalization of the $L^1$ distance for functions taking values in a compact metric space.
\begin{defn}
Given a compact space $Y$ with a Borel probability measure, and a compact space $Z$ with a compatible metric $d$, we define the metric $d_1$ on $L(Y,Z)$ by
\begin{equation}\label{eq:gen-L1}
d_1(f_1,f_2)=\E_{v\in Y}\; d\big(f_1(v),f_2(v)\big).
\end{equation} 
\end{defn}
\begin{lemma}\label{lem:Ltop-gen-restrict}
Let $V,W,Z$ be compact spaces, let $\{\mu_w:w\in W\}$ be a CSM on $\pi:V\to W$,  let $d$ be a compatible complete metric on $Z$, and let $w\in W$. Then a sequence $(f_n)$ in $L(\pi^{-1}(w),Z)$ converges to $f$ in  $\cL(V,Z)$ if and only if $d_1(f,f_n)\to 0$ as $n\to \infty$.
\end{lemma}
\begin{proof}
Let $\varphi_{F_1,F_2}$ be any of the functions generating the topology on $\cL(V,Z)$. We then have
\[
|\varphi_{F_1,F_2}(f_n)- \varphi_{F_1,F_2}(f)|\leq  \|F_2\|_\infty \int_{\pi^{-1}(w)} |F_1(f_n(v))-F_1(f(v)) | \ud\mu_w(v).
\]
We claim that if $d_1(f,f_n)\to 0$ then the last integral tends to 0 as $n\to\infty$. To see this fix any $\epsilon>0$ and note that by uniform continuity of $F_1$ there is $\delta_0>0$ such that if $d(f_n(v),f(v))<\delta_0$ then $|F_1(f_n(v))-F_1(f(v))|<\epsilon/2$. Fix a positive $\delta<\min(\delta_0,\epsilon/4 \|F_1\|_\infty )$. For $n$ sufficiently large we have $d_1(f_n,f)<\delta^2$, whence, by Markov's inequality, the set $D_\delta=\{v\in \pi^{-1}(w):d(f_n(v),f(v))\geq \delta\}$ has $\mu_w(D_\delta)\leq \delta$. It follows that
\[
\int_{\pi^{-1}(w)} |F_1(f_n(v))-F_1(f(v)) | \,\ud\mu_w\;  \leq\;  \epsilon/2+ 2 \|F_1\|_\infty\, \mu_w(D_\delta) \;\leq\; \epsilon.
\]
The claim follows, and so $d_1(f_n,f)\to 0$ implies that $f_n\to f$ in $\cL(V,Z)$.\\ 
\indent For the converse, suppose that $f_n\to f$ in 
$\mathcal{L}(V,Z)$. By \cite[Theorem (4.17)]{Ke} and \cite[Theorem 1]{Ander}, there exists a compact subset $B$ of the real Hilbert space $\ell^2$ such that there is a homeomorphism $F:Z\to B$. We have that $F\co f_n$, $F\co f$ are Bochner measurable functions in $L^2(V,\ell^2)$. We claim that if $F\co f_n\to F\co f$ in $L^2(V,\ell^2)$ (that is $\int_{\pi^{-1}(w)}\|F\co f_n(v) - F\co f(v)\|_{\ell^2}\ud\mu_w \to 0$) then $d_1(f_n,f)\to 0$ as required. Indeed, for any $\epsilon>0$, the set $E=\{(z,z'):d(z,z')\geq \epsilon\}$ is closed in $Z\times Z$ so compact, and therefore the continuous function $\frac{d(z,z')}{\|F(z) - F(z')\|_{\ell^2}}$ on $E$ reaches a maximum $C$, whence for $n$ sufficiently large we have
\begin{eqnarray*}
d_1(f_n,f) & = & \int_{\pi^{-1}(w)} d\big(f_n(v),f(v)\big)\ud\mu_w \;\;\leq\;\; \epsilon+C \int_{\pi^{-1}(w)}1_E\big(f_n(v),f(v)\big)\,\|F\co f_n(v) - F\co f(v)\|_{\ell^2}\ud\mu_w\\
& \leq & 2 \epsilon.
\end{eqnarray*}
This proves our claim. We now show that $F\co f_n\to F\co f$ in $L^2(V,\ell^2)$.\\
\indent Firstly, for every $F_0\in L^2(V,\ell^2)$, we show that as $n\to \infty$ we have $\int_{\pi^{-1}(w)} \langle F\co f_n(v),F_0(v)\rangle_{\ell^2}\ud\mu_w\to \int_{\pi^{-1}(w)} \langle F\co f(v),F_0(v)\rangle_{\ell^2}\ud\mu_w$. By definition of the Bochner integral, it suffices to prove this for simple functions $F_0$. By linearity, it then  suffices to prove it for $F_0$ of the form $\alpha\,1_A$ where $A\subset \pi^{-1}(w)$ is measurable and $\alpha\in \ell^2(\R)$. But in this case the function $F_1:Z\to \R$, $z\mapsto \langle F(z),\alpha\rangle$ is continuous, and the function $1_A:\pi^{-1}(w)\to \R$ can be approximated in $L^1(\mu_w)$ by continuous functions, hence in this case the convergence follows from the definition of the topology on $\cL(V,Z)$.\\
\indent Secondly, we show that $\int_{\pi^{-1}(w)} \| F\co f_n(v) \|_{\ell^2}^2\ud\mu_w$ converges to $\int_{\pi^{-1}(w)} \| F\co f(v) \|_{\ell^2}^2\ud\mu_w$. This follows from the fact that $F_1:Z\to \R$, $z\to \| F(z) \|_{\ell^2}^2$ is continuous, so that letting $F_2$ be the function with value 1 on all of $V$ we have $\int_{\pi^{-1}(w)} \| F\co f_n(v) \|_{\ell^2}^2\ud\mu_w=\varphi_{F_1,F_2}(f_n)$, which converges to $\varphi_{F_1,F_2}(f)$ as required.\\
\indent From the last two paragraphs we deduce the desired convergence by a standard argument, namely
\begin{eqnarray*}
&& \int_{\pi^{-1}(w)}\|F\co f_n(v) - F\co f(v)\|_{\ell^2}^2\,\ud\mu_w \; = \; \int_{\pi^{-1}(w)} \langle F\co f_n - F\co f,F\co f_n - F\co f\rangle\, \ud\mu_w\\
& = & \int \|F\co f_n\|_{\ell^2}^2 -2 \langle F\co f_n ,F\co f\rangle +\|F\co f\|_{\ell^2}^2\,\ud\mu_w 
\; \to \;  \int 2\|F\co f\|_{\ell^2}^2 -2 \langle F\co f ,F\co f\rangle\, \ud\mu_w\;= \;0.\qedhere
\end{eqnarray*}
\end{proof}
\noindent It is a basic fact that if $(f_n)$ is a sequence of real-valued functions on a standard probability space such that $f_n\to f$ in $L^1$ and each $f_n$ is constant almost surely, then the limit $f$ is also constant almost surely. We shall use the following analogous fact for $\cL(V,Z)$.
\begin{lemma}\label{lem:constconv}
Let $V,W,Z$ be compact spaces, let $\{\mu_w:w\in W\}$ be a CSM on $\pi:V\to W$, and let $(f_n)$ be a convergent sequence in $\cL(V,Z)$ such that each $f_n$ is constant almost surely on its fibre $\pi^{-1}(w_n)$. Then the limit is also constant almost surely.
\end{lemma}
\begin{proof}
Let $f$ be the limit of the sequence $(f_n)$. By Urysohn's lemma on $Z$, it suffices to show that for every continuous function $F:Z\to\R$ the function $F\co f$ is constant. Now since $F$ and $F^2$ are both continuous, by definition of the topology on $\cL(V,Z)$ we have $\int_{\pi^{-1}(w_n)} F\co f_n \ud\mu_{w_n} \to \int_{\pi^{-1}(w)} F\co f \ud\mu_w$ and $\int_{\pi^{-1}(w_n)} F^2\co f_n \ud\mu_{w_n} \to \int_{\pi^{-1}(w)} F^2\co f \ud\mu_w$. On the other hand, since each $f_n$ is constant almost everywhere, we have $\Big(\int_{\pi^{-1}(w_n)} F\co f_n \ud\mu_{w_n}\Big)^2 - \int_{\pi^{-1}(w_n)} F^2\co f_n \ud\mu_{w_n} =0$. Taking limits as $n\to\infty$, we deduce that $\Big(\int_{\pi^{-1}(w)} F\co f \ud\mu_w\Big)^2 - \int_{\pi^{-1}(w)} F^2\co f \ud\mu_w =0$. By the equality case of the Cauchy-Schwarz inequality, this implies that $F\co f$ is constant almost everywhere.
\end{proof}

The following result will be used several times in the next section.
\begin{lemma}\label{lem:prodCSMconv}
Let $V,W,Z$ be compact spaces, and let $\{\mu_w : w\in W\}$ be a CSM on $\pi:V\to W$. Suppose that $(f_n)$ and $(g_n)$ are sequences in $\cL(V,Z)$ converging to $f$ and $g$ respectively, where $f_n,g_n$ are defined on $\pi^{-1}(w_n)$ and $f,g$ are defined on $\pi^{-1}(w)$. Then the functions $v\mapsto (f_n(v),g_n(v))$ converge to $v\mapsto (f(v),g(v))$ in $\cL(V,Z\times Z)$.
\end{lemma}
\begin{proof}
It suffices to show that for any continuous functions $F_1:Z\times Z\to\C$ and $F_2:V\to\C$ we have $\lim_{n\to\infty}\int_{\pi^{-1}(w_n)}F_1(f_n(v),g_n(v))\, F_2(v)\ud\mu_{w_n}(v)=\int_{\pi^{-1}(w)}F_1(f(v),g(v))\,F_2(v)\ud\mu_w(v)$. 
By the Stone-Weierstrass theorem the function $F_1$ can be approximated arbitrarily closely in the uniform metric by finite linear combinations of  functions of the form  $(z_1,z_2)\mapsto F_{1,1}(z_1)F_{1,2}(z_2)$ where $F_{1,1},F_{1,2}:Z\to\C$ are continuous, so we may assume that $F_1$ is of this form. Fix any $\epsilon>0$. By Lusin's theorem there exists a continuous function $q':\pi^{-1}(w)\to \C$ such that $\|F_{1,2}\co g-q'\|_{L^2}<\epsilon$ and $\|q'\|_\infty\leq \|F_{1,2}\|_\infty$, and then by the Tietze extension theorem \cite[Theorem 35.1]{Munkres} there exists a continuous function $q:V\to \C$ such that the restriction $q|_{\pi^{-1}(w)}$ equals $q'$ and $\|q\|_\infty\leq \|F_{1,2}\|_\infty$. For each $n$ let $q_n=q|_{\pi^{-1}(w_n)}$. Now  
\begin{eqnarray}\label{eq:L2-arg}
\|F_{1,2}\co g_n - q_n\|_{L^2}^2 & = & \int_{\pi^{-1}(w_n)}|F_{1,2}\co g_n - q_n|^2\ud\mu_{w_n} \nonumber \\
& = & \int_{\pi^{-1}(w_n)}|F_{1,2}\co g_n|^2 - 2\tRe\,(F_{1,2}\co g_n\,\overline{ q_n})+|q_n|^2\ud\mu_{w_n}.
\end{eqnarray}
Using that $g_n\to g$ in $\cL(V,Z)$, that $w_n\to w$, and the uniform continuity of $q$, we deduce that the integral above converges to $\|F_{1,2}\co g -q' \|_{L^2}^2$. Setting $h_n=(F_{1,1}\co f_n) (F_{1,2}\co g_n-q_n)$, we then have
\[
F_1\co(f_n,g_n)=(F_{1,1}\co f_n) \big(q_n+(F_{1,2}\co g_n-q_n)\big)=(F_{1,1}\co f_n)\,q_n+h_n
\]
and similarly $F_1\co(f,g)=(F_{1,1}\co f)\, q'+h$, where for $n$ sufficiently large we have $\|h_n\|_{L^2}$ and $\|h\|_{L^2}$ both at most $2\epsilon \|F_{1,1}\|_\infty$. The convergence $f_n\to f$ also implies that
\[
\int_{\pi^{-1}(w_n)} (F_{1,1}\co f_n)\,q_n\, F_2\, \ud\mu_{w_n}\;\to\; \int_{\pi^{-1}(w)}(F_{1,1}\co f)\, q'\, F_2\, \ud\mu_w.
\]
Combining the above estimates we deduce that for all $n$ sufficiently large we have
\[
\Big|\int_{\pi^{-1}(w_n)}F_1\co (f_n,g_n)\, F_2\,\ud\mu_{w_n}- \int_{\pi^{-1}(w)}F_1\co(f,g)\,F_2\ud\mu_w\Big| \leq 5 \epsilon \|F_{1,1}\|_\infty\|F_2\|_\infty.
\]
Since $\epsilon> 0$ was arbitrary, the result follows.
\end{proof}
\noindent For the main applications of the space $\cL(V,Z)$ in the sequel, we will suppose that $Z$ is a compact abelian group. We now show that in this case the group has a continuous action on the space.
\begin{lemma}\label{lem:contact} Let $V,W$ be compact spaces, let $\ab$ be a compact abelian group, and let $\{\mu_w: w\in W\}$ be a CSM on $\pi:V\to W$. Then we have the following free continuous action of $\ab$ on $\cL(V,\ab)$:
\[
\alpha\,:\;\ab\times\cL(V,\ab)\;\to\;\cL(V,\ab),\quad  \alpha(z, f): v\mapsto f(v) + z.
\]
\end{lemma}

\begin{proof}
The topology on $\cL(V,\ab)$ is generated by the functions $\varphi_{F_1,F_2}$, so to check continuity of $\alpha$ it suffices to show that for every $\varphi_{F_1,F_2}$ the map $\varphi_{F_1,F_2}\circ \alpha$ is continuous. Thus, it suffices to show that for any continuous functions $F_1 : \ab \to \C$ and $F_2 : V \to \C$, the following  function is continuous:
\begin{equation}\label{eq:contact}
q: \;(z, f) \;\mapsto\; \int_{\pi^{-1}(\tilde{\pi}(f))} F_1(f(v)+z)\; F_2(v) \; \ud\mu_{\tilde{\pi}(f)}(v).
\end{equation}
Applying the Stone-Weierstrass theorem as in previous proofs, the continuous function $\ab \times \ab \to \C$ defined by $(z_1,z_2) \mapsto F_1(z_1 + z_2)$ can be approximated uniformly by linear combinations of functions of the form $(z_1, z_2) \mapsto h_1(z_1)\,h_2(z_2)$ where $h_1$ and $h_2$ are continuous. Substituting any of these into \eqref{eq:contact}, we obtain a function of the form
\[
q':\;(z, f) \;\mapsto \;h_2(z) \int_{\pi^{-1}(\tilde{\pi}(f))} h_1(f(v))\; F_2(v) \; \ud\mu_{\tilde{\pi}(f)}(v) \;=\;h_2(z)\, \varphi_{h_1,F_2}(f).
\]
This is continuous by definition of $\cL(V,\ab)$. Hence $q$ is a uniform limit of continuous functions and is therefore continuous.
\end{proof}
\noindent In our uses of $\cL(V,\ab)$, we shall want to view this space as a continuous $\ab$-bundle, for  some appropriate base $S$ and projection $\cE$. We would like  $S,\cE$ to have the following features. Firstly, the projection should send two functions $f_1,f_2$ in $\cL(V,\ab)$ to the same point of $S$ if and only if $f_1$ and $f_2$ are in the same $\ab$-orbit (i.e. $f_2=f_1+z$ for some $z\in \ab$). Secondly, we would like $S$ to be also of the form $\cL(Y,\ab)$, for some other similar space $Y$ having a projection to $W$, as this would help to relate the topologies on $S$ and $\cL(V,\ab)$.\\
\indent A natural way to obtain these features is to define $\cE$ as a  difference map, sending $f$ to the function $\cE(f):V\times V\to \ab$, $(v_0,v_1)\mapsto f(v_0)-f(v_1)$. We are thus led to the following construction.
\begin{defn}
Let $V,W$ be compact spaces and let $\{\mu_w:w\in W\}$ be a CSM on $\pi:V\to W$. We define the compact space
\[
V\times_W V = \{(v_0,v_1): v_0,v_1\in V,\;\pi(v_0)=\pi(v_1)\},
\]
and define the projection $\pi':V\times_W V\to W$ by $\pi'(v_0,v_1)=\pi(v_0)$.
\end{defn}
\begin{lemma}
The family of measures $\{\mu_w\times\mu_w:w\in W\}$ on $V\times_W V$ is a  CSM on $\pi'$.
\end{lemma}
\begin{proof}
For every $w\in W$ the measure $\mu_w$ is concentrated on $\pi^{-1}(w)$, so in $V\times V$ the measure $\mu_w\times \mu_w$ is concentrated on $\pi^{-1}(w) \times \pi^{-1}(w)={\pi'}^{-1}(w)$. For any continuous function $f:V\times_W V\to \C$, the continuity of $w\mapsto \int_{{\pi'}^{-1}(w)} f \ud (\mu_w\times \mu_w)$ is proved by an argument similar to previous ones (e.g. in the proof of Lemma \ref{lem:contact}), showing that this function is a uniform limit of continuous functions $F_n$, where $F_n$ is obtained by approximating $f$ to within $1/n$ in the supremum norm by a finite linear combination of functions of the form $(v_0,v_1)\mapsto h_0(v_0)\,h_1(v_1)$ with $h_i:V\to \C$ continuous, and then using the continuity of $w\mapsto \prod_{i=0,1}\int_{\pi^{-1}(w)} h_i  \ud \mu_w$ guaranteed by the CSM on $\pi$.
\end{proof}
\noindent We are now able to complete the construction of the continuous abelian bundle announced at the beginning of this section, by showing that the map $\cE$ has the required properties to be a suitable projection.

\begin{proposition}\label{prop:CSM-dif-bund}
Let $V,W$ be compact spaces, let $\{\mu_w :w\in W\}$ be a CSM on $\pi:V\to W$, and let $\ab$ be a compact abelian group. Let $\cE:\cL(V,\ab)\to\cL(V\times_W V,\ab)$ be defined by $\cE(g): (v_0,v_1) \mapsto g(v_0)-g(v_1)$. Then $\cL(V,\ab)$ is a continuous $\ab$-bundle over $\cE(\cL(V,\ab))$ with projection $\cE$.
\end{proposition}

\begin{proof} Recall that the topology on $\cL(V\times_W V,\ab)$ is generated by functions $\varphi_{F_1,F_2}$ sending $f:V\times_W V\to \ab$ with $\tilde\pi(f)=w$ to $\int_{{\pi'}^{-1}(w)} F_1(f(v_0,v_1))\,F_2(v_0,v_1) \ud\mu_w(v_0,v_1)$. To see that $\cE$ is continuous, fix any such function and suppose that $g_n\to g$ in $\cL(V,\ab)$. Then, approximating the continuous functions $(z_0,z_1)\mapsto F_1(z_0-z_1)$ on $\ab^2$ and $F_2$ on ${\pi'}^{-1}(w)=\pi^{-1}(w)\times \pi^{-1}(w)$ by combinations of continuous functions $h_0\,h_1$ as in previous proofs, we approximate $\varphi_{F_1,F_2}(\cE(g_n))$ by finite combinations of products $\varphi_0(g_n)\varphi_1(g_n)$, where $\varphi_0,\varphi_1$ are functions among those generating the $\cL(V,\ab)$ topology. By assumption, each of these products satisfies $\varphi_0(g_n)\varphi_1(g_n)\to \varphi_0(g)\varphi_1(g)$, and it follows that $\varphi_{F_1,F_2}(\cE(g_n))\to \varphi_{F_1,F_2}(\cE(g))$. This proves the continuity of $\cE$.\\
\indent To ensure that we have a continuous $\ab$-bundle, it only remains to check the `if' part of condition (iv) in Definition \ref{def:CpctAbBund}. Equivalently, it now suffices to prove that if $\cE^{-1}(U)$ is closed then $U$ is closed. For that, we claim it is enough to show that if for a sequence of functions $f_n\in\cL(V,\ab)$ the images $\cE(f_n)$ converge, then there is a convergent subsequence  $(f_m)$ of $(f_n)$. Indeed, suppose that this holds and suppose for a contradiction that $U$ is not closed. Then (using the surjectivity of $\cE$)  there is a sequence of functions $\cE(f_n)\in U$ converging to a function $\cE(f)\not\in U$. But then letting $f'$ be the limit of the guaranteed subsequence $(f_m)$, the closure of $\cE^{-1}(U)$ implies that $f'\in \cE^{-1}(U)$, and so by continuity of $\cE$ we have $U\ni \cE(f')=\lim_{m\to \infty} \cE(f_m) =\cE(f) \not\in U$, a contradiction. This proves our claim.\\
\indent So let us suppose that $\cE(f_n)$ converges to $\cE(f)$ for some $f\in\cL(V,\ab)$. Let $w_n=\tilde\pi(f_n)$ and $w=\tilde\pi(f)$, and note that $w_n\to w$ in $W$ by continuity of $\tilde{\pi}$.\\
\indent Let $\cC$ denote the circle group, viewed as the unit circle in $\C$. We shall prove the following fact:
\begin{eqnarray}\label{eq:cE-keyfact}
&& \textrm{There exists a subsequence }(f_m)\textrm{ and some }t\in \ab\textrm{ such that, for every character }\\
&& \chi:\ab\to \cC\textrm{ in }\widehat{\ab},\textrm{ we have }\chi\co f_m\to \chi\co\,(f+t)\textrm{ in }\cL(V,\cC)\textrm{ as }m\to \infty. \nonumber 
\end{eqnarray}
We prove this below, but before that let us show how it implies that $f_m$ converges to $f+t$ in $\cL(V,\ab)$.\\
\indent Let $\varphi_{F_1,F_2}$ be any of the functions generating the $\cL(V,\ab)$ topology, and note that by the Stone-Weierstrass theorem for every $\epsilon>0$ there exists a trigonometric polynomial $F_1'=\sum_{j\in [K]} \lambda_j \chi_j$ on $\ab$ such that $\|F_1-F_1'\|_\infty \leq\epsilon$. Therefore it suffices to show that $\varphi_{F_1',F_2}(f_m)\to \varphi_{F_1',F_2}(f+t)$ for any such polynomial $F_1'$. We have
\[
\varphi_{F_1',F_2}(f_m)=\sum_{j\in [K]} \lambda_j \int_{\pi^{-1}(w_m)} (\chi_j\co f_m)\, F_2 \;\ud\mu_{w_m}=\sum_{j\in [K]} \lambda_j\, \varphi_{1,F_2}(\chi_j\co f_m).
\]
Since $\varphi_{1,F_2}$ is one of the functions generating the topology on $\cL(V,\cC)$, we have $\varphi_{1,F_2}(\chi_j\co f_m)\to \varphi_{1,F_2}\big(\chi_j\co\,(f+t)\big)$ by assumption, and the desired convergence follows.\\
\indent We now prove \eqref{eq:cE-keyfact}. Given $\chi\in\widehat{\ab}$, the function $\chi\circ f$ is in $L^2(\mu_w)$. Continuous functions are dense in this space, and any such function can be extended to a continuous function on $V$ (by the Tietze extension theorem). Hence, for every $\delta>0$ there exists a continuous function $\chi_\delta:V\to\cC$ whose restriction to $\pi^{-1}(w)$ satisfies $\|\chi_\delta-\chi\co f\|_{L^2(\mu_w)}\leq \delta$. We claim that for every $\chi\in\widehat{\ab}$ and $\epsilon>0$, there exist  $\delta>0$, $N\in\N$, and $\lambda\in \cC$, such that for all $n>N$ we have $\|(\chi\co f_n)\,\overline{\chi_\delta}-\lambda\|_{L^2(\mu_{w_n})} \leq \epsilon$. 

To prove this claim, let $\cE':\cL(V,\cC)\to\cL(V\times_W V,\cC)$ be defined by $\cE'(g):(v_0,v_1)\mapsto g(v_0)\overline{g(v_1)}$. We have that  $\cE'$ is continuous, by an argument similar to the one used above for $\cE$. It is also clear that for any $g_1,g_2\in \cL(V,\cC)$ lying in the same fibre of $\tilde\pi$ we have $\cE'(g_1g_2)=\cE'(g_1)\cE'(g_2)$, and that $\cE'(\chi\circ g)=\chi\circ\cE(g)$ for every $\chi\in\widehat{\ab}$, $g\in \cL(V,\ab)$. Now, from the definition of $\chi_\delta$, the continuity of $\cE'$, and Lemma \ref{lem:Ltop-gen-restrict}, it follows that for $\delta$ sufficiently small we have $\big\| \cE'\big((\chi\circ f)\,\overline{\chi_\delta}\big)-1\big\|_{L^2(\mu_w\times\mu_w)}\leq \epsilon/2$. On the other hand, we have $\cE'\big((\chi\circ f_n)\overline{\chi_\delta}\big)=\cE'(\chi\circ f_n)\,\cE'(\overline{\chi_\delta})=\chi\co \cE(f_n)\,\cE'(\overline{\chi_\delta})$ and similarly for $f$. This together with the assumption $\cE(f_n)\to\cE(f)$ implies, by an argument similar to \eqref{eq:L2-arg}, that for $n$ sufficiently large we have $\big\| \cE'\big((\chi\circ f_n)\,\overline{\chi_\delta}\big)-1\big\|_{L^2(\mu_{w_n}\times \mu_{w_n})}\leq \epsilon$. Taking a minimum over $v_1$ in this integral, we deduce that there exists $v_1\in \pi^{-1}(w_n)$ such that, for $\lambda= (\chi\circ f_n)(v_1)\overline{\chi_\delta}(v_1)$, we have indeed $\|(\chi\circ f_n)\overline{\chi_\delta}-\lambda\|_{L^2(\mu_{w_n})}\leq \epsilon$.

Since there are at most countably many elements in $\widehat{\ab}$, it follows by a standard diagonalization argument (using the above claim with smaller and smaller $\epsilon$ for each $\chi$) that for some sequence $(m_n)_{n\in \N}$ of positive integers, for every $\chi\in\widehat{\ab}$ there is a constant $\lambda_\chi\in\cC$ such that $\chi\co f_{m_n}\to \lambda_\chi (\chi\co f)$ in $\cL(V,\cC)$ as $n\to\infty$. A straightforward calculation shows that the function $\chi\mapsto \lambda_\chi$ is a homomorphism, hence a character on $\widehat{\ab}$, which by duality must be of the form $\chi\mapsto \chi(t)$ for some $t\in \ab$. This confirms \eqref{eq:cE-keyfact} and completes the proof.
\end{proof}
We close this section with two technical lemmas that we shall use several times.
\begin{lemma}\label{lem:csmtech} Let $\{\mu_w:w\in W\}$ be a CSM on $\pi:V\to W$. Let $K$ be a compact space with a Borel measure $\nu$. Suppose that $f_1:V \to K$ is continuous and that, for every $w\in W$, the restriction of $f_1$ to  $\pi^{-1}(w)$ satisfies $\mu_w\co f_1^{-1}=\nu$. Let $Z$ be a compact space and let $f_2:K\to Z$ be a Borel function. Let $h: W\to\cL(V,Z)$ be the map $w\;\mapsto \;f_2\co f_1|_{\pi^{-1}(w)}$. Then $h$ is continuous. 
\end{lemma}

\begin{proof} By definition of the topology on $\cL(V,Z)$ it suffices to show that for every $\varphi_{F_1,F_2}$ the function $\varphi_{F_1,F_2}\circ h:W\to \C$ is continuous. So let $F_1:Z\to\mathbb{C}$ and $F_2:V \to \C$ be continuous and let $q$ denote $\varphi_{F_1,F_2}\circ h$, that is $q: \; w\;\mapsto\;\int_{\pi^{-1}(w)}\;F_1(f_2\circ f_1(v))\;F_2(v)~\ud\mu_w(v)$. To prove that $q$ is continuous we show that it is a uniform limit of continuous functions.

Fix an arbitrary $\epsilon>0$. By Lusin's theorem there exists a continuous function $F_3:K\to \C$ such that $\|F_1\co f_2-F_3\|_{L^1(\nu)}\leq\epsilon$. Letting $q': w\mapsto\int_{\pi^{-1}(w)}\;F_3(f_1(v))\;F_2(v)~\ud\mu_w(v)$, the triangle inequality and the assumption that $f_1$ is measure-preserving imply that $\sup_{w\in W}|q(w)-q'(w)|\leq \epsilon \|F_2\|_\infty$. The function $q'$ is continuous by condition \eqref{contprop} from Definition \ref{def:CSM}. The result follows.
\end{proof}

\begin{lemma}\label{lem:CSMlincombin}
Let $\{\mu_w:w\in W\}$ be a CSM on $\pi:V\to W$, let $\ab$ be a compact abelian group, let $P$ be a finite set, and for each $r\in P$ let $\lambda_r$ be an integer. Let $\Sigma_P:\cL(V,\ab^P)\to \cL(V,\ab)$, $f\mapsto \sum_{r\in P} \lambda_r\, \pi_r\co f$, where $\pi_r$ is the projection to the $r$-component on $\ab^P$. Then $\Sigma_P$ is continuous.
\end{lemma}
\begin{proof}
For any of the functions $\varphi_{F_1,F_2}$ generating the topology on $\cL(V,\ab)$, letting $w=\tilde\pi(f)$, we have
$\varphi_{F_1,F_2}(\Sigma_P(f))=\int_{\pi^{-1}(w)} F_1\co\Sigma_P(f)\, F_2\, \ud\mu_w$. But $F_1\co \Sigma_P(f)=F_1'\co f$, where $F_1':\ab^P\to \C$, $(z_r)_{r\in P}\mapsto F_1\big(\sum_r \lambda_r z_r\big)$ is continuous. Hence, by definition of $\cL(V,\ab^P)$, if $f_n\to f$ in this space then $\varphi_{F_1,F_2}( \Sigma_P(f_n))\to\varphi_{F_1,F_2}( \Sigma_P(f))$, and since this holds for every $\varphi_{F_1,F_2}$, we have $\Sigma_P(f_n)\to \Sigma_P(f)$ in $\cL(V,\ab)$.
\end{proof}

\medskip
\section{Borel measurable cocycles generate compact extensions}

\medskip
Let $\ns$ be a compact nilspace and let $\ab$ be a compact abelian group. The purpose of this section is to show that any measurable cocycle $\rho:\cu^k(\ns)\to\ab$ yields a compact nilspace that is a $\ab$-bundle over $\ns$.

Recall that for each $x\in \ns$ we denote by $\cu^k_x(\ns)$ the set of $k$-cubes $\q$ satisfying $\pi(\q):=\q(0^k)=x$. By Lemma \ref{lem:cube-set-CSM} each space $\cu^k_x(\ns)$ has a Haar measure, denoted $\mu_x$, and the family $\{\mu_x: x\in \ns\}$ is a CSM on $\pi$ of strictly positive measures.
\begin{defn}
We denote by $\cL_k(\ns,\ab)$ the space $\cL(V,\ab)$ with $V=\cu^n(\ns)$, $W=\ns$, and $\pi:\q\mapsto \q(0^k)$. 
\end{defn}
\noindent Thus $\cL_k(\ns,\ab)=\bigcup_{x\in \ns} L\big(\cu_x^k(\ns),\ab\big)$, its topology is second-countable regular Hausdorff, and we have a projection $\tilde\pi:\mathcal{L}_k(\ns,\ab)\to \ns$ defined by $\tilde\pi(f)=x$ for $f\in L(\cu_x^k(\ns),\ab)$. For such a function $f$, we denote by $f+\ab$ the set of functions $\{f+z: z\in \ab\}$.

The desired nilspace extending $\ns$ is found inside $\cL_k(\ns,\ab)$ by taking, for each $x\in \ns$, the orbit of the restricted cocycle $\rho_x=\rho|_{\cu_x^k(\ns)}$.

\begin{proposition}\label{prop:cocyclext}
Let $\rho:\cu^k(\ns)\to \ab$ be a Borel measurable cocycle and let  $M=\bigcup_{x\in \ns} \,(\rho_x+\ab)$. Then, as a subspace of $\cL_k(\ns,\ab)$ with the inherited $\ab$-action, the space $M$ is a compact $\ab$-bundle over $\ns$.
\end{proposition}

\begin{proof}
Let $Y=\cu^k(\ns)\times_{\ns} \cu^k(\ns)$, that is (recalling the notation from Proposition \ref{prop:CSM-dif-bund}),
\[
Y\,=\,\big\{ (\q_0,\q_1):\,\q_0,\q_1 \in \cu^k(\ns),\, \q_0(0^k)=\q_1(0^k)\big\}\,=\,\bigcup_{x\in \ns} \cu^k_x(\ns)\times \cu^k_x(\ns).
\]
Recall that $\cE:\cL(\cu^k(\ns),\ab)\to \cL(Y,\ab)$ is defined by $\cE(f)\;:\; (\q_0,\q_1)\mapsto f(\q_0)-f(\q_1)$. We now compose the function $x\mapsto \rho_x$ with $\cE$, obtaining the function
\begin{eqnarray}
g':\;\ns & \to & L\big(\,\cu^k_x(\ns)\times \cu^k_x(\ns)\,,\,\ab\,\big)\subset \cL(Y,\ab)\\
 x & \mapsto & \cE(\rho_x).\nonumber
\end{eqnarray} 
Note that $g'$ is injective, and that its inverse is the projection $\tilde\pi$ on $\cL(Y,\ab)$, a continuous map. We have
\begin{equation}\label{eq:MoverX}
M=\cE^{-1}(g'(\ns)).
\end{equation}
Indeed, if $f\in \cL(\cu^k(\ns),\ab)$ satisfies $\cE(f)=\cE(\rho_x)$, then for every $(\q_0,\q_1)\in \cu^k_x(\ns)^2\subset Y$ we have $f(\q_0)-\rho_x(\q_0)= f(\q_1)-\rho_x(\q_1)$, so $f-\rho_x$ is a constant $\ab$-valued function on $\cu^k_x(\ns)$.

Now, by Proposition \ref{prop:CSM-dif-bund} the preimage under $\cE$ of every point is a compact set (homeomorphic to $\ab$), and $\cE$ is also a closed map (using Remark \ref{rem:open-and-closed}). Hence $\cE$ is a proper map (see \cite[\S 10.2, Theorem 1]{Bourb1}). In particular, the preimage under $\cE$ of any compact set is compact (\cite[\S 10.2, Proposition 6]{Bourb1}). Therefore, if we show that $g'$ is continuous, then $g'(\ns)$ is compact and then $M$ must be compact, by \eqref{eq:MoverX}. (Note that $M$ is then also a continuous $\ab$-bundle over $\ns$, since by \eqref{eq:MoverX} it is such a bundle over $g'(\ns)$, and $g'$ is a homeomorphism $\ns\to g'(\ns)$.)

To prove the continuity of $g'$, the key idea is to express this function in terms of compositions of $\rho$ with certain tricube morphisms $T_k\to \ns$ (recall \cite[Definition 3.1.13]{Cand:Notes1}), and thereby to use these morphisms in a smoothening operation that turns the measurability of $\rho$ into the continuity of $g'$.\\
\indent Let $Q=\hom(T_k,\ns)$ and $Q_x=\hom_{1^k\to x}(T_k,\ns)$. For each $v\in \{0,1\}^k\setminus\{0^k\}$, let
\[
g_v:\ns\to\cL(Q,\ab),\;\; x\;\mapsto \;\big(g_v(x):t\mapsto \rho( t\co\Psi_v)\,\in L(Q_x,\ab)\big).
\]
We claim that $g_v$ is continuous. Indeed, by Lemma \ref{lem:key-tricube-mp} if $v\neq 0^k$ then $t\mapsto t\circ\Psi_v$ is measure-preserving from $Q_x$ to $\cu^k(\ns)$, so we can deduce the continuity of $g_v$ by applying Lemma \ref{lem:csmtech} (setting $V=Q$, $K=\cu^k(\ns)$, $W=\ns$, $f_2=\rho$, $h=g_v$).\\
\indent Now consider the following function:
\[
g\;:\; \ns \to \cL(Q,\ab),\quad x\;\mapsto\; \sum_{v\in\{0,1\}^k\setminus \{0^k\}}(-1)^{|v|}\,g_v.
\]
The continuity of each function $g_v$ implies that $g$ is continuous. Indeed, letting $P=\{0,1\}^k\setminus \{0^k\}$, by Lemma \ref{lem:prodCSMconv} the function $\ns \to \cL(Q,\ab^P)$, $x\,\mapsto\, \prod_{v\in P}g_v(x)$ is continuous, and then Lemma \ref{lem:CSMlincombin} gives the continuity of $g$. Now, by \cite[Lemma 3.3.31]{Cand:Notes1}, for each $x\in \ns$ we have that $g(x)$ is the function $t\mapsto \rho_x(t\circ\omega_k)-\rho_x(t\circ\Psi_{0^k})$ (defined on $Q_x$). In particular, given any $t\in Q_x$, we have that  $\q_0:=t\co\omega_k$ and $\q_1:=t\co\Psi_{0^k}$ are both cubes in $\cu^k_x(\ns)$ ($\q_0$ is a cube by \cite[Lemma 3.1.16]{Cand:Notes1}), and $g(x)(t)=\rho_x(\q_0)-\rho_x(\q_1) = g'(x)(\q_0,\q_1)$. Given this equality, we now claim that continuity of $g'$ follows from continuity of $g$. Indeed, we know that $g'$ is continuous if for any of the functions $\varphi_{F_1,F_2}$ generating the topology on $\cL(Y,\ab)$ we have that $\ns \to \C$, $x\mapsto \varphi_{F_1,F_2}(g'(x))$ is continuous. Letting $Y_x=\cu^k_x(\ns)\times \cu^k_x(\ns)$, we have 
\[
\varphi_{F_1,F_2}(g'(x)) = \int_{Y_x} F_1(g'(x)(\q_0,\q_1))\; F_2((\q_0,\q_1)) \;\ud\mu((\q_0,\q_1) ),
\]
where $F_1:\ab\to \C$ and $F_2:Y_x\to \C$ are both continuous.
Now let $\theta : Q_x \to Y_x $, $t\mapsto (t\co \omega_k, t\co\Psi_{0^k})$. If we show that this map is surjective and measure-preserving, then since $g'(x)\co\theta(t)=g(x)(t)$, we would conclude that
\[
\varphi_{F_1,F_2}(g'(x)) = \int_{Q_x} F_1(g(x)(t))\; F_2(\theta(t))\; \ud\mu(t)
\]
and continuity of $g'$ would then indeed follow from that of $g$ (via approximating the measurable function $F_2\co \theta$ by continuous functions as in previous proofs).\\
\indent The surjectivity of $\theta$ follows from combining Lemma  \ref{lem:Gpair-in-Tn} and \ref{lem:Gpair-union-ext}. (To apply the latter result recall that we are implicitly embedding $T_k$ as a subcubespace in $\{0,1\}^{2k}$, using \cite[Lemma 3.1.17]{Cand:Notes1}.)\\
\indent To see that $\theta$ is measure-preserving, first note that by Lemma \ref{lem:GoodPairHoms} there is a CSM on the map $Q_x\to \cu^k_x(\ns),\; t\mapsto t\circ \omega_k$ consisting of the Haar probabilities on the fibres, giving a disintegration of the Haar probability on $Q_x$. On each of these fibres, of the form $\hom_{\q\co \omega_k^{-1}}(T_k,\ns)$, $\q\in \cu^k_x(\ns)$, it follows from Corollary \ref{cor:ext-in-Tn} that the map $\hom_{\q\co \omega_k^{-1}}(T_k,\ns) \to \cu^k_x(\ns)$, $t\mapsto t\co\Psi_{0^k}$  preserves the Haar measures. Using this, a straightforward calculation shows that for product sets $A_1\times A_2\subset \cu^k_x(\ns)\times \cu^k_x(\ns)$ we have $\mu_{Q_x}\co\theta^{-1}(A_1\times A_2)=\mu_{\cu^k_x(\ns)\times \cu^k_x(\ns)}(A_1\times A_2)$, and it follows that $\theta$ is measure-preserving as claimed. 
\end{proof}
\noindent Let us now recall the definition of cubes on $M$ from \cite[Definition 3.3.25]{Cand:Notes1}. Note that, by definition of $M$, for every function $f:\{0,1\}^k\to M$ there is some function $a=a_f:\{0,1\}^k\to \ab$ such that $\rho_x(v)=f(v)+a(v)$, where $x=\tilde\pi(f(v))\in \ns$.
\begin{defn}\label{def:compextcubes}
We define $\cu^k(M)$ to be the set of functions $f:\{0,1\}^k\to M$ such that $\tilde\pi\co f\in \cu^k(\ns)$ and
\begin{equation}\label{eq:compextcubes}
\rho(\tilde\pi\co f) = \sigma_k(a).
\end{equation}
For $n\neq k$, a function $f:\{0,1\}^n\to M$ is declared to be in $\cu^n(M)$ if $\tilde\pi\co f\in \cu^n(\ns)$ and every $k$-dimensional face-restriction of $f$ is in $\cu^k(M)$.
\end{defn}

\indent We can now complete the main task of this section.

\begin{proposition}\label{prop:comp-extcompleted}
The space $M$ together with the cube sets $\cu^n(M)$ is a compact nilspace.
\end{proposition}

\begin{proof}
By \cite[Proposition 3.3.26]{Cand:Notes1} we have that $M$ is a nilspace, and by Proposition \ref{prop:cocyclext} we have that $M$ is a compact space. Arguing as in the proof of Lemma \ref{lem:compclosure}, we see that it suffices to show that $\cu^k(M)$ is a closed subset of $M^{\{0,1\}^k}$.  
We use the probability space $Q=\hom(T_k,\ns)$ again, but this time we let $\pi_o:Q\to \cu^k(\ns)$ be the map $t\mapsto t\co \omega_k$, and apply Lemma \ref{lem:GoodPairHoms} (with $P=T_k$ embedded in $\{0,1\}^{2k}$, $P_1=\emptyset$, and $P_2=\omega_k(\{0,1\}^k)$ also embedded in $\{0,1\}^{2k}$) to obtain a CSM on $\pi_o$ consisting of the Haar measures on the fibres $\pi_o^{-1}(\q)=\hom_{\q\co\omega_k^{-1}}(T_k,\ns)$. Let $\tilde\pi$ be the projection $M\to \ns$, $\rho_x+z\mapsto x$, and let
\[
R_k(M)=\{f:\{0,1\}^k\to M : \tilde\pi\co f \in \cu^k(\ns)\}.
\]
Since $\tilde\pi$ is continuous and $\cu^k(\ns)$ is closed, we have that $R_k(M)$ is closed in $M^{\{0,1\}^k}$. We define the following map:
\begin{eqnarray*}
\phi\;:\;\;R_k(M)& \quad\to\quad & L\big(\pi_o^{-1}(\tilde\pi\co f),\ab\big) \; \subset \; \mathcal{L}(Q,\ab)\\
f &\quad \mapsto\quad & \phi_f(t)=\sum_{v\in\{0,1\}^k}(-1)^{|v|}\,f(v)(t\co\trem_v).
\end{eqnarray*}
Arguing as in the proof of \cite[Lemma 3.3.32]{Cand:Notes1}, we see that $\phi_f$ is actually a constant function for every $f\in R_k(M)$, namely the constant $\rho(\tilde\pi\co f)-\sigma_k(a)\in \ab$. By \cite[(3.16)]{Cand:Notes1} we have that $\cu^k(M)=\phi^{-1}(0)$. Hence it now suffices to show that $\phi$ is continuous, as then $\cu^k(M)$  is closed as required.\\
\indent To see the continuity of $\phi$, we first express it as the composition of simpler functions. For each $v\in \{0,1\}^k$, let
\begin{eqnarray*}
\phi_v\;:\;\;R_k(M)& \quad\to\quad & L\big(\pi_o^{-1}(\tilde\pi\co f),\ab\big) \; \subset \; \mathcal{L}(Q,\ab)\\
f &\quad \mapsto\quad & \phi_{v,f}(t)= f(v)(t\co\trem_v).
\end{eqnarray*}
We claim that $\phi_v$ is continuous. To see this, we suppose that $f_n\to f$ in $R_k(M)$, which implies that the component $f_n(v)$ converges to $f(v)$ in $M$, and we want to show that then $\phi_{v,f_n}\to \phi_{v,f}$ in $\cL(Q,\ab)$. For this it suffices to show that $\varphi_{F_1',F_2'}(\phi_{v,f_n})\to \varphi_{F_1',F_2'}(\phi_{v,f})$ for any continuous functions $F_1':\ab\to \C$, $F_2':Q\to \C$, where
\[
\varphi_{F_1',F_2'}(\phi_{v,f})=\int_{\pi_o^{-1}(\tilde\pi\co f)} F_1'(\phi_{v,f}(t))\; F_2'(t) \ud\mu_{\tilde\pi \co f}(t)=\int_{\pi_o^{-1}(\tilde\pi\co f)} F_1'\big(f(v)(t\co\trem_v)\big)\; F_2'(t) \ud\mu_{\tilde\pi\co f}(t).
\]
We shall now show that, for one of the functions $\varphi_{F_1,F_2}$ generating the topology on $M$, we have $\varphi_{F_1',F_2'}(\phi_{v,f})=\varphi_{F_1,F_2}(f(v))$. Recall that for any $g\in M$ we have
\begin{equation}\label{eq:testfn1}
\varphi_{F_1,F_2}(g)=\int_{\pi^{-1}(\tilde\pi(g))} F_1(g(\q))\; F_2(\q) \ud\mu_{\tilde\pi(g)}(\q).
\end{equation}
Now for each $t\in Q$ let $Q_t$ denote the space of morphisms $t'\in Q$ such that $\trem_v\co t'=\trem_v\co t$, and let $\mu_{Q_t}$ denote the Haar measure on this space given by Corollary \ref{cor:tricubeprobspaces}. These measures form a CSM by Lemma \ref{lem:GoodPairHoms}. We then have
\begin{eqnarray*}
\varphi_{F_1',F_2'}(\phi_{v,f}) & = & \int_{\pi_o^{-1}(\tilde\pi\co f)}\;\int_{Q_t} F_1'\big(f(v)(t'\co\trem_v)\big)\; F_2'(t')\ud\mu_{Q_t}(t') \;\ud\mu_{\tilde\pi\co f}(t)\\
& = & \int_{\pi_o^{-1}(\tilde\pi\co f)}F_1'\big(f(v)(t\co\trem_v)\big)\; \Big(\int_{Q_t}F_2'(t')\ud\mu_{Q_t}(t')\Big)\ud\mu_{\tilde\pi\co f}(t)\\
& = & \int_{\pi_o^{-1}(\tilde\pi\co f)}F_1'\big(f(v)(t\co\trem_v)\big)\; F_2''(t\co \trem_v)\ud\mu_{\tilde\pi\co f}(t),
\end{eqnarray*}
where $F_2''(t\co \trem_v)=\int_{t'\in Q_t}F_2'(t')\ud\mu_{Q_t}(t')$ is a continuous function of $\q=t\co \trem_v\in \cu^k(\ns)$, by Definition \ref{def:CSM}. We can now use the fact that $t\mapsto t\co \trem_v$ is measure-preserving from $\pi_o^{-1}(\tilde\pi\co f)$ to $\pi^{-1}(\tilde\pi\co f(v))$, to conclude that for $f(v)\in M$ we have
\[
\varphi_{F_1',F_2'}(\phi_{v,f}) = \int_{\pi^{-1}(\tilde\pi\co f(v))}F_1'\big(f(v)(\q)\big)\; F_2''(\q)\ud\mu_{\tilde\pi\co f}(\q).
\]
This is of the form $\varphi_{F_1',F_2''}(f(v))$ as in \eqref{eq:testfn1}, and the continuity of $\phi_v$ follows.\\
\indent The continuity of each function $\phi_v$ implies that the following function is continuous:
\begin{eqnarray*}
\phi_1\;:\;\;R_k(M)& \quad\to\quad & L\big(\pi_o^{-1}(\tilde\pi\co f),\ab^{\{0,1\}^k}\big) \; \subset \; \mathcal{L}(Q,\ab^{\{0,1\}^k})\\
f &\quad \mapsto\quad & \phi_{1,f}(t)= (f(v)(t\co\trem_v))_{v\in \{0,1\}^k}.
\end{eqnarray*}
Indeed, if $f_n\to f$ in $R_k(M)$ then for each $v$ by continuity of $\phi_v$ we have $\phi_{v,f_n}\to \phi_{v,f}$ in $\cL(Q,\ab)$. This then implies that $\phi_{1,f_n}\to\phi_{1,f}$ in $\cL(Q,\ab^{\{0,1\}^k})$, by Lemma \ref{lem:prodCSMconv}.\\
\indent The proof can now be completed using Lemma \ref{lem:CSMlincombin}. Indeed, for $P=\{0,1\}^k$ and appropriate choices of coefficients $\lambda_r$ in that lemma, the original function $\phi: R_k(M)\to \cL(Q,\ab)$ is the composition $\Sigma_P\co \phi_1$. The lemma tells us that $\Sigma_P$ is continuous, and so $\phi$ is continuous.
\end{proof}

\noindent Recall from \cite[Lemma 3.3.28]{Cand:Notes1} that, up to isomorphisms of nilspaces, every extension of $\ns$ by $\ab$ is of the form  $M(\rho)$ above for some cocycle $\rho$. We now want an analogue of this result for compact nilspaces, which requires that the isomorphism be also a homeomorphism. To ensure this, recall first from \cite[Lemma 3.3.21]{Cand:Notes1} that given a degree-$k$ extension of a nilspace, every cross section $\cs$ for this extension generates a cocycle,  denoted  $\rho_{\cs}$. We now show that for compact nilspaces one can take $\rho_{\cs}$ to be Borel measurable.

\begin{lemma}\label{lem:meascross}
Let $\ns$ be a compact nilspace, and let $\nss$ be a compact nilspace that is a degree-$k$ extension of $\ns$ by a compact abelian group $\ab$. Then there is a Borel cross section $\cs$ for this extension and therefore a Borel cocycle $\rho_{\cs}:\cu^{k+1}(\ns)\to \ab$.
\end{lemma}

\begin{proof}
Let $\pi : \nss \to \ns$ be the projection of the extension. Let $P = \{(x,y) \in \ns \times \nss \,:\, \pi(y) = x \}$. A cross-section for $\pi$ is a subset of $P$ which happens to be the graph of a function $\ns \to \nss$. Let $\mathcal{P}(\nss)$ denote the set of Borel probability measures on $\nss$. By \cite[Corollary 18.7]{Ke}, a sufficient condition for a Borel cross section to exist is that for some Borel function $\mu : \ns \to \mathcal{P}(\nss)$ we have $\mu_x(P \cap (\{x\} \times \nss)) >0$. The measures $\mu_x$ from the CSM structure on $\pi$ satisfy this, and the function $x \to \mu_x$ is continuous, by Definition \ref{def:CSM}. There is therefore a Borel cross section $\cs:\ns\to \nss$. We can then define a Borel function $\rho:\cu^{k+1}(\nss)\to \ab$ by $\rho(\q')=\sigma_{k+1}(\cs\co\pi\co\q' -\q')$. Moreover, as is shown in the proof of \cite[Lemma 3.3.21]{Cand:Notes1}, the function $\rho$ is constant on each fibre of the map $\cu^{k+1}(\nss)\to \cu^{k+1}(\ns)$, $\q'\mapsto \pi\co\q'$, thus yielding the well-defined cocycle $\rho_{\cs}:\cu^{k+1}(\ns)\to \ab$. It can then be seen by general results that $\rho_{\cs}$ is also Borel measurable.  For instance, one can use that by \cite[Definition 3.3.13]{Cand:Notes1} the space $\cu^{k+1}(\nss)$ is a continuous abelian bundle over $\cu^{k+1}(\ns)$ with Polish structure group $\cu^{k+1}(\cD_k(\ab))$, so that by \cite[Theorem (12.16) and Corollary (15.2)]{Ke} there is a Borel cross section $\q\mapsto\q'$ for this bundle.
\end{proof}
\noindent We shall also need the following useful generalization of a classical  automatic-continuity result. The classical result in question is a theorem of Kleppner stating that a Borel measurable homomorphism between two locally compact groups must be continuous \cite[Theorem 1]{Klep}. We extend the case of compact abelian groups as follows.

\begin{theorem}\label{thm:Klepgen}
Let $\ns,\nss$ be compact nilspaces of finite step, and let $\phi:\ns\to \nss$ be a Borel measurable morphism. Then $\phi$ is continuous.
\end{theorem}
\begin{proof} Suppose that $\nss$ is a $(k-1)$-step nilspace, and consider the CSM on $\pi:\cu^k(\ns)\to \ns$, $\q\mapsto \q(0^k)$, given by Lemma \ref{lem:cube-set-CSM}. For each $v\in \{0,1\}^k\setminus \{0^k\}$ let $\pi_v:\cu^k(\ns)\to \ns$, $\q\mapsto \q(v)$. We have that $\pi_v$ is continuous and each of its restrictions to a space $\cu^k_x(\ns)$ is measure-preserving, by Lemma \ref{lem:GoodPairHoms} (applied with $P_1=\{0^k\}$, $P_2=\{v\}$ and $f:0^k\mapsto x$, identifying $\hom_{f|_{P_1\cap P_2}}(P_2,\ns)$ with $\ns$). Now for each $v$ let $f_v:\ns\to\cL(\cu^k(\ns),\nss)$ be the function sending $x$ to the restriction $\phi\co\pi_v|_{\cu_x^k(\ns)}$. Then $f_v$ is continuous, by Lemma \ref{lem:csmtech} (applied with $V=\cu^k(\ns)$, $W=K=\ns$, $f_1=\pi_v$ and $f_2=\phi$).\\
\indent Let $f:\ns\to\cL\big(\cu^k(\ns),\nss\big)$ be the function mapping $x\in \ns$ to the constant function with value $\phi(x)$ on $\cu_x^k(\ns)$. From the definition of $\cL\big(\cu^k(\ns),\nss\big)$ it follows that if $f$ is continuous then so is $\phi$.\\
\indent To see that $f$ is continuous, recall that $\cor^k(\nss)$ denotes the space of $k$-corners on $\nss$ and consider the following function:
\[
\xi : \ns \to \cL\big(\cu^k(\ns),\cor^k(\nss)\big),\;\;\; x\mapsto\;\; \Big( \xi_x:\q\in\cu^k_x(\ns)\; \mapsto\; \big(f_v(x)(\q)\big)_{v\neq 0^k}=\big(\phi\co\q(v)\big)_{v\neq 0^k}\Big).
\]
Combining the continuity of each function $f_v$ with Lemma \ref{lem:prodCSMconv}, we deduce  that $\xi$ is continuous. Now let $\cK$ denote the function $\cL\big(\cu^k(\ns),\cor^k(\nss)\big)\to \cL\big(\cu^k(\ns),\nss\big)$ that sends a function $g$ to $\comp\co g$, where $\comp:\cor^k(\nss)\to\nss$ is the unique completion of $k$-corners on $\nss$. By Lemma \ref{lem:contcomp} we know that $\comp$ is  continuous, and we claim that $\cK$ is therefore continuous. Indeed, suppose we have a sequence $g_n$ in $\cL\big(\cu^k(\ns),\cor^k(\nss)\big)$ converging to $g$, and let $F_1:\nss\to \C$ and $F_2:\cu^k(\ns)\to \C$ be continuous. Then
\[
\varphi_{F_1,F_2}(\cK(g_n))\;=\;\int_{\cu^k_{x_n}(\ns)} F_1(\comp \co g_n(\q))\;F_2(\q)\; \ud\mu(\q)\;=\;\varphi_{F_1\co\comp, F_2}(g_n),
\]
where $\varphi_{F_1\co\comp, F_2}$ is one of the functions generating the topology on $\cL(\cu^k(\ns),\cor^k(\nss))$. Hence by assumption this converges to $\varphi_{F_1\co\comp, F_2}(g)=\varphi_{F_1,F_2}(\cK(g))$, so $\cK$ is continuous.\\
\indent Finally, note that $f=\cK\co\, \xi$, so $f$ is continuous, which completes the proof.
\end{proof}

We are now able to prove the analogue of \cite[Lemma 3.3.28]{Cand:Notes1} for compact nilspaces.

\begin{lemma}\label{lem:compextisotoM} Let $\ns$ be a compact nilspace of finite step. Let $\nss$ be a degree-$k$ extension of $\ns$ by a compact abelian group $\ab$. Let $\cs:\ns\to \nss$ be a Borel cross-section and let $\rho=\rho_{\cs}$ be the associated Borel cocycle. Then $\nss$ is isomorphic as a compact nilspace to the extension $M(\rho)$ from Proposition \ref{prop:cocyclext}.
\end{lemma}

\begin{proof} Let $\pi:\nss\to \ns$ be the projection for the extension. The isomorphism is given by the following map:
\begin{equation}
\theta:\nss\to M,\;\; x\mapsto \rho_{\pi(x)}+(x-\cs\co\pi(x)).
\end{equation}
It is checked in the proof of \cite[Lemma 3.3.28]{Cand:Notes1} that $\theta$ is an isomorphism of nilspaces. Since $\nss$ and $M$ are both compact Hausdorff spaces, if we show that $\theta$ is continuous then it is a homeomorphism \cite[Theorem 26.6]{Munkres}. Since $\theta$ is Borel measurable, it is continuous by Theorem \ref{thm:Klepgen}.
\end{proof}
\noindent Recall from \cite[Subsection 3.3.3]{Cand:Notes1} that $H_k(\ns,\ab)$ denoted the quotient of the abelian group of degree-$k$ cocycles by the subgroup of coboundaries. In the category of compact nilspaces we restrict these groups to the Borel measurable cocycles. Combining the results from this section, we now deduce the following analogue of \cite[Corollary 3.3.29]{Cand:Notes1} for compact nilspaces.
\begin{corollary}\label{cor:compclassrep} Let $\Phi$ denote the map which sends each class $C\in H_k(\ns,\ab)$ to the isomorphism class of $M(\rho)$, for any choice of $\rho\in C$. Then $\Phi$ is a surjection from $H_k(\ns,\ab)$ to the set of isomorphism classes of compact degree-$k$ extensions of $\ns$ by $\ab$.
\end{corollary}
\noindent The proof is similar to that of \cite[Corollary 3.3.29]{Cand:Notes1}.

\medskip
\section{Compact abelian bundles of finite rank, and averaging}\label{sec:CFRdef}

\medskip
In this section we define a notion of rank for compact abelian bundles and nilspaces, and we give a first topological description of nilspaces of finite rank (Lemma \ref{lem:finrankloctriv}). We then define an averaging operation for certain functions on these spaces, which enables us to prove a useful rigidity result concerning cocycles (Lemma \ref{lem:smallco}).
\begin{defn}
Let $\bnd$ be a $k$-fold compact abelian bundle, with structure groups $\ab_1,\dots,\ab_k$. We define the \emph{rank} of $\bnd$ by 
\[
\rk(\bnd)=\sum_{i=1}^k \rk(\widehat{\ab_i})
\]
where $\widehat{\ab_i}$ is the Pontryagin dual of $\ab_i$ and $\rk(\widehat{\ab_i})$ is the minimal number of generators of $\widehat{\ab_i}$. If $\ns$ is a $k$-step compact nilspace then we define its rank $\rk(\ns)$ to be $\rk(\bnd)$ for the associated bundle $\bnd$ given by Proposition \ref{prop:topbundec}.
\end{defn}
\noindent When working with a compact $\ab$-bundle $\bnd$ over $S$, we shall often want to ensure that the bundle is locally trivial, meaning that for every $x\in \bnd$ there is an open set $U\subset S$ containing the projection $\pi(x)$, and a homeomorphism $\pi^{-1}(U)\to U\times \ab$ of the form $x\mapsto (\pi(x),\varphi(x))$ where $\varphi:\pi^{-1}(U)\to \ab$ is $\ab$-equivariant (that is $\varphi(x+z)=\varphi(x)+z$ for every $x\in \pi^{-1}(U)$, $z\in \ab$). A theorem of Gleason \cite[Theorem 3.3]{Gl} tells us that this local triviality holds provided that $\ab$ is a Lie group (see also \cite[Theorem 1.1]{A&D}). We record the special case of this result that we shall use.
\begin{proposition}[Consequence of Gleason's slice theorem]\label{prop:Gleason}
Let $\ab$ be a compact abelian Lie group, and let $\bnd$ be a continuous $\ab$-bundle. Then $\bnd$ is locally trivial.
\end{proposition}
\noindent Recalling that a compact abelian group is a Lie group if and only if it has finite rank, we deduce the following result.
\begin{lemma}\label{lem:finrankloctriv}
Let $\ns$ be a $k$-step compact nilspace of finite rank with structure groups $\ab_1,\dots,\ab_k$ and factors $\ns_0,\dots,\ns_k$. Then $\ns_i$ is a locally trivial $\ab_i$-bundle over $\ns_{i-1}$ for each $i\in [k]$.
\end{lemma}
Thus, topologically a $k$-step compact nilspace of finite rank is a finite dimensional manifold.
\begin{defn}\label{def:calgmetric}
Recall that every compact abelian Lie group $\ab$ is isomorphic to a direct product $F\times \T^n$ for some finite abelian group $F$ and $n\geq 0$. We can then define a natural metric $d_2$ on $\ab$ as follows. For two elements $x,y\in \T^n=\R^n/\Z^n$ we define $d_2(x,y)$ as the minimum of the Euclidean distances between preimages of $x$ and $y$ under the map $\R^n\to\R^n/\Z^n$. If $\ab$ is not connected, then for points $x,y$ in different connected components we declare that $d_2(x,y)=\infty$.
\end{defn}
\noindent We now turn to the definition of averaging for a function on a nilspace taking values in a compact abelian group $\ab$. Here we face the problem that in general the average of a $\ab$-valued function is not well-defined. This is clear already in the simple case of the circle group $\ab=\T$, where there is no suitable notion of multiplication by a real number. However, if we assume that the function takes all its values in a short interval, then we can lift this $\T$-valued function to a real-valued one, then take the usual average, and then project the obtained value back to $\T$. We capture this idea more generally as follows.

\begin{defn}\label{def:calgaver}
Let $\ab\cong F\times \T^n$ be a compact abelian Lie group. Let $f$ be a Borel measurable function from a probability space $X$ to $\ab$, and suppose there exists $z=(g,\theta)\in \ab$, ($\theta\in \T^n$) such that every value of $f$ lies in the ball $B_{1/4}(z)$ in the metric $d_2$. Then, letting $\pi$ denote the projection $\R^n\to \T^n$, for any fixed lift $\theta'\in \R^n$ there is a unique function $f':X\to \R^n$ such that $f(x)=(g,\pi\co f'(x))$ and $f'$ takes all its values in $B_{z'}(1/4)$ in the euclidean distance. We then define the average value $\E f$ to be $(g,\pi(\E f'))$.
\end{defn}
\noindent Note that $\E(f)$ does not depend on the choice of the lift of $z$. Let us say that a function $f:X\to \ab$ is $\delta$\emph{-concentrated} if there exists $z\in \ab$ such that $f(x)\in B_\delta(z)$ for all $x\in X$.
\begin{lemma}\label{lem:avadd}
Let $\ab$ be a compact abelian Lie group, and let $m\in \N$. Then for any $(1/4m)$-concentrated measurable functions $f_1,\dots,f_m$ from a probability space $X$ to $\ab$, the average of $f_1+\cdots+f_m$ is well defined and satisfies $\E(f_1+\cdots+f_m)=\E f_1+\cdots+\E f_m$.
\end{lemma}
\begin{proof}
The assumption implies that $f_1+\cdots+f_m$ is $1/4$-concentrated, so its average is well-defined. The additivity then follows in a straightforward way from the additivity of the usual average for real-valued functions.
\end{proof}
\noindent As a first application of this notion of averaging, we obtain the following rigidity result for cocycles. Stronger versions of this result will be very useful in later sections. Recall from \cite[Definition 3.3.18]{Cand:Notes1} that a function $\cu^k(\ns)\to\ab$ is called a coboundary if it is of the form $\q\mapsto \sigma_k(f\co \q)$ for some function $f:\ns\to \ab$. 
\begin{lemma}\label{lem:smallco} Let $\ns$ be a compact $\ell$-step nilspace and let $\ab$ be a compact abelian group of finite rank. There exists $\epsilon>0$ such that, for every Borel measurable cocycle $\rho:\cu^k(\ns)\to \ab$, if  $d_2(\rho(\q),0)\leq\epsilon$ for every $\q\in \cu^k(\ns)$ then there exists a Borel measurable function $g:\ns\to \ab$ such that $\rho$ is the  coboundary $\q\mapsto \sigma_k(g\co\q)$. 
\end{lemma}

\begin{proof}
For any $x\in\ns$, consider the space of cubes $\cu^k_x(\ns)$ with the Haar probability given by Lemma \ref{lem:cube-set-CSM}. By assumption $\rho$ is $\epsilon$-concentrated on $\cu^k_x(\ns)$, so we may define
\[
g:\ns\to \ab,\; x\mapsto \E_{\cu^k_x(\ns)}\,\rho(\q).
\]
Fix any $\q\in \cu^k(\ns)$ and let $\Omega=\hom_{\q\co\omega_k^{-1}}(T_k,\ns)$. For any $t\in \Omega$, by \cite[Lemma 3.3.31]{Cand:Notes1} we have $\rho(\q)=\sum_{v\in\{0,1\}^k} (-1)^{|v|} \rho(t\co\Psi_v)$. By Corollary \ref{cor:ext-in-Tn}, we have that for each $v$ the map $t\mapsto t\co\Psi_v$, $\Omega\to \cu^k_{\q(v)}(\ns)$ preserves the Haar measures (where the Haar measure on $\Omega$ is given by Corollary \ref{cor:tricubeprobspaces}). Hence, by averaging over $t$ in the last equation, we obtain $\rho(\q) = \sum_{v\in\{0,1\}^k} (-1)^{|v|} g(\q(v))=\sigma_k(g\co\q)$.
\end{proof}

\medskip
\section{Counting isomorphism classes of compact finite-rank nilspaces}

\medskip
In this section we use the parametrization of compact extensions by cocycles, given in Corollary \ref{cor:compclassrep}, to prove the following  result.

\begin{theorem}\label{thm:countably} There are countably many isomorphism classes of compact nilspaces of finite rank.
\end{theorem}
\noindent To prove this we shall use the following lemma, that strengthens Lemma \ref{lem:smallco} by allowing its premise to fail on a null set.

\begin{lemma} \label{lem:smallco2} Let $\ns$ be a compact $\ell$-step nilspace and $\ab$ be a compact abelian group of finite rank. Then there exists $\epsilon>0$ such that every Borel measurable cocycle $\rho:\cu^k(\ns)\to \ab$ satisfying $d_2(\rho(\q),0)\leq\epsilon$ for almost every $\q\in \cu^k(\ns)$ is a coboundary.
\end{lemma}

This relies on the following fact.

\begin{lemma}\label{lem:almostzero} Let $\ns$ be a compact $\ell$-step nilspace and let $\ab$ be a compact abelian group of finite rank. Let $\rho:\cu^k(\ns)\to \ab$ be a Borel measurable cocycle such that $\rho=0$ for almost every element in $\cu^k(\ns)$. Then $\rho$ is a coboundary.
\end{lemma}

\begin{proof} 
Let $S=\{x\in \ns : \rho_x=0\textrm{ almost surely on }\cu_x^k(\ns)\}$. It follows from the properties of the CSM on $\cu^k(\ns)\to \ns$, $\q\mapsto \q(0^k)$ (given by Lemma \ref{lem:cube-set-CSM}) that $S$ has probability $1$.  For every $x\in S$ the set $\rho_x+\ab\,\subset \cL_k(\ns,\ab)$ consists of functions that are constant almost everywhere. Since by Proposition \ref{prop:posmeasopen} every open set in $\ns$ has positive measure, we have that $S$ is dense in $\ns$. We claim that for every $x\in \ns$ the function $\rho_x$ equals a constant almost everywhere. Indeed, by density of $S$ there exists a sequence of points $x_n\in S$ converging to $x$. For each $n$ choose an element $\rho_{x_n}+a_n \in M(\rho)$ equal to a constant almost surely. By Proposition \ref{prop:cocyclext} the extension $M$ is compact and so there is a subsequence of such functions $\rho_{x_m}+a_m$ converging in $M$. By Lemma \ref{lem:constconv}, the limit is constant almost everywhere, and (as a point of $M$) is $\rho_x+a$ for some $a\in \ab$, which proves our claim. From this, we deduce that the function $g:\ns\to \ab$, $x\mapsto \E_{\q\in \cu^k_x(\ns)} \rho(\q)$ is well-defined and that for each $x\in \ns$ we have $\rho(\q)=g(x)$ for almost every $\q\in \cu^k_x(\ns)$. We now claim that for every $x\in \ns$, we have $\rho(\q)=\sigma_k(g\co \q)$ for almost every $\q\in \cu^k(\ns)$. This follows from an argument similar to the second paragraph in the proof of Lemma \ref{lem:smallco}. Thus, denoting by $f$ the coboundary $\q\mapsto \sigma_k(g\co \q)$, the cocycle $\rho'=\rho-f$ has the property that for every $x\in \ns$ we have $\rho'_x=0$ almost surely on $\cu_x^k(\ns)$. This implies that $\rho'=0$ everywhere. Indeed, for an arbitrary fixed $\q\in \cu^k(\ns)$, and any $t\in \hom_{\q\co\omega_k^{-1}} (T_k,\ns)$, by \cite[Lemma 3.3.31]{Cand:Notes1} we have $\rho'(\q)=\beta(t,\rho')$, and then, averaging both sides of this equation over $t$, the right side vanishes since $\rho_x'$ is 0 almost everywhere (using again that each map $t\mapsto t\circ \Psi_v$ is probability-preserving, by Corollary \ref{cor:ext-in-Tn}).
\end{proof}
\begin{proof}[Proof of Lemma \ref{lem:smallco2}] We argue as in the proof of Lemma \ref{lem:smallco}. Let $S$ be the set of elements $x\in \ns$ for which $\rho_x$ is almost surely within $\epsilon$ of $0$ in $d_2$. By the assumption, we have that $S$ has probability $1$. We define $g:\ns\to \ab$ for $x\in S$ by $g(x)=\E_{\cu^k_x(\ns)}\,\rho(\q)$, and for $x\in \ns\setminus S$ by setting $g(x)=0$. Letting $\rho_2(\q)=\sigma_k(g\co \q)$, the same averaging argument shows that $\rho_2=\rho$ almost surely. By Lemma \ref{lem:almostzero} the difference $\rho-\rho_2$ is a coboundary, whence $\rho$ is a coboundary.
\end{proof}
We can now prove the main result of this section.

\begin{proof}[Proof of Theorem \ref{thm:countably}]
We claim that for every $k\in \N$, for every compact $(k-1)$-step nilspace $\ns$ of finite rank and every compact abelian Lie group $\ab$, there are at most countably many non-isomorphic degree-$k$ extensions of $\ns$ by $\ab$. This implies the theorem by induction on $k$.\\ 
\indent To prove the claim, we first associate with each $\ab$-extension of $\ns$ a measurable cocycle $\rho_{\cs}$ generated by a piecewise-continuous cross section $\cs$ for the extension. We construct $\cs$ as follows.

Since $\ns$ is a finite-dimensional compact manifold, there exist disjoint open sets $U_1,U_2,\dots,U_r$, each homeomorphic to the open unit Euclidean ball in $\R^n$ for some $n\geq 0$, such that $U=\bigcup_{i\in [r]} U_i$ has measure $1$ in $\ns$. Let $M$ denote the given degree-$k$ extension of $\ns$ by $\ab$, with projection $\pi:M\to \ns$. Every $\ab$-bundle over a contractible space is trivial \cite[Ch. 4, Corollary 10.3]{Husem}. Therefore there exists a continuous cross section $\cs_i$ locally on each $U_i$, and combining these we obtain a piecewise continuous (hence Borel-measurable) cross section $\cs:\ns\to M$, with $\cs|_{U_i}=\cs_i$. Let $T\subseteq \cu^{k+1}(\ns)$ denote the set of cubes whose vertices are all in $U$. We can partition the cubes $\q\in T$ according to which set $U_i$ contains each value $\q(v)$, that is $T=\bigsqcup_{i\in [r]^{\{0,1\}^{k+1}}} T_i$, where $T_i\,=\,T\;\cap \prod_{v\in \{0,1\}^{k+1}} U_{i_v}$. The cocycle $\rho_{\cs}:\cu^{k+1}(\ns)\to \ab$ is defined by $\rho_{\cs}(\q)=\sigma_{k+1}(\cs\co \q-\q')$ for any $\q'\in\cu^{k+1}(M)$ with $\pi\co\q'=\q$ (recall \cite[Lemma 3.3.21]{Cand:Notes1}). Note that $\rho_{\cs}$ is continuous on each part $T_i$.\\
\indent Now let $\epsilon>0$ be as given by Lemma \ref{lem:smallco2}. Then for any other degree-$k$ extension $M'$, we can generate a cocycle $\rho_{\cs'}$ as above, and if $d_2\big(\rho_{\cs}(\q),\rho_{\cs'}(\q)\big)<\epsilon$ for almost every $\q$ then by the lemma $\rho_{\cs'}$ must equal $\rho_{\cs}$ plus a coboundary, which implies, by Corollary \ref{cor:compclassrep}, that $M'$ is isomorphic to $M$.\\
\indent Let $d_\infty$ denote the uniform metric on the set of continuous functions $\cu^{k+1}(\ns)\to\ab$, defined by $d_\infty(f_1,f_2)=\sup_{\q}  d_2\big(f_1(\q),f_2(\q)\big)$. For each $i$, the space of continuous functions $T_i\to \ab$ with $d_\infty$ is separable (essentially by separability of the space $C(Y,\ab)$ for any compact subset $Y$ of a Euclidean space \cite[Theorem (4.19)]{Ke}). Therefore there exists a countable $\frac{\epsilon}{2}$-net $(f_{i,j})_{j\in \N}$ of such continuous functions on each $T_i$. Then, by the previous paragraph, for each choice of functions $\{f_{i,j_i}: i\in [r]^{\{0,1\}^{k+1}}\}$ from these nets, there can only be at most one isomorphism class of extensions having an associated cocycle $\rho_{\cs}$ satisfying $d_\infty(\rho_{\cs}|_{T_i},f_{i,j_i})<\epsilon/2$ for each $i$.
\end{proof}

\medskip
\section{Compact nilspaces as inverse limits of finite-rank nilspaces}\label{sec:invlim}

\medskip
In this section we treat one of the central results from \cite{CamSzeg}, namely that every compact nilspace can be expressed as an inverse limit of compact nilspaces of finite rank. Before we give the formal statement, let us detail the inverse limit construction in this category.

\begin{defn}
An \emph{inverse system}  (or projective system) of compact nilspaces  over $\N$ is a family of continuous nilspace morphisms $\{\varphi_{ij}: \ns_j\to \ns_i~|~ i,j\in \N, i\leq j\}$, where $\ns_j$, $j\in \N$ are compact nilspaces, such that $\varphi_{jj}$ is the identity morphism for all $j\in \N$, and for all $i,j,k\in \N$ with $i\leq j\leq k$ we have $\varphi_{ij}\co\varphi_{jk}=\varphi_{ik}$. The inverse system is said to be \emph{strict} if every morphism $\varphi_{ij}$ is fibre-surjective.
\end{defn}

\begin{lemma}
Let $S=(\varphi_{ij}: \ns_j\to \ns_i)_{i\leq j}$ be a strict inverse system of compact nilspaces. Let
\[\ns=\Big\{(x_i)\in \prod_{i\in \N} \ns_i ~|~ \varphi_{ij}(x_j)=x_i\;\; \forall\,i,j\in \N,\;i\leq j\Big\},
\]
and for each $n\geq 0$ let $\cu^n(\ns)=\big\{(\q_i)\in \prod_{i\in \N} \cu^n(\ns_i)~|~ \varphi_{ij}\co \q_j= \q_i\;\;\forall\,i,j\in \N,\;i\leq j\big\}$. The space $\ns$ together with the cube sets $\cu^n(\ns)$ is a compact nilspace, called the \emph{inverse limit} of $S$, and denoted $\varprojlim_i \ns_i$. We have that $\ns$ is $k$-step if and only if every $\ns_i$ is $k$-step.
\end{lemma}
\noindent We call the maps $\varphi_{ij}$ the \emph{transition morphisms}. The \emph{projections} on $\ns$, denoted by $\varphi_j$, are the coordinate projections on $\prod_{i\in \N} \ns_i$ restricted to $\ns$. If the inverse system is strict then the projections are also fibre-surjective morphisms.  The $\ell$-th structure group of $\varprojlim_i \ns_i$ is the inverse limit of the system $(\alpha_{\ell,ij}:\ab_{\ell,j}\to \ab_{\ell,i})_{i\leq j}$, where $\ab_{\ell,j}$ is the $\ell$-th structure group of $\ns_j$ and $\alpha_{\ell,ij}$ is the $\ell$-th structure morphism of $\varphi_{ij}$.
\begin{proof}
A straightforward argument with convergent sequences shows that $\ns$ is a closed subspace of the compact space $\prod_{i\in \N} \ns_i$, and similarly for each set $\cu^n(\ns)$. The ergodicity and composition axioms are clear. Let us prove that $\ns$ satisfies the completion axiom. Let $\q'$ be an $n$-corner on $\ns$. Then for each $i\in \N$ the projection $\q_i'=\varphi_i\co \q'$ is an $n$-corner on $\ns_i$. Now $\q_1'$ has a completion $\q_1$ and by fibre-surjectivity there exists $\tilde \q_1\in \cu^n(\ns_2)$ such that $\varphi_{12}\co \tilde\q_1=\q_1$ (recall \cite[Lemma 3.3.9]{Cand:Notes1}). 
The restriction of $\tilde\q_1$ to $P=\{0,1\}^n\setminus\{1^n\}$ is in the same fibre of the map $\q'\mapsto \varphi_{12}\co \q'$ as $\q_2'$, so by \cite[Lemma 3.3.12 (ii)]{Cand:Notes1} these two corners differ by an $n$-corner on $\cD_k(\ker(\alpha_k))$, where $\alpha_k$ is the $k$-th structure morphism of $\varphi_{12}$. Now by modifying $\tilde\q_1$, we can obtain a cube $\q_2$ that is still in this fibre (so that $\varphi_{12}\co \q_2=\q_1$) and such that $\q_2|_P=\q_2'|_P$. Indeed, we can obtain $\q_2$ by adding first an appropriate element of $\ker(\alpha_k)$ to the values of $\tilde\q_1$ at the vertices of the 1-face $\{0^n,(1,0,\dots,0)\}$ so that it agrees with $\q_2'$ at $0^n$, then repeat this for each subsequent 1-face along a Hamiltonian path in the graph of 1-faces on $\{0,1\}^n$ (similar arguments were used several times in \cite{Cand:Notes1}, for instance in the proof of \cite[Lemma 3.2.25]{Cand:Notes1}). We have thus obtained a cube $\q_2$ completing $\q_2'$ and such that $\varphi_{12}\co \q_2=\q_1$. Now we repeat this operation to obtain $\q_3$ completing $\q_3'$ such that $\varphi_{23}\co \q_3=\q_2$, and so on. This yields a sequence of cubes $\q_i\in \cu^n(\ns)$ such that for each $i$ we have $\varphi_j\co \q_i$ is a completion of $\q_j'$ for all $j\leq i$. The limit of a convergent subsequence of $(\q_i)$ is a completion of $\q'$. To see the last claim in the lemma, recall from \cite[Definition 3.3.7]{Cand:Notes1} that a fibre-surjective morphism preserves the property of being $k$-step.
\end{proof}
We can now state the main result of this section.
\begin{theorem}\label{thm:invlim} Every $k$-step compact nilspace is an inverse limit of compact nilspaces of finite rank.
\end{theorem}

\noindent The basic case of this theorem, concerning 1-step nilspaces, is given by the standard result that every compact abelian group is an inverse limit of compact abelian Lie groups (see \cite[Corollary 2.43]{H&M}). This standard result also yields the following example of an inverse limit of nilspaces, which will be used in the proof of Theorem \ref{thm:invlim}.
\begin{proposition}\label{prop:k-level-invlim}
Let $\ns$ be a $k$-step compact nilspace. Then $\ns$ is an inverse limit of $k$-step nilspaces $\ns_i$, each with $k$-th structure group a compact abelian Lie group, and with $\cF_{k-1}(\ns_i)$ isomorphic to $\cF_{k-1}(\ns)$. \end{proposition}
\begin{proof}
By the standard result for compact abelian groups recalled above, for the $k$-th structure group $\ab_k$ of $\ns$ we have for each $i\in \N$ a compact abelian Lie group $\ab_{k,i}$ and a surjective continuous homomorphism $\alpha_i:\ab_k\to \ab_{k,i}$ such that $\ab_k=\varprojlim \ab_{k,i}$. (For $i\leq j$ the transition morphism $\alpha_{i,j}$ is $x\mapsto x+\ker \alpha_i$ for $x\in \ab_{k,j}\cong \ab_k/\ker \alpha_j$.) Define for each $i$ the $k$-step compact nilspace $\ns_i$ as the image of $\ns$ under the map $\varphi_i$ that sends $x\in \ns$ to the orbit $x+\ker \alpha_i$. A straightforward calculation shows that $\ns_i$, with the quotient cube-structure of $\ns$ by $\ker \alpha_i$ and the quotient topology, is a $k$-step compact nilspace with $k$-th structure group $\ab_{k,i}$. The projections $\ns\to \cF_{k-1}(\ns)$ and $\ns_i\to \cF_{k-1}(\ns_i)$ are both given by the orbit map for the $\ab_k$-action, and it follows that $\cF_{k-1}(\ns)\cong \cF_{k-1}(\ns_i)$ as compact nilspaces. The transition morphisms $\varphi_{ij}:\ns_j\to \ns_i$, defined by $x+\ker \alpha_j\mapsto x+\ker \alpha_i$, are all clearly fibre-surjective.
\end{proof}
\noindent Recall from \cite[Lemma 3.3.21]{Cand:Notes1} that given a degree-$k$ extension $\nss$ of $\ns$ by $\ab$, and given a cross section $\cs:\ns\to \nss$, letting $f:\nss\to \ab$, $y\mapsto \cs\co \pi(y)-y$, the cocycle generated by $\cs$ is the degree-$k$ cocycle $\rho_{\cs}:\cu^{k+1}(\ns)\to \ab$ defined by $\rho_{\cs}(\q)=\sigma_{k+1}(f\co \q')$, for any $\q'\in \cu^{k+1}(\nss)$ such that $\pi\co \q'=\q$.

For the proof of Theorem \ref{thm:invlim}, it will be useful to be able to tell in a simple way whether a given cocycle on a nilspace $\ns$ induces a well-defined cocycle on another nilspace $\ns'$ via some given fibre-surjective morphism $\ns\to \ns'$. The following definition provides such a criterion.

\begin{defn}\label{def:crossfactdef}
Let $\ns$ be a $k$-step compact nilspace, and for some compact nilspace $\ns'$ let $\psi:\cF_{k-1}(\ns)\to \ns'$ be a continuous fibre-surjective morphism. We say that a measurable cross section $\cs:\cF_{k-1}(\ns)\to \ns$ is \emph{consistent with the factor} $\ns'$ if the cocycle $\rho_{\cs}:\cu^{k+1}(\cF_{k-1}(\ns))\to \ab_k$ satisfies $\rho_{\cs}(\q_1)=\rho_{\cs}(\q_2)$ for every $\q_1,\q_2$ such that $\psi\co \q_1=\psi\co \q_2$.
\end{defn}
\noindent Thus if $\cs$ is consistent with $\ns'$ then $\rho_{\cs}$ induces a  cocycle $\rho':\cu^{k+1}(\ns')\to \ab_k$, well-defined by $\rho'(\q')=\rho_{\cs}(\q)$ for any $\q$ such that $\psi\co\q=\q'$. Note that we shall often refer to the image of a fibre-surjective morphism on $\ns$ as a \emph{fibre-surjective factor} of $\ns$.

\begin{lemma}\label{lem:crossfactor} 
Let $\psi:\cF_{k-1}(\ns)\to \ns'$ be a fibre-surjective morphism and let $\cs:\cF_{k-1}(\ns)\to \ns$ be a cross section consistent with $\ns'$. Let $\sim$ be the equivalence relation on $\ns$ defined by
\[
x\sim y\;\;\Leftrightarrow\;\; \big(\;\psi\co \pi_{k-1}(x)=\psi\co \pi_{k-1}(y)\textrm{ and }x-\cs\co\pi_{k-1}(x)=y-\cs\co\pi_{k-1}(y)\;\big).
\]
Then the quotient nilspace $\ns/_{\sim}$ is a fibre-surjective factor of $\ns$ which is an extension of $\ns'$ by $\ab_k$. 
\end{lemma}

\begin{proof} Let $\rho':\cu^{k+1}(\ns')\to \ab_k$ be the cocycle induced by $\rho_{\cs}$. Let $M(\rho')$ be the compact nilspace extending $\ns'$ by $\ab_k$  given by Proposition \ref{prop:cocyclext}. Let $f:\ns\to M(\rho')$ be defined by $f(x)=\rho'_{\psi(\pi_{k-1}(x))}+x-\cs\co\pi_{k-1}(x)$. Using Definition \ref{def:compextcubes} we check that $f$ is a morphism (a similar calculation appears in the proof of \cite[Lemma 3.3.28]{Cand:Notes1}). It is also checked easily that $f$ is fibre-surjective. Moreover, clearly $f(x_1)=f(x_2)$ if and only if $x_1\sim x_2$. It follows that $f$ factors through $\sim$ giving a bijective measurable morphism $f':\ns/_{\sim}\to\nss$. Then $f'$ is a continuous isomorphism of compact nilspaces, by Theorem \ref{thm:Klepgen}.
\end{proof}
\noindent The following lemma will be used in the proof of Theorem \ref{thm:invlim} and also in Section \ref{sec:CFRnilsnilm}. Recall the notation $d_2$ for the metric from Definition \ref{def:calgmetric}. 

\begin{lemma}\label{lem:vertiparathin}
Let $\ns$ be a $k$-step compact nilspace such that the $k$-th structure group $\ab_k$ has finite rank, and let $d$ be a compatible metric on $\ns$. Then for any $\epsilon>0$ there exists $\delta>0$ such that the following holds. Let $x_0,x_1,y_0,y_1\in \ns$ be such that $\pi_{k-1}(x_i)=\pi_{k-1}(y_i)$ for $i=0,1$ and $d(x_0,x_1),d(y_0,y_1)$ are both at most $\delta$. Then in $\ab_k$  we have $d_2(y_0-x_0,y_1-x_1)\leq \epsilon$.
\end{lemma}
\begin{proof}
Consider the local translation $\phi=\phi_{x_0,x_1}$ (recall \cite[Definition 3.2.26]{Cand:Notes1}). By definition of $\phi$ we have $y_0-x_0=\phi(y_0)-x_1$ in $\ab_k$ (using implicitly the isomorphism from \cite[Lemma 3.2.24]{Cand:Notes1}), so it suffices to show that $d_2(\phi(y_0)-x_1,y_1-x_1)\leq \epsilon$ for $\delta$ sufficiently small. 

Since $d(x_0,x_1)$ is unchanged by adding $y_0-x_0=\phi(y_0)-x_1$ to $x_0$ and $x_1$, we have $d(y_0,\phi(y_0))\leq \delta$. This together with $d(y_0,y_1)\leq \delta$ implies that $d(\phi(y_0),y_1)\leq 2\delta$. Thus, the lemma will follow if we prove the following statement: for every $\epsilon>0$ there exists $\delta>0$ such that whenever $x_1,x_2,x_3\in \ns$ are in the same fibre of $\pi_{k-1}$ and $d(x_2,x_3)\leq \delta$, we have $d_2(x_2-x_1,x_3-x_1)\leq \epsilon$. This can be shown by a straightforward argument using the local triviality of the $\ab_k$-bundle $\ns$ and the compactness of $\ns_{k-1}$.
\end{proof}

\subsection{Proof of the inverse limit theorem}

To prove Theorem \ref{thm:invlim}, we argue by induction on $k$, starting with the trivial case $k=0$. For $k>0$ we suppose that $\ns$ is a $k$-step compact nilspace such that the nilspace $\nss=\cF_{k-1}(\ns)$ is the inverse limit of a strict inverse system $(\tau_{ij}:\nss_j\to\nss_i)_{i,j\in \N, i\leq j}$. We denote by $\pi$ the projection $\ns\to \nss$, and for each $i$ we denote by $\tau_i$ the projection $\nss\to\nss_i$ and let $\mathcal{Q}_i=\{\tau_i^{-1}(Q):Q\subset \nss_i,\; Q\textrm{ open}\}$. Since by assumption the topology on $\nss$ is the initial topology generated by the maps $\tau_i$, we have that the finite intersections of sets in $\bigcup_{i\in \N}\mathcal{Q}_i$ form a  base for this topology.

\subsubsection{The inductive step: construction of the sequence $S$}

We now apply Proposition \ref{prop:k-level-invlim} to $\ns$. Let $B_0=\ab_k$ be the $k$-th structure group of $\ns$, and let $B_i=\ker \alpha_i\leq \ab_k$ for each $i\in \N$ (for the homomorphisms $\alpha_i$ from the proof of the proposition). Then $B_0 \geq B_1 \geq \cdots$, with $\bigcap_{i\in \N} B_i=\{0\}$, and $\ns = \varprojlim_i \ns/_{B_i}$. Let $q_{i-1,i}$ denote the transition morphism $\ns/_{B_i}\to \ns/_{B_{i-1}}$ and $q_i$ the projection $\ns\to \ns/_{B_i}$. Note that $\ns/_{B_i}$ has $k$-th structure group $\ab_k/B_i$, of finite rank.

The main goal of the inductive step is to construct, starting with the 1-point nilspace $\ns_0$, a sequence $S=\{(\ns_i,\psi_i,\psi_i',h_i):i\in \N\}$, where for each $i$ we have a compact nilspace $\ns_i$ of finite rank, and fibre-surjective morphisms $\psi_i:\ns/_{B_i}\to \ns_i$ and $\psi_i':\ns_i\to \ns_{i-1}$, with the following properties:
\begin{enumerate}
\item For every $i$, the restriction of $\psi_i$ to each class of $\sim_{k-1}$ in $\ns/_{B_i}$ is injective. In other words, the $k$-th structure group of $\ns_i$ is the same as that of $\ns/_{B_i}$, namely $\ab_k/B_i$. 
\item For every $i$, we have $\psi'_i\circ\psi_i=\psi_{i-1}\co q_{i-1,i}$.
\item Letting $\pi_i$ denote the projection $\ns_i\mapsto  \cF_{k-1}(\ns_i)$, and $\kappa_i$ denote the projection $\ns/_{B_i}\to \nss$, we have $\cF_{k-1}(\ns_i)\cong \nss_{h_i}$ and $\pi_i\co \psi_i= \tau_{h_i}\co \kappa_i$. \end{enumerate}
Properties (i) and (iii) imply that any two distinct points in $\ns$ are separated by a map $\psi_i\co q_i$ for some $i$, so the initial topology on $\ns$ generated by these maps is Hausdorff. Since the original topology on $\ns$ is compact, it follows that this initial topology equals the original topology (see \cite[\S 9.4, Corollary 3]{Bourb1}). The maps $\varphi_{ij}=\psi'_{i+1}\co\cdots\co \psi'_j$ and $\varphi_i:=\psi_i\co q_i$ satisfy the relations $\varphi_{ij}\co \varphi_j= \varphi_i$ thanks to property (ii). Hence $\ns$ is the inverse limit of the system $(\varphi_{ij}:\ns_j\to\ns_i)_{i\leq j}$. \medskip \\ 
\includegraphics{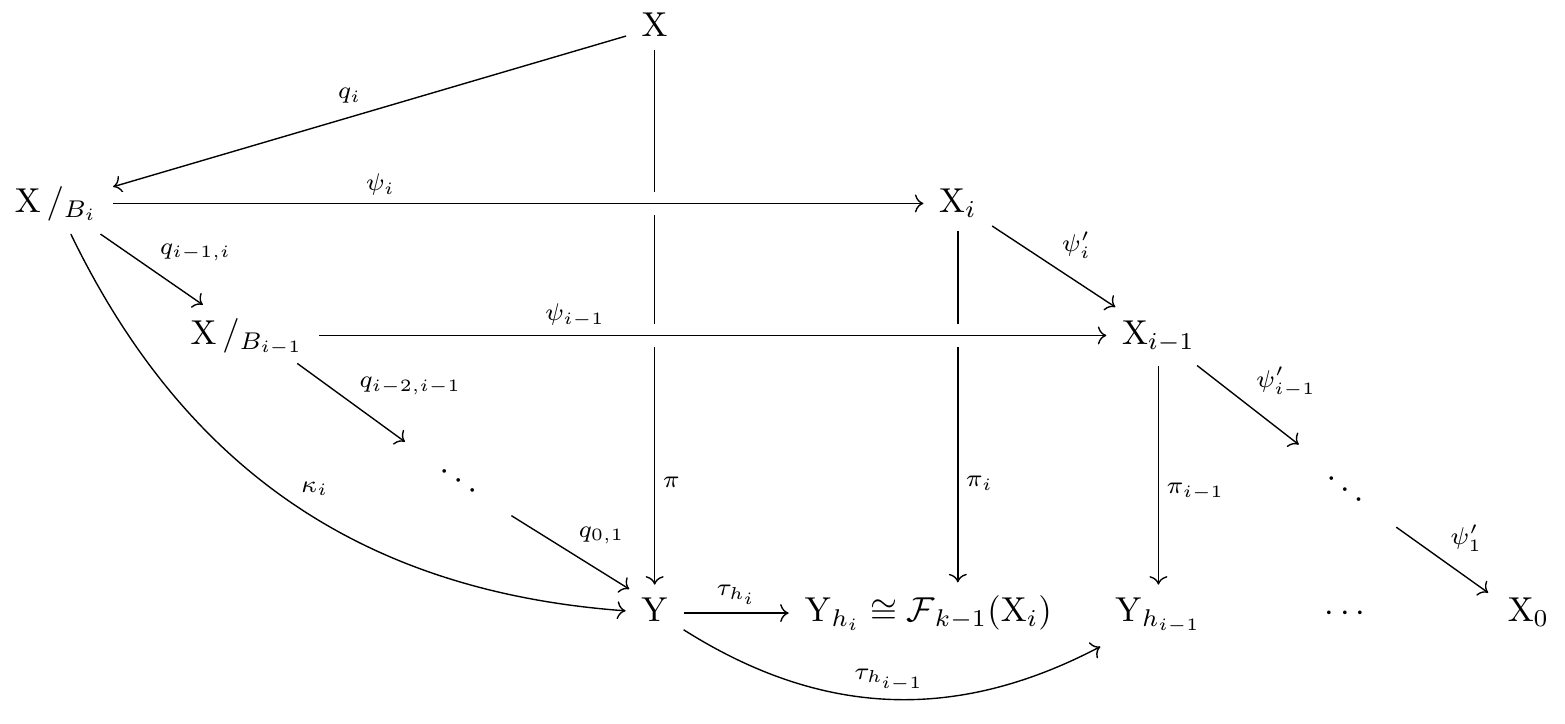}\\
\noindent We shall now construct the sequence $S$ inductively starting from $\ns_0$. Suppose that $(\ns_i,\psi_i,\psi'_i,h_i)$ has been constructed for $i\leq m$. Then, to begin with, we construct a Borel cocycle $\cu^{k+1}(\nss)\to \ab_k/B_{m+1}$ generated by a cross section $\cs$ that is almost consistent with some appropriate factor of $\nss$.

\noindent \textbf{Constructing a measurable cross section $\cs$:} since $\ns_m$ is a compact locally-trivial $\ab_k/B_m$-bundle over $\nss_{h_m}$, there is a finite cover of $\nss_{h_m}$ by closed subsets $W_1,\ldots,W_r$ such that every preimage $\pi_m^{-1}(W_i)\subset \ns_m$ is a trivial $\ab_k/B_m$-bundle over $W_i$. For each $i\in [r]$, let $\theta_i:W_i\to \ns_m$ be a continuous cross section (thus $\pi_m\co\theta_i$ is the identity map on $W_i$).\\
\indent For each $a\in [r]$ let $W'_a$ denote the preimage of the set $\theta_i(W_a)$ under $\psi_m\co q_{m,m+1}:\ns/_{B_{m+1}}\to \ns_m$. Note that $W'_a$ is a $\ab_k/B_{m+1}$-bundle over $\tau_{h_m}^{-1}(W_a)\subset \nss$, so it is locally trivial by Proposition \ref{prop:Gleason}.  Let $d$ be a metric on $\ns/_{B_{m+1}}$ generating its topology.\\
\begin{center}\includegraphics{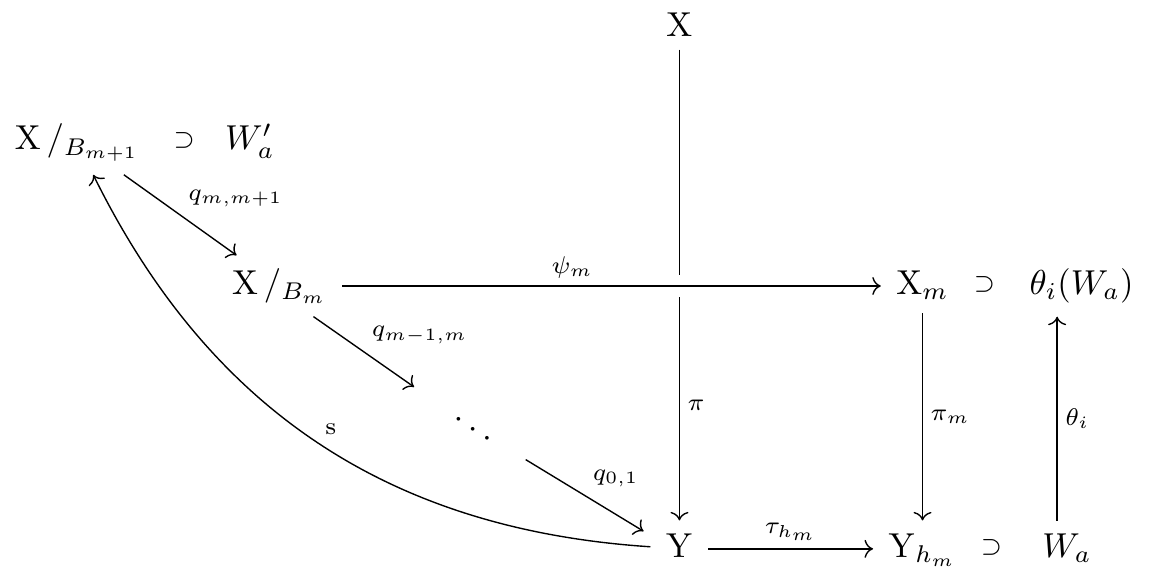}\end{center}

\noindent Now fix an arbitrary $\epsilon>0$. We claim that for every $p\in \tau_{h_m}^{-1}(W_a)$ there exists an open set $U_p\subset  \tau_{h_m}^{-1}(W_a)$ containing $p$ with the following properties:\\ \vspace{-0.7cm}
\begin{enumerate}
\setcounter{enumi}{3}
\item There exists a continuous cross section $\cs_p:U_p\to \ns/_{B_{m+1}}$.\\ \vspace{-0.7cm}
\item The diameter of $\cs_p(U_p)$ in the metric $d$ is at most $\epsilon$.\\ \vspace{-0.7cm}
\item For some $t(p)\in\N$, the set $U_p$ is in the collection $\mathcal{Q}_{t(p)}$ defined above.\\ \vspace{-0.7cm}
\end{enumerate}
\noindent Property (iv) is ensured by the local triviality of the bundle, and it can be satisfied together with (v) by letting $U_p$ be a ball of sufficiently small radius, say. Property (vi) can also be satisfied since $\bigcup_i \cQ_i$ generates the topology on $\nss$ and for every $i\leq j$ and open $Q\subset \nss_i$ we have $\tau_i^{-1}Q=\tau_j^{-1}(\tau_{ij}^{-1}Q)\in \cQ_j$. Since $\nss$ is compact there exists a finite cover by such sets $U_p$, thus for some points $p_1,p_2,\dots,p_n$ we have $\bigcup_{i=1}^n U_{p_i}\,\supset\,\tau_{h_m}^{-1}(W_a)$. Let $t_a=\max\{t(p_i):i \in [n]\}$, so that for every $i\in [n]$ we have $U_{p_i}\subset \mathcal{Q}_{t_a}$. Then, by dividing $\tau_{h_m}^{-1}(W_a)$ into the atoms of the Boolean algebra generated by $U_{p_1},\dots,U_{p_n}$, and then using one of the cross sections $\cs_p$ for each atom, we can construct a Borel measurable cross section $\cs_a:\tau_{h_m}^{-1}(W_a)\to \ns/_{B_{m+1}}$ with the following properties:\\ \vspace{-0.7cm}
 \begin{enumerate}
 \setcounter{enumi}{6}
\item The cross section $\cs_a$ is continuous on $\tau_{t_a}^{-1}(v)$ for every $v\in \nss_{t_a}$.\\ \vspace{-0.7cm}
 \item The diameter of $\cs_a(\tau_{t_a}^{-1}(v))$ in $d$ is at most $\epsilon$ for every $v\in \nss_{t_a}$.\\ \vspace{-0.7cm}
 \end{enumerate}
Note that we obtain (viii) from (v) because each $\tau_{t_a}^{-1}(v)$ is entirely contained in an atom of the Boolean algebra, which contains $\tau_{t_a}^{-1}(Q)$ for some open $Q\supset v$. Let $t=\max\{t_a: a\in [r]\}\cup \{h_m+1\}$. Using these partial cross sections, we now claim that we can construct a global Borel cross section $\cs:\nss\to \ns/_{B_{m+1}}$ with the following properties:\\ \vspace{-0.7cm}
\begin{enumerate}
\setcounter{enumi}{8}
\item The cross section $\cs$ is continuous on every preimage $\tau_t^{-1}(v)$ where $v\in \nss_t$.\\ \vspace{-0.7cm}
\item The diameter of $\cs(\tau_t^{-1}(v))$ in $d$ is at most $\epsilon$ for every $v\in \nss_t$.\\ \vspace{-0.7cm}
\item If $v_1,v_2\in \nss$ satisfy $\tau_{h_m}(v_1)=\tau_{h_m}(v_2)$, then $\psi_m\co q_{m,m+1} (\cs(v_1))=\psi_m\co q_{m,m+1}(\cs(v_2))$. 
\end{enumerate}
Indeed, we can construct $\cs$ by again dividing $\nss$ into the atoms of the Boolean algebra generated by the sets $\tau_{h_m}^{-1}(W_a)$ and then using one of the cross sections $\cs_a$ on each atom. This immediately gives (ix) and (x), and then (xi) follows from (iii) above. \\
\indent Let $\rho_{\cs}:\cu^{k+1}(\nss)\to \ab_k/B_{m+1}$ be the cocycle generated by $\cs$.\\
\indent Now for any fixed $\epsilon_2>0$, we note that if $\epsilon>0$ is small enough, then property (x) implies the following fact, which tells us that $\cs$ is almost consistent with the factor $\nss_t$:
\begin{equation}\label{eq:rhoclose}
\forall\,\q_1,\q_2\in \cu^{k+1}(\nss)\textrm{ with } \tau_t\co\q_1= \tau_t\co\q_2,\textrm{ we have } d_2\big(\rho_{\cs}(\q_1)-\rho_{\cs}(\q_2)\big)\leq \epsilon_2.
\end{equation}
Indeed $\tau_t\co\q_1= \tau_t\co\q_2$ tells us that for each $v$ the values $\q_1(v),\q_2(v)$ are in the same fibre of $\tau_t$ and so by (x) we have that $d\big(\cs\co \q_1(v),\cs\co\q_1(v)\big)\leq \epsilon$ for all $v$. Then by continuity of $\kappa_{m+1}$ the values $\q_i(v)$, $i=1,2$ are also close in $\nss$ for each $v$. Therefore by Lemma \ref{lem:closelifts} there exist $\tilde \q_1,\tilde \q_2\in \cu^{k+1}(\ns/_{B_{m+1}})$ such that $d\big(\tilde\q_1(v),\tilde\q_2(v)\big)\leq \epsilon_2$ for every $v$ and such that $\kappa_{m+1}\co \tilde\q_i=\q_i$, $i=1,2$. Finally, from the expression $\rho_{\cs}(\q_i)=\sigma_{k+1}(\cs\co \q_i-\tilde\q_i)$ we obtain the inequality in \eqref{eq:rhoclose} by applying Lemma \ref{lem:vertiparathin}.\\
\indent Note that property (xi) also implies that
\begin{equation}\label{eq:modBm}
\forall\,\q_1,\q_2\in \cu^{k+1}(\nss)\textrm{ with } \tau_{h_m} \co\q_1 = \tau_{h_m}\co \q_2,\textrm{ we have } \rho_{\cs}(\q_1)-\rho_{\cs}(\q_2)\in B_m/B_{m+1}. 
\end{equation}
Indeed, since $\psi_m$ is injective by (i), we have by (xi) that $\tau_{h_m} \co\q_1 = \tau_{h_m}\co \q_2$ implies $q_{m,m+1}\co \cs\co\q_1(v)=q_{m,m+1}\co \cs\co \q_2(v)$, so $\cs\co \q_1(v)$, $\cs\co \q_2(v)$ are in the same orbit of $B_m/B_{m+1}$ for every $v$. In particular we have $\rho_{\cs}(\q_1)-\rho_{\cs}(\q_2)=\sigma_{k+1}(\cs\co \q_1-\cs\co \q_2)$, whence \eqref{eq:modBm} follows.

\noindent To complete the proof we shall now use $\rho_{\cs}$ to find a new cross section $\cs': \nss\to \ns/B_{m+1}$ that is consistent with the factor $\nss_t$, and that is also congruent to $\cs$ modulo $B_m$, that is $q_{m,m+1}\co\cs'=q_{m,m+1}\co\cs$. If we find $\cs'$ then, firstly, thanks to the consistency, we can apply Lemma \ref{lem:crossfactor} with map $\psi=\tau_t$, obtaining a new fibre-surjective factor $\ns_{m+1}$ of $\ns/_{B_{m+1}}$ that is an extension of $\nss_t$ by $\ab_k/B_{m+1}$, and so (i) holds. Secondly $\psi_m\co q_{m,m+1}$ factors through $\ns_{m+1}$, which yields $\psi_{m+1}'$ satisfying (ii). Indeed, by definition of the relation $\sim$ in Lemma  \ref{lem:crossfactor} we have that $x,y\in \ns/_{B_{m+1}}$ are in the same fibre of $\psi_{m+1}$ if and only if $\tau_t\co\kappa_{m+1}(x)=\tau_t\co\kappa_{m+1}(y)$ and $x-\cs'\co\kappa_{m+1}(x)=y-\cs'\co\kappa_{m+1}(y)$. The first equality here implies, thanks to the congruence mod $B_m$ and (xi), that $\psi_m\co q_{m,m+1}(\cs'\co \kappa_{m+1}(x))=\psi_m\co q_{m,m+1}(\cs'\co\kappa_{m+1}(y))$, and this together with the second equality gives  $\psi_m\co q_{m,m+1}(x)=\psi_m\co q_{m,m+1}(y)$. Finally, property (iii) holds, since $x\sim y\Rightarrow \tau_t\co\kappa_{m+1}(x)=\tau_t\co\kappa_{m+1}(y)$ implies that $\tau_t\co\kappa_{m+1}$ factors through $\psi_{m+1}$, and it must then factor as $\pi_{m+1}$, by definition of $\kappa_{m+1}$. We will have thus found the next term in the sequence $S$.\\

\noindent \textbf{Averaging $\rho_{\cs}$ to obtain a fibre-surjective factor $\ns_{m+1}$ of $\ns/_{B_{m+1}}$:} to find $\cs'$, the idea is firstly to define a new function $\rho'$ by averaging $\rho_{\cs}$, in such a way that the approximate equality in \eqref{eq:rhoclose} becomes an equality for $\rho'$, and secondly then to show that $\rho'$ is itself a cocycle generated by some cross section, which will be the desired $\cs'$.\\
\indent Recall that, by \cite[Lemma 3.3.12 (ii) \& (iii)]{Cand:Notes1}, for each $i\in \N$ the map
$\beta_i:\cu^i(\nss)\to\cu^i(\nss_t)$, $\q\mapsto \tau_t\co \q$ is a totally surjective bundle morphism, and for every $\q\in \cu^i(\nss_t)$ the preimage $\beta_i^{-1}(\q)$ is a $(k-1)$-fold sub-bundle of $\cu^i(\nss)$. (Recall also Lemmas \ref{lem:surjmorphcubemeas} and \ref{lem:cube-set-CSM}.)
We define $\rho':\cu^{k+1}(\nss)\to \ab_k/B_{m+1}$ by
\begin{equation}\label{eq:avcocycle}
\rho'(\q)=\E_{\q'\in\beta_{k+1}^{-1}(\beta_{k+1}(\q))}\;\rho_{\cs}(\q').
\end{equation}
The averaging here is relative to the Haar measure on $\beta_{k+1}^{-1}(\beta_{k+1}(\q))$, and is well-defined for $\epsilon_2$ sufficiently small, thanks to  \eqref{eq:rhoclose} (recall Definition \ref{def:calgaver}). By \eqref{eq:modBm} and the fact that $t>h_m$, we have that $\rho'(\q)-\rho_{\cs}(\q)\in B_m/B_{m+1}$ for every $\q\in \cu^{k+1}(\nss)$. Note also that it follows clearly from \eqref{eq:avcocycle} that
\begin{equation}\label{eq:rhoexact}
\forall\,\q_1,\q_2\in \cu^{k+1}(\nss)\textrm{ with } \tau_t\co \q_1 = \tau_t\co \q_2,\textrm{ we have }\rho'(\q_1)=\rho'(\q_2).
\end{equation}
\noindent We shall now prove that $\rho'$ is a cocycle by showing that it inherits the required properties from $\rho_{\cs}$ thanks to the linearity of the averaging in \eqref{eq:avcocycle}. The required properties are the two axioms from \cite[Definition 3.3.14]{Cand:Notes1}. Axiom (i) follows straight from the same property for $\rho_{\cs}$ and linearity using the fact that any set $\beta_{k+1}^{-1}(\beta_{k+1}(\q))$ is globally invariant under composition with automorphisms in $\aut(\{0,1\}^{k+1})$. To see the second axiom we take two adjacent cubes $\q_1,\q_2\in \cu^{k+1}(\nss)$ with concatenation $\q_3$ and embed them into a $(k+2)$-dimensional cube $\q\in \cu^{k+2}(\nss)$ as restrictions to two adjacent $(k+1)$-dimensional faces $F_1$ and $F_2$ in $\{0,1\}^{k+2}$. The concatenation of $F_1$ and $F_2$ is a diagonal subcube $F_3$. We have that the concatenation of $\q_1$ and $\q_2$ is the restriction of $\q$ to $F_3$. Note that the existence of $\q$ is guaranteed by simplicial completion \cite[Lemma 3.1.5]{Cand:Notes1}. By \cite[Lemma 3.3.12 (iii)]{Cand:Notes1}, we have a Haar probability on $\Omega=\beta_{k+2}^{-1}(\beta_{k+2}(\q))$. Then, by \cite[Lemma 3.3.12 (iv)]{Cand:Notes1}, the probability spaces $\beta_{k+1}^{-1}(\beta_{k+1}(\q_i))$, $i=1,2,3$ are faithfully embedded as factors and coupled in the big probability space $\Omega$. Using the concatenation property for $\rho$ in $\Omega$ (when the random cube is restricted to $F_1,F_2,F_3$) and linearity of expectation we obtain that $\rho'(c_3)=\rho'(c_1)+\rho'(c_2)$.

It remains to prove that the cocycle $\rho'$ is generated by a cross section. Let $\rho''$ denote the cocycle $\rho'-\rho_{\cs}$, which takes values in $B_m/B_{m+1}$ (as noted just before \eqref{eq:rhoexact}). It follows from \eqref{eq:rhoclose} that $d_2(\rho''(\q),0)\leq\epsilon_2$ for every $\q\in \cu^{k+1}(\nss)$, so for $\epsilon_2$ sufficiently small we have that $\rho''$ is a coboundary, by Lemma \ref{lem:smallco}. Thus for some Borel function $g:\nss\to B_m/B_{m+1}$ we have $\rho''(\q)=\sigma_{k+1}(g\co \q)$. Hence, defining the Borel cross section $\cs':\nss\to \ns/_{B_{m+1}}$ by $\cs'(y)= \cs(y)+g(y)$, we have $\rho'=\rho_{\cs'}$.

By \eqref{eq:rhoexact} we have that $\cs'$ is consistent with the factor $\nss_t$. Applying Lemma \ref{lem:crossfactor} as described above, we find the new term in the sequence $S$. 

This completes the proof of Theorem \ref{thm:invlim}.

\medskip
\section{Rigidity of morphisms}

\medskip
Given a set $X$, a metric space $(Y,d)$, and maps $\phi,\phi':X\to Y$, we say that $\phi'$ is an \emph{$\epsilon$-modification} of $\phi$ if for every $x\in X$ we have $d(\phi(x),\phi'(x))\leq \epsilon$.

\begin{defn}
Let $\ns$ be a compact nilspace, and let $\nss$ be a $k$-step compact nilspace with a metric $d$ generating its topology. A map $\phi:\ns\to\nss$ is a $\delta$\emph{-quasimorphism} if for every $\q\in \cu^{k+1}(\ns)$ there exists $\q'\in \cu^{k+1}(\nss)$ such that $d\big(\phi\co\q(v),\q'(v)\big)\leq\delta$ for every $v\in \{0,1\}^{k+1}$.
\end{defn}
In this section we establish the following result.

\begin{theorem}\label{thm:rigidity}
Let $\nss$ be a $k$-step compact nilspace of finite rank, with a metric $d$ generating its topology. Then for every $\epsilon>0$ there exists $\delta>0$ such that the following holds. If $\ns$ is a  compact nilspace and $\phi:\ns\to \nss$ is a Borel $\delta$-quasimorphism, then $\phi$ has an $\epsilon$-modification that is a continuous morphism.
\end{theorem}
\noindent In the proof we shall use the fact that the metric $d$ can be assumed to be invariant under the action of the structure group $\ab_k$ (Lemma \ref{lem:d-invariance}). We shall also use the following `rectification result' for cubes.
\begin{lemma}\label{lem:verticubapprox}
Let $\ns$ be a $k$-step compact nilspace with metric $d$. For every $\epsilon>0$ there exists $\delta>0$ such that the following holds. If $\q\in \cu^{k+1}(\ns)$ satisfies $d'\big(\pi_{k-1}\co \q(v,0),\pi_{k-1}\co\q(v,1)\big)\leq \delta$ for every $v\in \{0,1\}^k$, then $\q$ has an $\epsilon$-modification $\q'\in \cu^{k+1}(\ns)$ such that $\pi_{k-1}\co \q'(v,0) = \pi_{k-1}\co\q'(v,1)$ for every $v\in \{0,1\}^k$.
\end{lemma}
\begin{proof}
Suppose for a contradiction that for some $\epsilon>0$ there is a sequence of cubes $\q_n\in \cu^{k+1}(\ns)$ such that $d'\big(\pi_{k-1}\co \q_n(v,0),\pi_{k-1}\co\q_n(v,1)\big)\leq 1/n$ for every $v\in \{0,1\}^k$ and yet, for every cube $\q'\in \cu^{k+1}(\ns)$ such that $\pi_{k-1}\co \q'(v,0) = \pi_{k-1}\co\q'(v,1)$ for all $v\in\{0,1\}^k$, there is some $v\in \{0,1\}^{k+1}$ such that $d(\q_n(v),\q'(v))> \epsilon$. By compactness of $\cu^{k+1}(\ns)$, there is a subsequence $(\q_m)$ of $(\q_n)$ such that $\q_m\to \q^*\in \cu^{k+1}(\ns)$. The topology on $\cu^{k+1}(\ns)$ can be metrized by $d_\infty(\q_1,\q_2):=\max_{v\in \{0,1\}^{k+1}} d(\q_1(v),\q_2(v))$. It then follows from $d_\infty(\q_m, \q^*)\to 0$ and the triangle inequality for $d'$ that $\pi_{k-1}\co \q^*(v,0) = \pi_{k-1}\co\q^*(v,1)$ for every $v\in \{0,1\}^k$. However, it also follows from $\q_m\to \q^*$ that $d(\q_m(v),\q^*(v))\leq \epsilon$ for all $v\in \{0,1\}^{k+1}$, for sufficiently large $m$, and this  contradicts our initial assumption since $\pi_{k-1}\co \q^*(\cdot,0) = \pi_{k-1}\co\q^*(\cdot,1)$.
\end{proof}
\begin{proof}[Proof of Theorem \ref{thm:rigidity}]
We argue by induction on $k$, starting from the trivial case $k=0$. For $k>0$, suppose that the result holds for $k-1$, and fix any $\epsilon>0$. Let $\delta>0$ be a parameter to be determined later, suppose that $\phi:\ns\to\nss$ is a $\delta$-quasimorphism, and let $\phi_1=\pi_{k-1}\co\phi$. Then $\phi_1$ is a $\delta$-quasimorphism into $\cF_{k-1}(\nss)$. By induction, for some $\delta_1(\epsilon)>0$ to be fixed later, we have that if $\delta$ is small enough then $\phi_1$  has a $\delta_1$-modification $\phi_2:\ns\to \cF_{k-1}(\nss)$ that is a continuous morphism.

We shall now obtain a Borel measurable map $\phi_3:\ns\to \nss$ that lifts $\phi_2$, i.e. such that $\pi_{k-1}\co \phi_3=\phi_2$, and which is a $(\delta+\delta_1)$-modification of $\phi$, so that $\phi_3$ is a $\delta_2$-quasimorphism with $\delta_2=\delta_1+2\delta$. To obtain this, the idea is first to obtain such a lift locally for each neighbourhood in some appropriate cover of $\ns$, and then combine the local lifts into a global one. To this end we use the following set:
\[
G=\{(x,y):\pi_{k-1}(y)=\phi_2(x)\}\subset \ns\times \nss.
\]
We claim that $G$ is compact. Indeed, on one hand the map $f_1\times f_2:\ns\times \nss\to \cF_{k-1}(\nss)\times \cF_{k-1}(\nss)$, $(x,y)\mapsto (\phi_2(x),\pi_{k-1}(y))$ is continuous, and on the other hand $(x,y)\in G$ if and only if $f_1\times f_2(x,y)$ is in the diagonal of $\cF_{k-1}(\nss)\times \cF_{k-1}(\nss)$, which is a closed set in the product topology since $\cF_{k-1}(\nss)$ is Hausdorff. Hence $G$ is closed in $\ns\times \nss$ and therefore compact. We also claim that $G$ is a $\ab_k$-bundle over $\ns$, where $\ab_k$ is the $k$-th structure group of $\nss$. Indeed, the action of $\ab_k$ is given by $a\cdot (x,y)=(x,y+a)$, and the projection map is just $(x,y)\mapsto x\in \ns$.\\
\indent By Proposition \ref{prop:Gleason}, the bundle $G$ is locally trivial and so for every $p\in \ns$ we can find an open set $U_p\subset \ns$ containing $p$ with a continuous cross section $U_p\to G$, which must then be of the form $x\mapsto (x,\tau(x))$ for some continuous map $\tau:U_p\to \nss$. Thus $\phi_2(x)=\pi_{k-1}(\tau(x))$ for every $x\in U_p$. Then, by the definition of $d'$, since for all $x\in U_p$ we have
\[
d'\big(\pi_{k-1}(\tau(x)),\pi_{k-1}(\phi(x))\big)=d'(\phi_2(x),\phi_1(x))\leq \delta_1,
\]
in particular for $p$ there exists $z_p\in \ab_k$ such that $d\big(\tau(p)+z_p,\phi(p)\big)\leq \delta_1$. But then, since the function $f_p:x\mapsto \tau(x)+z_p$ is continuous on $U_p$, there is an open set $U_p'\subset U_p$ containing $p$ such that for all $x\in U_p'$ we have $d(\phi(x),\tau(x)+z_p)<\delta_1+\delta$. The function $f_p$ on $U'_p$ is the desired local lift of $\phi_2$.\\
\indent Now by compactness there is a finite covering of $\ns$ by such open sets $U_p'$. Let $U_1,\dots,U_s$ be the atoms of the Boolean algebra generated by the sets $U_p'$. For each $i\in [s]$ we fix one of the functions $f_p$ defined on $U_s$ and rename it $f_i$. We then define the Borel function $\phi_3$ by $\phi_3(x)=\sum_{i\in [s]}1_{U_i}(x) f_i(x)$.

We shall now apply an averaging argument to $\phi_3$ that will produce
a continuous morphism $\phi_4$ which is close to $\phi$ as desired.

Let $P_2=\{0,1\}^{k+1}\setminus\{0^{k+1}\}$. Note that since $\phi_2$ is a morphism $\ns\to \cF_{k-1}(\nss)$, it preserves $k$-cubes,
and since $\phi_3$ is a lift of $\phi_2$, by \cite[Remark 3.2.12]{Cand:Notes1} we have that $\phi_3$ also preserves $k$-cubes on $\nss$. This implies that for every $\q\in \cu^{k+1}(\ns)$ the restriction of $\phi_3\co \q$ to $P_2$ is a $(k+1)$-corner on $\nss$. For any given cube $\q\in \cu^{k+1}(\ns)$ let us denote by $\comp_0(\q)$ the value at $0^{k+1}$ of the unique completion in $\cu^{k+1}(\nss)$ of $(\phi_3\co \q)|_{P_2}$. Let
\begin{equation}\label{eq:fi4}
\phi_4(x)=\E_{\q\in \cu^{k+1}_x(\ns)}\, \comp_0(\q).
\end{equation}
Let us show that this averaging is well-defined. Note that for every $\q\in\cu^{k+1}_x(\ns)$ we have $\comp_0(\q)\sim_{k-1} \phi_3(x)$, indeed since the morphism $\phi_2$ preserves $(k+1)$-cubes, and $\phi_3$ is a lift of $\phi_2$, we must have $\comp_0(\q)$ and $\phi_3(x)$ both in the fibre  $\pi_{k-1}^{-1}(\phi_2(x))$. We have to show that if $\delta_2>0$ is small enough then the set of values $\comp_0(\q)$, $\q\in \cu^{k+1}_x(\ns)$ has small diameter in the fibre. We will show that $\comp_0(\q)$ is actually close to $\phi_3(x)$ for every $\q\in \cu^{k+1}_x(\ns)$, which will also be useful later. For any $\q\in \cu^{k+1}_x(\ns)$, since $\phi_3$ is a $\delta_2$-quasimorphism, there  is $\q_0\in \cu^{k+1}(\nss)$ such that for all $v$ we have $d\big(\phi_3\co\q(v),\q_0(v)\big)\leq\delta_2$. In particular, this inequality holds for all $v\in P_2$. By continuity of the corner-completion (Lemma \ref{lem:contcomp}) it follows that for $\lambda>0$, if $\delta_2$ is small enough then we have $d\big(\comp_0(\q),\q_0(0^{k+1})\big)\leq\lambda$. Hence $d\big(\phi_3(x),\comp_0(\q)\big)=d\big(\phi_3\co \q(0^{k+1}),\comp_0(\q)\big)\leq \delta_2+\lambda$. We can now fix $\lambda$ so that the averaging in \eqref{eq:fi4} is indeed well-defined.\\ 
\indent Let us now prove that $\phi_4$ is a morphism. By \cite[Lemma 3.2.13]{Cand:Notes1}, it suffices to show that for every $\q\in \cu^{k+1}(\ns)$ we have $\phi_4\co \q\in \cu^{k+1}(\nss)$. We shall do this using another averaging argument, but working this time with the tricube $T_{k+1}$. Let $B=T_{k+1}\setminus \{-1,1\}^{k+1}$. For every $t\in\hom_{\q\co \omega_{k+1}^{-1}}(T_{k+1},\ns)$, there is a unique completion of $\phi_3 \co t|_B$ to a morphism $t':T_{k+1}\to \nss$. Let $\q_t=t'\co \omega_{k+1}$, which is in $\cu^{k+1}(\nss)$ by \cite[Lemma 3.1.16]{Cand:Notes1}. If $\delta_2$ is small enough then the following function is well defined:
\[
\q_2: \{0,1\}^{k+1}\to \nss,\; v\mapsto \E_{t\in\hom_{\q\co \omega_{k+1}^{-1}}(T_{k+1},\ns)}\,\q_t(v).
\]
Indeed, by Corollary \ref{cor:ext-in-Tn}, for each $v$ this average is equal to $\E_{\q'\in \cu^{k+1}_{\q(v)}(\ns)} \comp_0(\q')=\phi_4(\q(v))$. Thus, we have to show that $\q_2\in \cu^{k+1}(\ns)$. To see this, recall that for every $\q'\in \cu^{k+1}_{\q(v)}(\ns)$ we have $\comp_0(\q')\sim_{k-1}\phi_3(\q(v))$, and so $\q_2(v)\sim_{k-1}\phi_3\co \q(v)$, which implies that $\pi_{k-1}\co\q_2= \phi_2\co \q\in \cu^{k+1}(\cF_{k-1}(\nss))$. Therefore, by \cite[Theorem 3.2.19]{Cand:Notes1} and the definition of degree-$k$ extensions, we just have to check that for some (any) $t_0$ we have $\sigma_{k+1}(\q_2-\q_{t_0})=0$. This follows from linearity of averaging (Lemma \ref{lem:avadd}) and the fact that for $\q_t,\q_{t_0}$ are cubes with the same projection to $\cF_{k-1}(\nss)$ for every $t$, so 
\[
\sigma_{k+1}(\q_2-\q_{t_0})\;=\;\sum_v (-1)^{|v|}\; \E_{t\in\hom_{\q\co \omega_{k+1}^{-1}}(T_{k+1},\ns)}\,\q_t(v) - \q_{t_0}(v)\;=\;\E_t\, \sigma_{k+1}\big(\q_t - \q_{t_0}\big) = 0.
\]
Finally, let us prove that $\phi_4$ is continuous. The argument is similar to the proof of Theorem  \ref{thm:Klepgen}. Consider the CSM structure on $\pi_0:\cu^{k+1}(\ns)\to \ns$, $\q\mapsto \q(0^{k+1})$, given by Lemma \ref{lem:cube-set-CSM}. For each $v\in \{0,1\}^{k+1}\setminus \{0^{k+1}\}$ let $\pi_v:\cu^{k+1}(\ns)\to \ns$, $\q\mapsto \q(v)$. We have that $\pi_v$ is continuous and each of its restrictions to a space $\cu^{k+1}_x(\ns)$ is measure-preserving, by Lemma \ref{lem:GoodPairHoms}. For each $v$ we let $f_v:\ns\to\cL(\cu^{k+1}(\ns),\nss)$ be the function sending $x$ to the restriction of $\phi_3\co\pi_v$ to $\cu_x^{k+1}(\ns)$. Then, just as in the proof of Theorem  \ref{thm:Klepgen}, we have that $f_v$ is continuous, by Lemma \ref{lem:csmtech}. We now want to express $\phi_4$ in terms of the functions $f_v$. To that end, consider again the map $\xi : \ns \to \cL(\cu^{k+1}(\ns),\cor^{k+1}(\nss))$, sending $x$ to the function $\xi_x:\q\in\cu^{k+1}_x(\ns) \mapsto (\q(v)=f_v(x)(\q))_{v\neq 0^{k+1}}$. Just as in the proof of Theorem \ref{thm:Klepgen}, the continuity of each function $f_v$ implies that $\xi$ is continuous. We then use again the continuous map $\cK: \cL(\cu^{k+1}(\ns),\cor^{k+1}(\nss))\to \cL(\cu^{k+1}(\ns),\nss)$ that sends a function $g$ to $g\co \comp$, where $\comp$ is the unique completion of $(k+1)$-corners on $\nss$. Now note that $\cK\co\,\xi: \ns\to \cL(\cu^{k+1}(\ns),\nss)$ is precisely the map that sends each $x\in \ns$ to the function $\q\mapsto \comp_0(\q)$ on $\cu^{k+1}_x(\ns)$. This map is continuous, and by the discussion justifying the averaging \eqref{eq:fi4}, we have that in fact $\cK\co\,\xi$ takes values in the following subset of $\cL\big(\cu^{k+1}(\ns),\nss\big)$:
\[
\cU= \bigcup_{x\in \ns} \Big\{g\in L\big(\cu^{k+1}_x(\ns),\pi_{k-1}^{-1}(\phi_2(x))\big)~:~\forall \q\in \cu^{k+1}_x(\ns),\;\; d\big(g(\q),\phi_3(x)\big)\leq \delta_2+\lambda\Big\}.
\]
Now $\phi_4$ is the composition of $\cK\co\,\xi$ with the averaging operator $\cA:\,\cU \to \pi_{k-1}^{-1}(\phi_2(x))$ that sends a function $g$ to $\E_{\q\in \cu^{k+1}_x(\ns)} g(\q)$. Thus it now suffices to show that $\cA$ is continuous. Note that since $ \pi_{k-1}^{-1}(\phi_2(x))$ is homeomorphic to $\ab_k$, it follows from the definition of the metric $d_2$ on $\ab_k$ that if $\delta_2+\lambda$ is sufficiently small then every function $g\in \cU$ can be lifted continuously to a function $g':\cu^{k+1}_x(\ns)\to \R^n$, where $\ab_k=F\times \T^n$, in such a way that two functions $g_1,g_2\in \cU$ are close in the metric $d_1$ from \eqref{eq:gen-L1} if and only if their lifts $g_1',g_2'$ are close in the same metric but with $d$ now being the Euclidean distance on $\R^n$. By Lemma \ref{lem:Ltop-gen-restrict}, if $g_n\to g$ in the topology on $\cL(\cu^{k+1}(\ns),\nss)$ restricted to $L(\cu^{k+1}_x(\ns),\nss)$ , then $d_1(g_n,g)\to 0$, whence by the triangle inequality $\cA(g_n)\to \cA(g)$, so $\cA$ is  continuous as required.
\end{proof}

\medskip
\section{Characterizing compact connected nilspaces of finite rank}\label{sec:CFRnilsnilm}

\medskip
Recall from \cite[Subsection 3.2.4]{Cand:Notes1} that given a $k$-step nilspace $\ns$ we denote by $\tran_i(\ns)$ the group of translations of height $i$ on $\ns$ (or $i$-translations). By a slight abuse of notation, when $\ns$ is a compact nilspace we shall write $\tran_i(\ns)$ for the group of $i$-translations that are also \emph{continuous} functions. A central goal in this section is to show that if $\ns$ has finite rank then $\tran_i(\ns)$ is a Lie group for each $i$ and $\tran_1(\ns)$ acts transitively on the connected components of $\ns$. This will then enable us to show that if $\ns$ has connected structure groups then it can be identified with a filtered nilmanifold (Theorem \ref{thm:toralnilspace}).

Throughout this section we shall abbreviate `compact and finite-rank' by writing `\textsc{cfr}'.

\begin{lemma}
Let $\ns$ be a \textsc{cfr} $k$-step nilspace. Then every element of $\tran(\ns)$ is a homeomorphism $\ns\to \ns$ preserving the Haar measure.
\end{lemma}
\begin{proof}
Each translation is an invertible continuous map from the compact Hausdorff space $\ns$ to itself and so its inverse is also continuous \cite[Theorem 26.6]{Munkres}. Moreover, the translation is a fibre-surjective automorphism of $\ns$, so it is measure preserving by Corollary \ref{cor:ctsfibsurmorph}.
\end{proof}
\noindent For a compact space $X$ and a space $Y$ with metric $d$,  recall that $C(X,Y)$ denotes the space of continuous functions $f:X\to Y$, with the topology induced by the uniform metric $d_\infty(f_1,f_2)=\sup_{x\in X} d(f_1(x),f_2(x))$. We record the following basic fact.
\begin{lemma}\label{lem:transpolish}
Let $\ns$ be a \textsc{cfr} $k$-step nilspace and let $d$ be a metric generating the topology on $\ns$. Then for every $i\in [k]$,  the group $\tran_i(\ns)$ equipped with the restriction of the uniform metric on $C(\ns,\ns)$ is a Polish group.
\end{lemma}
\begin{proof}
The group $H(\ns)$ of homeomorphisms on $\ns$, with the relative topology from $C(X,X)$, is a Polish group \cite[\S 1.3, Example (v)]{Be&Ke}. It follows from the definition of translations \cite[Definition 3.2.27]{Cand:Notes1} and the closure of cube sets that $\tran_i(\ns)$ is a closed subgroup of $H(\ns)$, hence it is also a Polish group. By definition the metric $d_\infty$ generates the topology on $\tran_i(\ns)$.
\end{proof}
\noindent Recall that by \cite[Lemma 3.2.31]{Cand:Notes1} a translation maps every class of the relation $\sim_{k-1}$ onto another such class. This enables us to define a translation on the factor $\ns_{k-1}=\cF_{k-1}(\ns)$, as in the following lemma. Recall also that the topology on $\ns_{k-1}$ is the quotient topology from $\ns$, which can be metrized by the quotient metric $d'$ defined in \eqref{eq:quotientmetric}. We then denote by $d'_\infty$ the uniform metric on $\tran(\ns_{k-1})$ relative to $d'$.
\begin{lemma}\label{lem:hdefn}
For each $i\in \N$ let $h$ be the map sending each $\alpha\in \tran_i(\ns)$ to the map $h(\alpha)$ on $\ns_{k-1}$ defined by $h(\alpha)(y)=\pi_{k-1}(\alpha(x))$, for any $x\in \ns$ such that $\pi_{k-1}(x)=y$. Then $h$ is a continuous homomorphism $\tran_i(\ns)\to \tran_i(\ns_{k-1})$.
\end{lemma}
\begin{proof}
\noindent That $h$ is a homomorphism $\tran_i(\ns)\to \tran_i(\ns_{k-1})$ follows from the definitions. The continuity follows similarly, thus
\begin{eqnarray*}
d'_\infty\big(h(\alpha_1),h(\alpha_2)\big)& =& \sup_{y\in \ns_{k-1}} d'\big(h(\alpha_1)(y),h(\alpha_2)(y)\big)\;\;=\;\;\sup_{x\in \ns} d'\big(\pi_{k-1}(\alpha_1(x)),\pi_{k-1}(\alpha_2(x))\big)\\
& = & \sup_{x\in \ns}\;\;\; \inf_{\substack{x_1\sim_{k-1} \alpha_1(x)\\ x_2\sim_{k-1} \alpha_2(x)}} d(x_1,x_2)\;\;\leq\;\;\sup_{x\in \ns}d\big(\alpha_1(x),\alpha_2(x)\big)= d_\infty(\alpha_1,\alpha_2).
\end{eqnarray*}
\end{proof}
\noindent An important result toward our goal in this section is that if an  $i$-translation on $\ns_{k-1}$ is sufficiently close to the identity then it can be lifted to an $i$-translation on $\ns$, in the following sense.
\begin{lemma}\label{lem:small-trans-lift}
Let $\ns$ be a \textsc{cfr} $k$-step nilspace and let $i\in \N$. There exists $\epsilon>0$ such that if $\alpha\in\tran_i(\ns_{k-1})$ satisfies $d'_\infty(\alpha,\id)< \epsilon$ then there exists $\beta\in\tran_i(\ns)$ such that $h(\beta)=\alpha$.
\end{lemma}

\begin{proof}
We shall use notation from \cite[Lemma 3.3.38]{Cand:Notes1}. Recall that by that lemma the translation bundle $\mathcal{T}^*(\alpha,\ns,i)=\cF_{k-1}(\mathcal{T}(\alpha,\ns,i))$ is, from the purely algebraic viewpoint, a degree-$(k-i)$ extension of $\ns_{k-1}$ by $\ab_k$, and note that by the results from Section \ref{sec:toprelims} we have that this extension is a continuous $\ab_k$-bundle. We show first that if $\epsilon$ is sufficiently small then this extension has a measurable cross section $\cs$ such that the corresponding cocycle $\rho_{\cs}$ is a coboundary.\\
\indent Let $\gamma$ be the projection $\mathcal{T}^*(\alpha,\ns,i)\to \ns_{k-1}$. We claim that if $\epsilon>0$ is sufficiently small, then we can choose a Borel cross section $\cs:\ns_{k-1}\to \cT^*$ with the property that for each $\cs(x)$, for every $(x_0,x_1)$ in the equivalence class $\pi_{k-1,\cT}^{-1}(\cs(x))\subset \cT$, we have $d(x_0,x_1)\leq \epsilon$. (Recall that we have $\alpha(\pi_{k-1}(x_0))=\pi_{k-1}(x_1)$ for $(x_0,x_1)\in \cT$, by \cite[Definition 3.3.34]{Cand:Notes1}.) Indeed, by definition $d'_\infty(\alpha,\id)< \epsilon$ implies that, for each $x\in \ns_{k-1}$, there exist $x_0,x_1\in \ns$ with $\pi_{k-1}(x_0)=x$, $\pi_{k-1}(x_1)=\alpha(x)$, and such that $d(x_0,x_1)< \epsilon$. Then, by $\ab_k$-invariance of $d$, we have that every pair $(y_0,y_1)$ in the same equivalence class as $(x_0,x_1)$ in $\cT^*$ satisfies $d(y_0,y_1)\leq \epsilon$. We choose $\cs(x)$ to be one of these pairs, say $(x_0,x_1)$. Note that this choice can be made in a Borel measurable way. Indeed, using the fact that $\ns$ is a locally trivial $\ab_k$-bundle over $\ns_{k-1}$ (Proposition \ref{prop:Gleason}), we can construct $\cs$ as a piecewise continuous function.

Now let $\rho_{\cs}$ be the cocycle generated by $\cs$, defined for $\q\in \cu^{k-i+1}(\ns_{k-1})$ by $\rho_{\cs}(\q)= \sigma_{k-i+1}(\cs\co\q-\q')$ for any lift $\q'\in \cu^{k-i+1}(\cT^*)$ of $\q$. We have that $\rho_{\cs}$ is measurable (recall the end of the proof of Lemma \ref{lem:meascross}). We claim that, for some $\epsilon_2>0$ to be fixed later, if $\epsilon$ is sufficiently small then $\rho_{\cs}$ is also small in the sense that $d_2(\rho_{\cs}(\q),0_{\ab_k})\leq \epsilon_2$ for every $\q\in \cu^{k-i+1}(\ns_{k-1})$. To show  this we first give an alternative expression of the function $\cs\co\q-\q'$.  The lift $\q'$, being a cube on  $\cT^*$, has itself a lift $\tilde \q\in \cu^{k-i+1}(\cT)$. By definition we have $\tilde \q=\tilde\q_0\times \tilde\q_1$ where $\pi_{k-1}\co\tilde\q_1(v) = \alpha (\pi_{k-1}\co\tilde\q_0(v))$ for all $v$. By the claim in the previous paragraph, for each $v$ we have that $\cs\co \q(v)$ is a class of pairs $(x_0,x_1)\in \pi_{k-1}^{-1}(\q(v))\times \pi_{k-1}^{-1}(\alpha\co\q(v))$ satisfying $d(x_0,x_1)\leq \epsilon$ and corresponding to a local translation $\phi$, thus $x_1=\phi(x_0)$ for every such pair. The difference $\cs\co\q(v)-\q'(v)$ is then the element $a_v\in \ab_k$ which has to be added to $\tilde\q_0(v)$ in order to have $\phi(\tilde\q_0(v)+a_v)=\tilde\q_1(v)$, namely $a_v=\phi^{-1}(\tilde\q_1(v))-\tilde\q_0(v)$. Thus, letting $a$ denote the function $\{0,1\}^{k-i+1}\to \ab_k$, $v\mapsto a_v$, we have $\rho_{\cs}(\cs\co\q-\q')=\sigma_{k-i+1}(a)$. To show that this must be a small element of $\ab_k$, we argue as follows. By assumption on $\alpha$, the cube $\tilde\q$ satisfies the premise of Lemma \ref{lem:verticubapprox} with $\delta=\epsilon$. Thus, given $\epsilon_1>0$ to be fixed later, if $\epsilon<\epsilon_1$ is sufficiently small then by that lemma there must exist $\q^*=\q^*_0\times \q^*_1\in \cu^{k-i+1}(\cT)$ such that for all $v$ we have $d(\q^*_i(v),\tilde\q_i(v))\leq \epsilon_1$ for  $i=0,1$, and with $\pi_{k-1}\co \q^*_1=\pi_{k-1}\co \q^*_0$. For each $v$ let $a^*(v)=\q^*_1(v)-\q^*_0(v)\in \ab_k$. It now suffices to show that $d_2(a^*(v),a(v))\leq \epsilon_2/2^{k-i+1}$ for every $v$, for then, by virtue of $\arr{\q^*_0,\q^*_1}_1$ being a cube, we must have $\sigma_{k-i+1}(a^*)=0$ and so $\sigma_{k-i+1}(a)\leq \epsilon_2$ as claimed. To show that $d_2(a^*(v),a(v))$ is small we can apply Lemma \ref{lem:vertiparathin} with $\delta=2\epsilon_1$, since we know on one hand that $d(\q_0^*(v),\tilde \q_0(v))\leq \epsilon_1$, and on the other hand that $d(\phi^{-1}(\tilde\q_1(v)),\tilde\q_1(v))\leq \epsilon$ and $d(\q_1^*(v),\tilde \q_1(v))\leq \epsilon_1$, so that $d(\phi^{-1}(\tilde\q_1(v)),\q_1^*(v))\leq \epsilon+\epsilon_1\leq 2\epsilon_1$.

We can now apply Lemma \ref{lem:smallco}. Thus if $\epsilon_2$ is sufficiently small then there exists a Borel function $g:\cF_{k-1}(\ns)\to \ab_k$ such that $\rho_{\cs}(\q)=\sigma_{k-1+i}(g\co \q)$ for all $\q\in \cu^{k-i+1}\big(\cF_{k-1}(\ns)\big)$. Let $m:\cF_{k-1}(\ns)\to \cT^*$ be the measurable function $x\mapsto \cs(x)-g(x)$. We have that $m$ is a morphism, indeed for any $\q\in \cu^n(\cF_{k-1}(\ns))$, on one hand we clearly have that the $\cT^*$-valued map $m\co \q$ is a lift of $\q$, and on the other hand for some (any) lift $\q'\in \cu^n(\cT^*)$ of $\q$ we have $\sigma_{k-i+1}(m\co \q - \q')=\sigma_{k-i+1}(\cs\co \q - \q')-\sigma_{k-i+1}(g\co \q)=\rho_{\cs}(\q)-\rho_{\cs}(\q)=0$, so by \cite[Definition 3.3.13 (ii)]{Cand:Notes1} we have indeed $m\co\q\in \cu^n(\cT^*)$. Now \cite[Proposition 3.3.39]{Cand:Notes1} gives us an $i$-translation $\beta$ on $\ns$ that is a lift of $\alpha$, and is defined as follows: for each $x\in \ns$ we have $\beta(x)=\phi(x)$, where $\phi$ is the local translation corresponding to the equivalence class $m\co\pi_{k-1}(x)\in \cT^*$. It only remains to check that $\beta$ is Borel measurable, as then, being a morphism, it must be continuous by Theorem \ref{thm:Klepgen}. By \cite[Definition 3.3.34 and Proposition 3.3.36]{Cand:Notes1}, the nilspace $\cT$ is a continuous $\ab$-bundle over $\cT^*$, for the polish group $\ab=\{(z,z):z\in \ab_k\}\leq \ab_k\times \ab_k$. Therefore, arguing as in the end of the proof of Lemma \ref{lem:meascross}, we obtain a Borel cross section $\lambda:\cT^*\to\cT$, and so the map $\lambda\co m\co\pi_{k-1}:\ns\to \cT$ is Borel. Moreover, letting $\lambda\co m\co\pi_{k-1}(x)=(y_0,y_1)$, we have $\beta(x)=y_1+(x-y_0)$, so $\beta$ is indeed Borel.
\end{proof}
\noindent We shall need a useful criterion for a translation to lie in the kernel of $h$. Note that if $\alpha\in \ker(h)$ then, since by \cite[Lemma 3.2.31]{Cand:Notes1} the restriction of $\alpha$ to each fibre of $\pi_{k-1}$ is a local translation, we deduce that the map $x\mapsto \alpha(x)-x$ is a constant element of $\ab_k$ for every $x$ in a given fibre. Thus $\alpha$ induces a well-defined map $\alpha':\ns_{k-1}\to \ab_k$, $\pi_{k-1}(x) \mapsto \alpha(x)-x$. Conversely, given $\alpha':\ns_{k-1}\to \ab_k$ we can define $\alpha:\ns \to \ns$, $x\mapsto x+\alpha'(\pi_{k-1}(x))$, and we can ask when is $\alpha$ in $\tran(\ns)$.
\begin{lemma}\label{lem:h-kernel}
Let $i<k$. Then $\alpha\in \ker (h)\cap \tran_i(\ns)$ if and only if $\alpha'\in \hom(\ns_{k-1},\cD_{k-i}(\ab_k))$.
\end{lemma}
\begin{proof}
By \cite[Lemma 3.2.32]{Cand:Notes1} we know that $\alpha\in \tran_i(\ns)$ if and only if for every $\q\in \cu^n(\ns)$ we have $\arr{\q,\alpha\co \q}_i\in \cu^{n+i}(\ns)$. If $\alpha\in \ker(h)\cap \tran_i(\ns)$ then we can take the difference $\alpha\co \q-\q=\alpha'\co\q'$ for $\q'=\pi_{k-1}\co\q$, and by definition of degree-$k$ extensions this difference must be a cube in $\cD_k(\ab_k)$. But we also have $\arr{\q,\alpha\co \q}_i-\arr{\q,\q}_i\in \cu^{n+i}(\cD_k(\ab_k))$, and this equals $\arr{0,\alpha'\co\q'}_i$. By \cite[Lemma 3.3.37]{Cand:Notes1}, we have $\arr{0,\alpha'\co\q'}_i\in \cu^{n+i}(\cD_k(\ab_k))$ if and only if $\alpha'\co\q'\in \cu^n(\cD_{k-i}(\ab_k)$. The converse follows similarly.
\end{proof}
\noindent The next main tool that we need, Lemma \ref{lem:hom-oscillation} below, is a type of rigidity result for morphisms between abelian torsors of the form $\cD_k(\ab)$. The proof will use the following basic fact.
\begin{lemma}\label{lem:abhomoscil}
Let $\ab'$ be a \textsc{cfr} abelian group, and let $d_2$ be the metric from Definition \ref{def:calgmetric}. Then there exists $\eta>0$ such that for every compact abelian group $\ab$ and every non-constant continuous affine homomorphism $\phi:\ab\to \ab'$, there exist $x,y\in \ab$ such that $d_2(\phi(x),\phi(y))\geq \eta$. 
\end{lemma}
\begin{proof}
We can assume without loss that $\phi$ is a homomorphism (not just an affine one). If $\ab'$ is a finite group then the claim is verified by any two points $x,y$ such that $\phi(x)\neq \phi(y)$. We may therefore suppose that $\ab'=F\times \T^n$ with $n>0$ and $F$ a finite abelian group, and that for some $i\in [n]$ the projection $\pi_i:\T^n\to \T$, $x\mapsto x_i$ satisfies that $\pi_i\co\phi(\ab)$ is a nontrivial subgroup of $\T$. Identifying $\T^n$ with $[0,1)^n$, there is then $y\in \ab$ such that $|\pi_i\co\phi(y)-0|_{\T}\geq 1/4$ and so with $x=0_{\ab}$ we have $d_2(\phi(x),\phi(y))\geq 1/4$.
\end{proof}
\begin{lemma}\label{lem:hom-oscillation}
Let $k,\ell\in\N$, let $\ab,\ab'$ be compact abelian groups, and suppose that $\ab'$ has finite rank. Then there exists $\epsilon=\epsilon(\ell,\ab')>0$ such that if $\phi\in \hom(\cD_k(\ab),\cD_\ell(\ab'))$ satisfies $d_2(\phi(x),\phi(y))\leq \epsilon$ for every $x,y\in \ab$ then $\phi$ is a constant function.
\end{lemma}
\begin{proof}
Since $\hom(\cD_k(\ab),\cD_\ell(\ab'))\subseteq\hom(\cD_1(\ab),\cD_\ell(\ab'))$ we can assume that $k=1$. Let $\phi$ be an arbitrary non-constant morphism from $\cD_1(\ab)$ to $\cD_\ell(\ab')$. As mentioned in \cite[Example 2.2.13]{Cand:Notes1} (recall also \cite[Definition 2.2.30]{Cand:Notes1}), the map $\phi$ is a polynomial map of degree at most $\ell$. Thus, if for every $t\in \ab$ and every $f:\ab\to \ab'$ we denote by $\Delta_t f$ the function $x\mapsto f(x+t)-f(x)$, then there is $i<\ell$ and elements $t_1,t_2,\dots,t_i\in \ab$ such that $\phi'=\Delta_{t_1}\co \Delta_{t_2}\co\cdots\co \Delta_{t_i}\phi$ is non-constant but $\Delta_t \phi'$ is constant for every $t\in \ab$. It follows that $\phi'$ is a non-constant continuous affine homomorphism from $\ab$ to $\ab'$, so by Lemma \ref{lem:abhomoscil} there are $x,y\in \ab$ with $d_2(\phi'(x),\phi'(y))\geq \eta(\ab')>0$. This cannot hold if $\max_{x,y}d_2(\phi(x),\phi(y))$ is too small, so this quantity must have a positive lower bound (depending only on $\ab'$ and $\ell$).
\end{proof}
Lemma \ref{lem:hom-oscillation} has the following consequence.

\begin{corollary}\label{cor:hom-oscillation}
Let $\ell\in \N$, let $\ns$ be a $k$-step compact nilspace, and let $\ab$ be a \textsc{cfr} abelian group. Then there exists $\epsilon=\epsilon(\ell,\ab)>0$ such that if $\phi\in \hom(\ns,\cD_\ell(\ab))$ satisfies $d_2(\phi(x),\phi(y))\leq \epsilon$ for every $x,y\in \ns$ then $\phi$ is constant. In particular, for any $x_0\in \ns$, $z_0\in \ab$, the set $Y=\{\phi\in \hom(\ns,\cD_\ell(\ab)): \phi(x_0)=z_0\}$, equipped with the metric $d_\infty(\phi,\phi')=\max_{x\in \ns}d_2(\phi(x),\phi'(x))$, is discrete.
\end{corollary}
\begin{proof}
Suppose that $d_2(\phi(x),\phi(y))<\epsilon$ for every $x,y\in \ns$, where $\epsilon=\epsilon(\ell,\ab)$ is the constant from Lemma \ref{lem:hom-oscillation}. We prove by induction on $k$ that $\phi$ is constant.\\
\indent If $k=1$ then $\ns$ is abelian and the result follows from Lemma \ref{lem:hom-oscillation}. Suppose that the statement holds for $k-1$. Then Lemma \ref{lem:hom-oscillation} tells us that $\phi$ is constant on the $\sim_{k-1}$ classes of $\ns$. This means that $\phi$ can be regarded as a function on $\cF_{k-1}(\ns)$, and so by our assumption $\phi$ is constant.\\
\indent To see the last sentence in the corollary, note that if $\phi,\phi'\in Y$ satisfy $d_\infty(\phi,\phi')=d_\infty(\phi-\phi',0)\leq \epsilon/2$, then by the triangle inequality $\phi-\phi'$ satisfies the premise of the previous sentence in the corollary, and then by that sentence and the definition of $Y$ we deduce that $\phi-\phi'$ must be the constant $0$.
\end{proof}
With these results we can now give a useful description of the kernel of $h$.
\begin{lemma}\label{lem:h-kernel-Lie}
Let $\ns$ be a \textsc{cfr} $k$-step nilspace, and let $h:\tran(\ns)\to \tran(\ns_{k-1})$ be the homomorphism from Lemma \ref{lem:hdefn}. Then $\ker(h)$ is isomorphic as a topological group to $F\times \ab_k$ for some discrete group $F$. In particular $\ker(h)$ is a Lie group.
\end{lemma}

\begin{proof}
Let $x\in \ns$ be arbitrary and let $F$ be the stabilizer $\stab_{\ker(h)}(x)=\{\alpha\in \ker(h):\alpha(x)=x\}$. Let $\tau:\ab_k\to \ker(h)$ be the map sending $z\in \ab_k$ to the translation $\tau_z: x\mapsto x+z$, and note that $\tau$ is an isomorphism of topological groups between $\ab_k$ and $\tau(\ab_k)\leq \ker(h)$. Now given any $\alpha \in \ker(h)$, letting $z=\alpha(x)-x$ we have that $\alpha-\tau_z\in F$, so $\ker(h)=F\cdot \tau(\ab_k)$. Moreover, since every $\alpha\in F$ must in fact stabilize every point in the fibre of $x$ (as the trivial local translation from this fibre to itself), it follows that $F\lhd\,\ker(h)$. We also have $\tau(\ab_k)\lhd\,\ker(h)$. (In fact $\tau(\ab_k)\subset Z(\ker(h))$.) Using Lemma \ref{lem:h-kernel} and Corollary \ref{cor:hom-oscillation} we see that $F$ is discrete. It follows that $\ker(h)\cong F\times \ab_k$ as claimed. Since $\ab_k$ is a Lie group, the proof is complete.
\end{proof}
\noindent If $G$ is a topological group then we denote the connected component of $\id_G$ by $G^0$. This is a closed normal subgroup of $G$. It is also standard that $G$ is a Lie group if and only if $G^0$ is a Lie group.
\begin{theorem}\label{thm:transurj} Let $\ns$ be a \textsc{cfr} $k$-step nilspace and let $i\in [k]$. Then the following statements hold.\\ \vspace{-0.7cm}
\begin{enumerate}
\item $\tran_i(\ns)$ is a Lie group,
\item $h\big(\tran_i(\ns)^0\big)=\tran_i(\ns_{k-1})^0$.
\end{enumerate}
\end{theorem}
\begin{proof}  We prove the statements by induction on $k$. The case $k=0$ being trivial, let $k>0$ and suppose that the statements hold for $k-1$.

To see statement $(i)$, note first that since $h$ is a continuous homomorphism and  $\tran_i(\ns)$ is closed, we have that $h(\tran_i(\ns))$ is a closed subgroup of $\tran_i(\ns_{k-1})$. By induction the latter is a Lie group, so by Cartan's theorem $h(\tran_i(\ns))$ is a Lie group \cite[Theorem 2.12.6]{Var}. By Lemma \ref{lem:h-kernel-Lie} we have that $\ker(h)$ is a Lie group, and so $H:= \tran_i(\ns)\cap\ker (h)$ is a Lie group (being a closed subgroup of $\ker(h)$). On the other hand $h(\tran_i(\ns))$ is isomorphic to $\tran_i(\ns)/H$ as a topological group (see  \cite[Theorem 1.2.6]{Be&Ke}). Therefore $\tran_i(\ns)/H$ must be a Lie group (this can be seen using \cite[Theorem 3.1]{GleasonSmall}). We can now deduce that $\tran_i(\ns)$ is itself a Lie group, by  \cite[Lemma A.3]{HK}.

Next we show that
\begin{equation}\label{eq:inclutrans}
\tran_i(\ns_{k-1})^0\subseteq h(\tran_i(\ns)).
\end{equation}
To see this we use that $\tran_i(\ns_{k-1})$ is a Lie group and so every element $\alpha\in\tran_i(\ns_{k-1})^0$ is connected to the identity by a continuous path $p:[0,1]\to\tran_i(\ns_{k-1})$ with $p(0)=\id$ and $p(1)=\alpha$. For $n\in\N$ let $\alpha_i=p((i-1)/n)^{-1}p(i/n)$ for $i\in [n]$ and let $\alpha=\prod_{i=1}^n\alpha_i$.
Lemma \ref{lem:small-trans-lift} implies that if $n$ is sufficiently large  then for every $\alpha_i$ there is $\beta_i\in\tran_i(\ns)$ with $h(\beta_i)=\alpha_i$. Let $\beta=\prod_{i=1}^n\beta_i$. Since $h$ is a homomorphism we have $h(\beta)=\alpha$.

To obtain statement $(ii)$, we first show that
\begin{equation}\label{eq:inclutrans2}
h(\tran_i(\ns)^0)= h(\tran_i(\ns))^0.
\end{equation}
On one hand we have $h(\tran_i(\ns)^0)\subset h(\tran_i(\ns))^0$ since $h$ is continuous. On the other hand $h(\tran_i(\ns))^0$ is open in $h(\tran_i(\ns))$  
and $h$ is an open map $\tran_i(\ns)\to h(\tran_i(\ns))$ (again by \cite[Theorem 1.2.6]{Be&Ke}), so $h(\tran_i(\ns)^0)$ is open in $h(\tran_i(\ns))$. Since it is also closed, we have that both $h(\tran_i(\ns)^0)$ and its complement in $h(\tran_i(\ns))^0$ are open, so this complement must be empty, since $h(\tran_i(\ns))^0$ is connected, and \eqref{eq:inclutrans2} follows. Now in statement $(ii)$, the inclusion $h(\tran_i(\ns)^0)\subset \tran_i(\ns_{k-1})^0$ follows by continuity of $h$, and for the opposite inclusion, note that by \eqref{eq:inclutrans} we have $\tran_i(\ns_{k-1})^0\subseteq h(\tran_i(\ns))^0$, so by \eqref{eq:inclutrans2} we have $\tran_i(\ns_{k-1})^0\subseteq h(\tran_i(\ns)^0)$.
\end{proof}
\begin{remark}
It follows from statement $(ii)$ above that $\tran_i(\ns)^0$ is a topological extension of $\tran_i(\ns_{k-1})^0$ by $\ker(h)\cap \tran_i(\ns)^0$, that is we have the short exact sequence
\[
0\;\;\xrightarrow{\makebox[0.5cm]{}}\;\; \ker(h)\cap \tran_i(\ns)^0\;\; \xrightarrow{\makebox[0.5cm]{$\iota$}}\;\; \tran_i(\ns)^0 \;\;\xrightarrow{\makebox[0.5cm]{$h$}}\;\; \tran_i(\ns_{k-1})^0 \;\;\xrightarrow{\makebox[0.5cm]{}}\;\; 0.
\]
This gives another way to obtain statement $(i)$ by induction given $(ii)$, using that a topological extension of a Lie group by a Lie group is again a Lie group (see \cite[\S 2.6]{Tao:Hilbert}).
\end{remark}
\begin{corollary}\label{cor:actrans} The Lie group $\tran(\ns)^0$ acts transitively on the connected components of $\ns$.
\end{corollary}

\begin{proof}
We argue by induction on $k$. For $k>0$ suppose that $\tran(\ns_{k-1})^0$ acts transitively on the connected components of $\ns_{k-1}$. Let $\tau$ be the isomorphism $\ab_k\to \tau(\ab_k)\leq \tran(\ns)$ in the proof of Lemma \ref{lem:h-kernel-Lie}. Consider the group $\tau(\ab_k)\cdot \tran(\ns)^0 = \{\tau(z)\,\alpha:z\in \ab_k,\,\alpha\in\tran(\ns)^0\}  \leq \tran(\ns)$, and let $M$ be any connected component of $\ns$. We claim that $\tau(\ab_k)\cdot \tran(\ns)^0$ acts transitively on $M$. Indeed, for any $x,y\in M$, by Theorem \ref{thm:transurj} (ii) and the induction hypothesis, there exists $\alpha\in \tran(\ns)^0$ such that $\pi_{k-1}\co\alpha(x)=\pi_{k-1}(y)$. Combining this with the transitivity of $\ab_k$ on the fibres of $\pi_{k-1}$ we deduce that there exists $\alpha'\in \tau(\ab_k)\cdot \tran(\ns)^0$ such that $\alpha'(x)=y$, which proves our claim. It follows that $\tran(\ns)$ acts transitively on $M$, and we can then deduce that $\tran(\ns)^0$ also acts transitively on $M$ by standard results (e.g. \cite[Chapter II, Theorem 3.2 and Proposition 4.3]{Helga}).
\end{proof}

\begin{example}\label{ex:heisentoral}
The main results of this section so far can be illustrated with the Heisenberg nilmanifold $H/\Gamma$, where $H=\begin{psmallmatrix} 1 & \R & \R\\[0.1em]  & 1 & \R \\[0.1em]  &  & 1 \end{psmallmatrix}$, $\Gamma=\begin{psmallmatrix} 1 & \Z & \Z\\[0.1em]  & 1 & \Z \\[0.1em]  &  & 1 \end{psmallmatrix}$ (recall \cite[Example 2.3.2]{Cand:Notes1}). Let $\ns$ denote $H/\Gamma$ equipped with the cubes associated with the lower central series on $H$. We then have that $\ns$ is a 2-step compact nilspace, with 1-step factor $\ns_1=\T^2$ equipped with the standard degree-1 cubes, and we compute that
\[
\tran_1(\ns)=\begin{psmallmatrix} 1 & \R & \T\\[0.1em]  & 1 & \R \\[0.1em]  &  & 1 \end{psmallmatrix}, \;\;\tran_2(\ns)=\begin{psmallmatrix} 1 & 0 & \T\\[0.1em]  & 1 & 0 \\[0.1em]  &  & 1 \end{psmallmatrix} \cong \T,\;\; \tran(\ns_1)=\T^2.
\] 
We also check that $\ker(h)=\begin{psmallmatrix} 1 & \Z & \T\\[0.1em]  & 1 & \Z \\[0.1em]  &  & 1 \end{psmallmatrix}\;\cong\; \Z^2 \times \T$. Note that the structure groups of $\ns$ are tori ($\ab_2=\T$, $\ab_1=\T^2$). Thus $\ns$ is also an example of the type of connected \textsc{cfr} nilspace to which we turn next.
\end{example}

\subsection{Toral nilspaces}

\indent Recall that a \textsc{cfr} abelian group $\ab$ has torsion-free dual group if and only if $\ab\cong \T^n$ for some $n\geq 0$.
\begin{defn} We say that a $k$-step nilspace $\ns$ is \emph{torsion-free} if every structure group of $\ns$ has torsion-free dual group. We call $\ns$ a \emph{toral nilspace} if it is \textsc{cfr} and torsion-free.
\end{defn}
\noindent Thus a $k$-step compact nilspace is toral if and only if its structure groups are all tori. One of the reasons for treating these nilspaces in particular is that they occur naturally in the study of the regularizations of ultralimits of functions on abelian groups, which can be used to prove inverse theorems for the Gowers uniformity norms (see the characteristic-0 case of \cite[Definition 1.3]{Szegedy:HFA}).
\begin{lemma}\label{lem:torcon}
Every toral nilspace is connected.
\end{lemma}
\begin{proof}
Recall the general fact that a fibre bundle is connected if the base space and the fibre are connected (see \cite[Part I, \S2.12]{Steenrod}). Thus if $\ab$ has finite rank and is connected then a continuous $\ab$-bundle over a connected space is connected. Now given a $k$-step toral nilspace $\ns$, with factors $\ns_i$, $i\in [k]$, a simple induction on $i$ using this fact shows that $\ns$ is connected. 
\end{proof}
\begin{remark}
The converse of Lemma \ref{lem:torcon} is false. Consider for example the connected 2-step nilspace $\ns$ consisting of the circle group $\T$ with cube sets $\cu^n(G_\bullet)$ for the degree-2 filtration $G_\bullet$ with $G_0=G_1=\T$ and $G_2=\{0,1/2\}\cong \Z/2\Z$ (viewing $\T$ as $[0,1)$ with addition mod 1). A straightforward calculation shows that $x\sim_1 y$ in $\ns$ if and only if $x-y\in G_2$, whence the structure group $\ab_2$ is the disconnected group $G_2$.
\end{remark}

We can now establish the main result of this section.
\begin{theorem}\label{thm:toralnilspace}
Let $\ns$ be a $k$-step toral nilspace, let $G=\tran(\ns)^0$,  let $G_\bullet$ denote the degree-$k$ filtration $(\tran_i(\ns)^0)_{i\geq 0}$, and for an arbitrary fixed $x\in \ns$ let $\Gamma=\stab_G(x)$. Then $\ns$ is isomorphic as a compact nilspace to the nilmanifold $G/\Gamma$ with cube sets $\cu^n(\ns)=(\cu^n(G_\bullet)\cdot \Gamma^{\{0,1\}^n})/\Gamma^{\{0,1\}^n}$, $n\geq 0$.
\end{theorem}
\noindent Note that for the nilmanifold obtained here, every group in the  filtration $G_\bullet$ is connected. (The fact that $G_\bullet$ here is indeed a filtration follows from the fact that if $H_1,H_2$ are closed subgroups of a topological group and one of these subgroups is connected, then $[H_1,H_2]$ is connected.)
\begin{proof} We argue by induction on $k$. For $k=1$, by Lemma  \ref{lem:topkfolderg} we have that $\ns$ is the principal homogeneous space of a torus with cubes being the projections of the standard cubes on the torus, so the statement holds. Supposing then that the statement holds for $k-1$, we fix $x\in \ns$ and let $\Gamma=\stab_G(x)$.

We first claim that $\Gamma$ is discrete. To see this, note that $h(\Gamma)$ is a subgroup of the stabilizer of $\pi_{k-1}(x)$ in $\tran(\ns_{k-1})$, so by induction it is discrete. Then using that $h^{-1}(h(\Gamma))$ is a union of cosets of $\ker(h)$,  we have that it suffices to show that $\Gamma\cap\ker(h)$ is discrete. But this follows from Lemma \ref{lem:h-kernel-Lie}, since no non-trivial element of $\tau(\ab_k)$  stabilizes $x$.\\
\indent By Corollary \ref{cor:actrans} the Lie group $\tran(\ns)^0$ acts transitively on the connected space $\ns$, and it follows that $\ns$ is homeomorphic to the coset space $G/\Gamma$ (see \cite[Chapter II, Theorem 3.2]{Helga}). Therefore $\Gamma$ is cocompact.

It now only remains to determine the cubes on $\ns$. Recall from \cite[Definition 3.2.38]{Cand:Notes1} that two cubes $\q_1,\q_2\in \cu^n(\ns)$ are said to be translation equivalent if there is an element $\q\in \cu^n(G_\bullet)$ such that $\q_2(v)=\q(v)\cdot \q_1(v)$. We claim that for every cube $\q\in \cu^n(\ns)$ there is a cube $\q'\in \cu^n(\ns)$ that is translation equivalent to  the constant $x$ cube and such that $\pi_{k-1}\co \q=\pi_{k-1}\co \q'$. Indeed, given $\q\in \cu^n(\ns)$, we have $\pi_{k-1}\co\q\in \cu^n(\ns_{k-1})$, and by induction the latter cube is translation equivalent to the cube with constant value $x'=\pi_{k-1}(x)$, i.e. it is of the form $\tilde\q\cdot x'$ for some cube $\tilde \q$ on the group $\tran(\ns_{k-1})^0$ with the filtration $\big(\tran_i(\ns_{k-1})^0\big)_{i\geq 0}$. Now by the unique factorization result for these cubes \cite[Lemma 2.2.5]{Cand:Notes1}, we have $\tilde\q={\tilde g_0}^{F_0}\cdots {\tilde g_{2^n-1}}^{F_{2^n-1}}$ where $\tilde g_j\in \tran_{\codim(F_j)}(\ns_{k-1})^0$. Theorem \ref{thm:transurj} (ii) then tells us that for each $j\in [0,2^n)$ there is $g_j\in \tran_{\codim(F_j)}(\ns)^0$ such that $h(g_j)=\tilde g_j$. 
Let $\q^*$ be the cube in $\cu^n(\tran(\ns)^0)$ defined by $\q^*={g_0}^{F_0}\cdots {g_{2^n-1}}^{F_{2^n-1}}$. Let $\q'=\q^*\cdot x$. This is in $\cu^n(\ns)$,  and is translation equivalent to the constant $x$ cube. Moreover, by construction we have
\[
\pi_{k-1}\co \q'= \pi_{k-1}(\q^*\cdot x)= \Big(\prod_j h(g_j)^{F_j}\Big)\cdot x'= \Big(\prod_j \tilde g_j^{F_j}\Big)\cdot x'=\tilde \q \cdot x'=\pi_{k-1}\co \q.
\]
This proves our claim.

It follows from \cite[Theorem 3.2.19]{Cand:Notes1} and the definition of degree-$k$ bundles (in particular \cite[(3.5)]{Cand:Notes1}) that $\q-\q'\in \cu^n(\cD_k(\ab_k))$. But then using translations from $\tau(\ab_k)$ we can correct $\q'$ further to obtain $\q$, thus showing that $\q$ is itself a translation cube with translations from $\tran(\ns)^0$. (Such a correction procedure has been used in previous arguments, see for instance the proof of \cite[Lemma 3.2.25]{Cand:Notes1}.)

We have thus shown that $\cu^n(\ns)\subset (\cu^n(G_\bullet)\cdot \Gamma^{\{0,1\}^n})/\Gamma^{\{0,1\}^n}$. The opposite inclusion is clear, by definition of the groups $\tran_i(\ns)$.
\end{proof}

\begin{remark}\label{rem:simpconn} In Theorem \ref{thm:toralnilspace} we can replace the groups $G_i$ by \emph{simply connected} Lie groups $\tilde G_i$ while conserving the nilspace isomorphism with $\ns$. More precisely, let $\tilde G$ be the universal covering group of the connected Lie group $G$, with covering homomorphism $p:\tilde G\to G$ (see \cite[\S 47, Theorem 61]{Pont}). Let $\tilde\Gamma=p^{-1}(\Gamma)$, and for each $i\geq 0$, let $\tilde G_i$ be the identity component of the closed (hence Lie) subgroup $p^{-1}(G_i)$ of $\tilde G$. Then $\tilde G_\bullet=(\tilde G_i)_{i\geq 0}$ is a degree-$k$ filtration of closed connected (hence simply-connected \cite[Corollary 1.2.2]{C&G}) Lie subgroups of $\tilde G$. It can be checked in a straightforward way that $\tilde G/\tilde \Gamma$ with the cubes determined by $\tilde G_\bullet$ is isomorphic to $\ns$ as a compact nilspace.
\end{remark}

\noindent We know by a theorem of Palais and Stewart that $\ns$ is the total space of a principal torus bundle over a torus if and only if it is a 2-step nilmanifold (see \cite[Theorem 1]{P&S}).  Theorem \ref{thm:toralnilspace} can be viewed as a generalization of this result for arbitrary step, in the category of compact nilspaces (in particular, making use of the assumed cube structure).

One may wonder whether, by adding some simple assumption in Theorem \ref{thm:toralnilspace}, the description of the filtered nilmanifold can then be made even more precise. For instance, having seen Example \ref{ex:heisentoral}, we may seek some simple additional condition under which Theorem \ref{thm:toralnilspace} yields the Heisenberg nilmanifold $H/\Gamma$. In this direction we have the following result.

\begin{proposition}
Let $\ns$ be a 2-step toral nilspace of dimension 3. Then $\ns$ is isomorphic to one of the following compact nilspaces:  \vspace{-0.15cm}
\begin{enumerate}
\item The torus $\R^3/\Z^3$ with cube structure determined by a filtration of degree at most 2 on $\R^3$. \vspace{-0.15cm}
\item The Heisenberg nilmanifold $H/\Gamma$ with cube structure determined by the lower central series on $H$.
\end{enumerate}
\end{proposition}
\noindent Here the dimension is that of $\ns$ when viewed as a manifold (a view which we noted after Lemma \ref{lem:finrankloctriv}).
\begin{proof}
Applying Theorem \ref{thm:toralnilspace} and Remark \ref{rem:simpconn},  we have that $\ns$ is isomorphic as a compact nilspace to $G/\Gamma$ for a simply-connected Lie group $G$ with a filtration $G_\bullet$ of degree at most 2 consisting of closed (therefore Lie) simply-connected groups $G_i\leq G$, and a discrete cocompact subgroup $\Gamma\leq G$. The assumption that $\ns$ has dimension 3 implies that $G$ is 3-dimensional.

If $\ns$ is a 1-step nilspace then $\ns$ is isomorphic to the 3-torus $\R^3/\Z^3$ in case $(i)$, with the standard degree-1 cube structure.

If $\ns$ is not 1-step then, if $G$ is abelian, we have again that $\ns$ is isomorphic to $\R^3/\Z^3$ but now the cube structure is given by a degree-2 filtration on $\R^3$ (for example it could be the filtration determining the degree-2 structure on $\R^3/\Z^3$ defined in \cite[Definition 2.2.30]{Cand:Notes1}). The remaining possibility is that $G$ is non-abelian 2-step nilpotent. The simply-connected Lie group $G$ is uniquely determined by its Lie algebra, which is a real non-abelian, 2-step nilpotent Lie algebra of dimension 3. It follows from the Bianchi classification \cite{Bianchi} that the only such Lie algebra is the 3-dimensional Heisenberg Lie algebra, and so $G$ is the Heisenberg group. Finally, we have that the only degree-2 filtration of closed connected subgroups of $G$ is the lower central series.
\end{proof}
\noindent To close this section, we prove additional pleasant properties of toral nilspaces. In particular, property $(iii)$ below gives a simple relation between the translation groups and the structure groups.

\begin{proposition}\label{prop:abstranstrucrel}
Let $\ns$ be a $k$-step toral nilspace. For each $j\in [k]$ let $h_{j-1,j}$ denote the continuous surjective homomorphism $\tran(\ns_j)^0\to \tran(\ns_{j-1})^0$ defined as in Lemma \ref{lem:hdefn}, and for each $i\in [k]$ let $h_i=h_{i,i+1}\co\cdots\co h_{k-1,k}:\tran(\ns)^0\to \tran(\ns_i)^0$. Fix an arbitrary element $x\in \ns$, and for each $i\in [k]$ let $\Gamma_i=\tran_i(\ns)^0\cap \stab_{\tran(\ns)^0}(x)$. Then for each $i$, the following statements hold:
\begin{enumerate}
\item $\tran_i(\ns)^0$ acts transitively on each fibre of $\pi_{i-1}:\ns\to \ns_{i-1}$.
\item $\ker(h_i)= \Gamma_i\cdot \tran_{i+1}(\ns)^0$.
\item $\tran_i(\ns)^0\;/\; \big(\Gamma_i\cdot \tran_{i+1}(\ns)^0\big)\;\cong\; \ab_i$.
\end{enumerate}
\end{proposition}

\begin{proof}
We prove $(i)$ by downward induction on $i$. For the case $i=k$, since each fibre of $\pi_{k-1}$ is a homogeneous space of $\ab_k$, it suffices to show that $\tran_k(\ns)^0=\tau(\ab_k)$. This is immediate from the following equality, which actually holds for a general nilspace and is established in \cite[Lemma 3.2.37]{Cand:Notes1}:
\begin{equation}\label{eq:abstktk}
\tran_k(\ns)=\tau(\ab_k).
\end{equation}
Suppose now that $i\leq k-1$ and that the statement holds for $i+1$, and note that any fibre of $\pi_{i-1}=\pi_{i-1,i}\co \pi_i$ has the form $\pi_i^{-1}(\pi_{i-1,i}^{-1}(y))$ for some $y\in \ns_{i-1}$. Given $w,w'$ in such a fibre, we have $\pi_i(w),\pi_i(w')\in \pi_{i-1,i}^{-1}(y)$. Since $h_i(\tran_i(\ns)^0)= \tran_i(\ns_i)^0=\tau(\ab_i)$, we have that this acts transitively on $\pi_{i-1,i}^{-1}(y)$ and so there is $\alpha\in \tran_i(\ns)^0$ such that $\pi_i(\alpha(w))=\pi_i(w')$. Then, since $\tran_i(\ns)^0$ contains $\tran_{i+1}(\ns)^0$, by induction it acts transitively on each $\pi_i$-fibre and so there exists $\alpha'\in \tran_i(\ns)^0$ such that $\alpha'\alpha(w)=w'$. This proves statement $(i)$.\\
\indent To see statement $(ii)$, note that $\tran_{i+1}(\ns)^0\leq \ker(h_i)$, since $h_i$ sends $\tran_{i+1}(\ns)^0$ into the trivial  group $\tran_{i+1}(\ns_i)^0$. We also have that $\Gamma_i\subset \ker(h_i)$, because by \eqref{eq:abstktk} we have $h_i(\Gamma_i)\subset \tau(\ab_i)$ but also each element of $h_i(\Gamma_i)$ fixes $\pi_i(x)$, so must be the trivial translation in $\tau(\ab_i)$. It follows that $\Gamma_i\cdot \tran_{i+1}(\ns)^0 \subset \ker(h_i)$. To see the opposite inclusion, note that for $\alpha\in \ker(h_i)$  we have from the definitions that $\pi_i(\alpha(x))=h_i(\alpha)(\pi_i(x))=\pi_i(x)$. Therefore, by statement $(i)$ there exists $\alpha'\in \tran_{i+1}(\ns)^0$ such that $\alpha'\alpha(x)=x$, so $\alpha'\alpha\in \Gamma_i$, and so $\alpha\in \tran_{i+1}(\ns)^0\cdot \Gamma_i$.\\
\indent Statement $(iii)$ now follows from $(ii)$, the first isomorphism theorem for $h_i$, and \eqref{eq:abstktk} for $k=i$.
\end{proof}



\begin{dajauthors}
\begin{authorinfo}[pgom]
  Pablo Candela\\
  Universidad Aut\'onoma de Madrid\\
  Madrid, Spain\\
  pablo\imagedot{}candela\imageat{}uam\imagedot{}es \\
  \url{http://verso.mat.uam.es/~pablo.candela/index.html}
\end{authorinfo}

\end{dajauthors}


\begin{thebibliography}{99}
\bibitem{A&D} A. Adem, J. F. Davis, \emph{Topics in transformation groups},  Handbook of geometric topology, 1--54, North-Holland, Amsterdam, 2002. 

\bibitem{Ander}  R. D. Anderson, \emph{Hilbert space is homeomorphic to the countable infinite product of lines}, Bull. Amer. Math. Soc. \textbf{72} (1966), 515--519. 

\bibitem{Be&Ke} H. Becker, A. S. Kechris, \emph{The descriptive theory of polish group actions}, London Mathematical Society Lecture Note Series, 232. Cambridge University Press, Cambridge, 1996.

\bibitem{Bianchi} L. Bianchi, \emph{On the three-dimensional spaces which admit a continuous group of motions}, Gen. Relativity Gravitation \textbf{33} (2001), no. 12, 2171--2253.

\bibitem{Bill2} P. Billingsley, \emph{Convergence of probability measures}, Second edition. John Wiley \& Sons, Inc., New York, 1999.

\bibitem{Bourb1} N. Bourbaki, \emph{Elements of mathematics. General topology. Part 1}, Hermann, Paris; Addison-Wesley Publishing Co., Reading, Mass.-London-Don Mills, Ont. 1966.

\bibitem{BourInt}  N. Bourbaki, \emph{Integration. I. Chapters 1--6},  translated from the 1959, 1965 and 1967 French originals by Sterling K. Berberian. Elements of Mathematics (Berlin). Springer-Verlag, Berlin, 2004.

\bibitem{Cand:Notes1} P. Candela, \emph{Notes on nilspaces: algebraic aspects}, Discrete Analysis, 2017, Paper No. 15, 59 pp.

\bibitem{CamSzeg} O. A. Camarena, B. Szegedy, \emph{Nilspaces, nilmanifolds and their morphisms}, preprint. \href{http://arxiv.org/abs/1009.3825}{arXiv:1009.3825}.

\bibitem{C&Gra} A. Censor, D. Grandini, \emph{Borel and continuous systems of measures}, Rocky Mountain J. Math. \textbf{44} (2014), no. 4, 1073--1110.

\bibitem{C&G} L. Corwin, F. P. Greenleaf, \emph{Representations of nilpotent Lie groups and their applications, Part 1: Basic theory and examples}, Cambridge studies in advanced mathematics 18, C.U.P. 1990.

\bibitem{D&E} A. Deitmar, S. Echterhoff, \emph{Principles of harmonic analysis}, Universitext. Springer, New York, 2009.

\bibitem{GleasonSmall} A. M. Gleason, \emph{Groups without small subgroups}, 
Ann. of Math. (2) \textbf{56} (1952), 193--212.

\bibitem{Gl} A. M. Gleason, \emph{Spaces with a  compact Lie group of  transformations}, Proc. Amer. Math. Soc. \textbf{1} (1950), 35--43.

\bibitem{GSz} W. T. Gowers, \emph{A new proof of Szemer\'edi's theorem}, GAFA \textbf{11} (2001), 465--588.

\bibitem{GHFA} W. T. Gowers, \emph{Generalizations of Fourier analysis, and how to apply them}, Bull. Amer. Math. Soc. \textbf{54} (2017), no. 1, 1--44.

\bibitem{GTarith} B. Green, T. Tao, \emph{An arithmetic regularity lemma, an associated counting lemma, and applications}, An irregular mind, 261-334, Bolyai Soc. Math. Stud., 21, Janos Bolyai Math. Soc., Budapest, 2010.

\bibitem{GTlin} B. Green, T. Tao, \emph{Linear equations in primes}, Ann. of Math. (2) \textbf{171} (2010), no. 3, 1753--1850.

\bibitem{GTZ} B. Green, T. Tao, T. Ziegler, \emph{An inverse theorem for the Gowers $U^{s+1}[N]$-norm}, Ann. of Math. \textbf{176} (2012), no. 2, 1231--1372.

\bibitem{GMV1} Y. Gutman, F. Manners, P. P. Varj\'u, \emph{The structure theory of nilspaces I}, preprint. \href{http://arxiv.org/abs/1605.08945}{arXiv:1605.08945}.

\bibitem{GMV2} Y. Gutman, F. Manners, P. P. Varj\'u, \emph{The structure theory of nilspaces II: Representation as nilmanifolds}, preprint. \href{http://arxiv.org/abs/1605.08948}{arXiv:1605.08948}.

\bibitem{GMV3} Y. Gutman, F. Manners, P. P. Varj\'u, \emph{The structure theory of nilspaces III: Inverse limit representations and topological dynamics}, preprint. \href{http://arxiv.org/abs/1605.08950}{arXiv:1605.08950}.

\bibitem{Helga} S. Helgason, \emph{Differential geometry, Lie groups, and symmetric spaces}, Pure and Applied Mathematics, 80. Academic Press, Inc., New York-London, 1978.

\bibitem{H&M} K. H. Hofmann, S. A. Morris, \emph{The structure of compact groups. A primer for the student -- a handbook for the expert}, Second revised and augmented edition. de Gruyter Studies in Mathematics, 25. Walter de Gruyter \& Co., Berlin, 2006.

\bibitem{HK} B. Host, B. Kra, \emph{Nonconventional ergodic averages and nilmanifolds}, Ann. of Math. (2) \textbf{161} (2005), no. 1, 397--488.

\bibitem{HKparas} B. Host, B. Kra, \emph{Parallelepipeds, nilpotent groups, and Gowers norms}, Bull. Soc. Math. France \textbf{136} (2008), no. 3, 405--437.

\bibitem{Husem} D. Husemoller, \emph{Fibre bundles}, third edition, Graduate Texts in Mathematics, 20, Springer-Verlag, New York, 1994.

\bibitem{Ke} A. S. Kechris, \emph{Classical descriptive set theory}, Graduate Texts in Mathematics, 156. Springer-Verlag, New York, 1995.

\bibitem{Klep} A. Kleppner, \emph{Measurable homomorphisms of locally compact groups}, Proc. Amer. Math. Soc. \textbf{106} (1989), no. 2, 391--395.

\bibitem{Leib} A. Leibman, \emph{Rational sub-nilmanifolds of a compact nilmanifold}, Ergodic Theory Dynam. Systems \textbf{26} (2006), no. 3, 787--798.

\bibitem{Malcev} A. I. Mal'cev, \emph{On a class of homogeneous spaces}, Amer. Math. Soc. Translation \textbf{1951}, (1951). no. 39, 33 pp.

\bibitem{Munkres} J. R. Munkres, \emph{Topology, Second Edition}, Prentice Hall,  2000.

\bibitem{Pal-G} R. S. Palais, \emph{The classification of $G$-spaces}, Mem. Amer. Math. Soc. No. 36, 1960.

\bibitem{P&S} R. S. Palais, T. E. Stewart, \emph{Torus bundles over a torus}, Proc. Amer. Math. Soc. \textbf{12} (1961), no. 1, 26--29.

\bibitem{Pont} L. Pontrjagin, \emph{Topological groups}, Translated from the Russian by Emma Lehmer, Princeton Mathematical Series, v. 2, Princeton University Press, Princeton, 1939.

\bibitem{Rip} B. D. Ripley, \emph{The disintegration of invariant measures}, Math. Proc. Cambridge Philos. Soc. \textbf{79} (1976), no. 2, 337--341.

\bibitem{Rudin} W. Rudin, \emph{Fourier analysis on groups}, Interscience Tracts in Pure and Applied Mathematics, no. 12, Interscience Publishers, New York-London 1962.

\bibitem{Steenrod} N. Steenrod, \emph{The topology of fibre bundles}, Princeton Landmarks in Mathematics. Princeton University Press, Princeton, NJ, 1999.

\bibitem{Szegedy:HFA} B. Szegedy, \emph{On higher order Fourier analysis}, preprint. \href{http://arxiv.org/abs/1203.2260}{arXiv:1203.2260}.

\bibitem{Tao:Hilbert} T. Tao, \emph{Hilbert's fifth problem and related topics}, Graduate Studies in Mathematics, 153. American Mathematical Society, Providence, RI, 2014.

\bibitem{Var} V. S. Varadarajan, \emph{Lie groups, Lie algebras, and their representations}, Graduate Texts in Mathematics, 102, Springer-Verlag, New York, 1984.

\end{thebibliography}
\end{document}